\documentclass[12pt]{amsart}
\usepackage{amsmath, amsthm, amssymb, amscd}
\usepackage{tikz-cd}

\usepackage[titletoc,toc, title]{appendix}

\usepackage{hyperref}
\usepackage{cleveref}
\usepackage{mabliautoref}
\usepackage{chngcntr}

\usepackage{geometry}
\DeclareMathOperator{\lc}{H}
\DeclareMathOperator{\rank}{rank}
\DeclareMathOperator{\Hom}{Hom}

\DeclareMathOperator{\Tor}{Tor}
\DeclareMathOperator{\Spec}{Spec}

\DeclareMathOperator{\image}{Im}
\DeclareMathOperator{\im}{Im}

\DeclareMathOperator{\vol}{vol}
\DeclareMathOperator{\In}{in}
\DeclareMathOperator{\Gr}{gr}
\DeclareMathOperator{\type}{type}
\DeclareMathOperator{\trdeg}{tr.deg}

\DeclareMathOperator{\Grass}{Grass}
\DeclareMathOperator{\trace}{Tr}

\DeclareMathOperator{\fsig}{s} 
\DeclareMathOperator{\rsig}{s_{rat}} 
\DeclareMathOperator{\csig}{s_{rel}} 
\DeclareMathOperator{\dsig}{s_{dual}} 
\DeclareMathOperator{\Wsig}{s_{Tr}} 
\DeclareMathOperator{\sWsig}{\widetilde{s_{Tr}}} 
\DeclareMathOperator{\eWsig}{s_{Tr}} 

\newcommand{\cotimes}[1]{%
  \mathbin{\widehat{\mathop{\otimes}_{#1}}}%
}

\newcommand{\hght}{\operatorname{ht}}
\newcommand{\length}{\ell}

\newcommand{\cB}{\mathcal B}

\newcommand{\eh}{\operatorname{e}}
\newcommand{\ehk}{\operatorname{e_{HK}}}
 
\newcommand{\mf}{\mathfrak}
\newcommand{\frq}[1]{{#1}^{[p^e]}}
\newcommand{\diml}{\dim_{\ell}\,}
\newcommand{\fdim}[1]{\dim_{#1}\,}

\newcommand{\ul}[1]{\underline{#1}}

\newcommand{\blank}{\ul{\;\;\;}}

\newtheorem*{statement*}{Statement}
\newtheorem*{fact*}{Fact}
\numberwithin{equation}{theorem}

\newtheorem{step}{Step}

\begin{document}
\title{The theory of F-rational signature}

\author{Ilya Smirnov}
\address{BCAM -- Basque Center for Applied Mathematics, Mazarredo 14, 48009 Bilbao, Spain \quad and \quad
IKERBASQUE, Basque Foundation for Science, Plaza Euskadi 5, 48009 Bilbao, Spain}
\thanks{During the preparation of this manuscript, the first author was partially supported through Starting Grant 2020-03970 by the Swedish Research Council, a fellowship from ``la Caixa'' Foundation (ID 100010434), and the European Union’s Horizon 2020 research and innovation programme under the Marie Sk\l{}odowska-Curie grant agreement No 847648 (fellowship code  ``LCF/BQ/PI21/11830033'').}
\email{ismirnov@bcamath.org}

\author{Kevin Tucker}
\address{Department of Mathematics, University of Illinois at Chicago, Chicago, IL 60607, USA}
\thanks{The second author was supported in part by NSF Grants DMS \#2200716 and \#1602070.}
\email{kftucker@uic.edu}

\begin{abstract}
    F-signature is an important numeric invariant of singularities in positive characteristic that can be used to detect strong F-regularity. One would like to have a variant that rather detects F-rationality, and there are two theories that aim to fill 
this gap: F-rational signature of Hochster and Yao and dual F-signature of Sannai. Unfortunately, several important properties of the original F-signature are unknown for these invariants.

We find a modification of the Hochster--Yao definition that agrees with Sannai's dual F-signature and push further the united theory to achieve a \emph{complete} generalization of F-signature.  
\end{abstract}

\maketitle

\section{Introduction}
\label{intro}
Let $(R, \mf m)$ be a commutative Noetherian local domain 
of positive characteristic $p$.
The world of positive characteristic is driven by 
the Frobenius endomorphism $F\colon R \to R$
defined by $r \mapsto r^p$.
A particular way to study this endomorphism is 
via the family of modules $F_*^e R$ obtained from $R$ by iterated
restriction of scalars, so that $r F_*^e x = F_*^e (r^{p^e} x)$.
Under mild assumptions, satisfied in most arithmetic or geometric settings, these modules are finitely generated; we shall assume this holds throughout the introduction. Kunz proved that these modules
detect regularity \cite{Kunz1}: $F_*^e R$ is free for all $e \in \mathbb{N}$ (or equivalently any $e \in \mathbb{N}$) if and only if $R$ is regular.  This result motivates  a number of numerical measures of singularities in positive characteristic, including F-signature and Hilbert--Kunz multiplicity.

The first of such invariants, the Hilbert--Kunz multiplicity,
was defined by Monsky in 1983 (\cite{Monsky}) as an extension of earlier 
work of Kunz (\cite{Kunz2}).
If $\length(\blank)$ denotes the length over $R$ and the dimension of $R$ is $d$, the Hilbert--Kunz multiplicity of an ideal $I$ with $\length(R/I) < \infty$ is defined as
$
\ehk(I) = \lim\limits_{e \to \infty} \frac{1}{ p ^{ed}} \length(R / I^{[p^e]}  )
$
where $I^{[p^e]} = \langle x^{p^e} \mid x \in I \rangle$ is the expansion of  $I$ over the $e$-iterated Frobenius. Similarly, the F-signature was formally defined by Huneke and Leuschke \cite{HunekeLeuschke} building upon the earlier work of Smith and Van den Bergh \cite{SmithVanDenBerghSimplicityOfDiff} on $R$-module direct sum decomposition of $F^e_*R$. In our setting, it is given by
\[
\fsig (R) = \lim_{e \to \infty} 
\frac{\max \{N \mid \text{there is a surjection $F_*^e R \twoheadrightarrow R^N$} \}}
{\rank F_*^e R}.
\]
Both $\fsig(R)$ and $\ehk(\mf m)$ are 
natural measures of singularity, as they encode asymptotically how far the modules  $F_*^e R$ are from being free. 
An alternate perspective on the F-signature,  pioneered in \cite{WatanabeYoshida2, Yao2} and borne out in \cite{PolstraTucker}, links the two invariants together and characterizes the F-signature as the infimum of all relative Hilbert--Kunz differences
\[
\fsig(R) = \inf
\left \{\ehk(I) - \ehk(\langle I, u\rangle) \mid u \notin I, \length (R/I) < \infty \right   \}.
\]

A crucial property of the F-signature is that it detects strong F-regularity, a class of singularities central to the celebrated theory of tight closure pioneered by Hochster and Huneke \cite{HochsterHuneke1}. (Strong) F-regularity can be viewed as the positive characteristic analogue of Kawamata log terminal singularities important to the minimal model program in higher dimensional complex algebraic geometry \cite{Hara, HaraWatanabe, Smith3}. 
Closely related to F-regularity, F-rationality has long been an important class of singularities in positive characteristic commutative algebra. Classically defined by the property that all ideals $\langle \ul{x} \rangle$ generated by a system of parameters $\ul{x} = x_1, \ldots, x_d$ are tightly closed (\cite{FedderWatanabe}), F-rationality can be interpreted geometrically as a positive characteristic analogue of rational singularities over the complex numbers (\cite{MehtaSrinivas, Smith2, Hara}).  

 Recent years have led to rapid advances in our understanding of the F-signature; focusing on those most relevant to our current purpose, 
we highlight the following five core properties of F-signature.
\begin{enumerate}
\item \label{existence} \textit{Existence}: the limit defining $\fsig(R)$ exists \cite{Tucker}.
\item \label{fregularity} \textit{Detects F-regularity}: $\fsig(R) \geq 0$, and $\fsig(R) > 0$ if and only if $R$ is strongly F-regular \cite[Theorem~0.2]{AberbachLeuschke}.
\item  \label{regularity} \textit{Detects regularity}: $\fsig(R) \leq 1$, and $\fsig(R) = 1$ if and only if $R$ is regular \cite{HunekeLeuschke}.
\item  \label{localization} \textit{Compatible with localization}: $\fsig(R) \leq \fsig(R_\mf p)$ for every prime ideal $\mf p$ \cite[Proposition~1.3]{AberbachLeuschke}.
\item \label{semic} \textit{Semicontinuity}: $\mf p \mapsto \fsig(R_{\mf p})$ is lower semicontinuous on $\Spec R$ \cite{Polstra, PolstraTucker}.
\end{enumerate}

Attempts have been made to find an invariant akin to F-signature which detects F-rationality rather than F-regularity.  The first, due to Hochster and Yao \cite{HochsterYao}, builds on the notion that relative Hilbert--Kunz multiplicity can be used to test for tight closure.  The F-rational signature of $R$,  denoted here by $\rsig(R)$, is defined as
\[
\rsig(R) = \inf
\left \{\ehk(\langle \ul{x} \rangle) - \ehk(\langle\ul{x}, u\rangle) \mid u \notin \langle\ul{x} \rangle, \ul{x} \mbox{ a system of parameters} \right \}
\]
where the infimum is taken over all systems of parameters $\ul{x}$ and
elements $u \notin \langle\ul{x} \rangle$.  When $R$ is Gorenstein, it is straightforward to check that $\rsig(R)$ and $\fsig(R)$ coincide (see \cite{HunekeLeuschke}).   Hochster and Yao show that $\rsig(R) > 0$ if and only if $R$ is F-rational, so that the F-rational signature detects F-rationality and satisfies the analogue of property (\ref{fregularity}) above.  Moreover, interpreted appropriately, one can show the F-rational signature satisfies an analogue of existence (\ref{existence}) as well; this property is particularly important in practice  as it allows for estimation and computation. 
However, a computation of Hochster and Yao (\cite[Example~7.4]{HochsterYao}, see also Remark~\ref{no HY}) shows that $\rsig(R) = 1$ does not determine regularity as in (\ref{regularity}). Moreover, to our knowledge, it is unclear (and perhaps unlikely) that the F-rational signature satisfies analogues of properties (\ref{localization}) and (\ref{semic}) above.

Following the introduction of the F-rational signature, an alternate construction was introduced by Sannai \cite{Sannai} mimicking the original definition of F-signature directly.  Called the dual F-signature of $R$ and denoted here $\dsig(R)$, the invariant is defined as
\[
\dsig(R) = \limsup_{e \to \infty} 
\frac{\max \{N \mid \text{ there is a surjection $F_*^e \omega_R \twoheadrightarrow \omega_R^N$}\}}
{\rank F_*^e \omega_R}
\]
where $R$ is assumed Cohen-Macaulay with a canonical module $\omega_R$. Once again, when $R$ is Gorenstein, it is clear that $\dsig(R)$ and $\fsig(R)$ coincide.  Sannai shows further (relying heavily upon \cite{HochsterYao}) that $\dsig(R) > 0$ if and only if $R$ is F-rational. Moreover, $\dsig(R)$ is known to detect regularity and to be compatible with localization as well, satisfying in total the analogues of properties (\ref{fregularity}), (\ref{regularity}),  and (\ref{localization}) above. However, outside of a small number of examples (\emph{cf.} \cite{Nakajima, HashimotoDual}), it has remained open whether the limit defining the dual F-signature exists. Not only is this problematic when attempting to compute or estimate $\dsig(R)$, it is also at the heart of the difficulty in attempting to show that the dual F-signature defines a lower semicontinuous function on $\Spec(R)$.  Thus, in short, we are left to wonder if  the dual F-signature indeed satisfies the analogue of properties (\ref{existence}) and (\ref{semic}). Note also that we will see in Example~\ref{different} that $\rsig(R)$ can be strictly larger than $\csig(R)$.

In this paper, we bring together the two approaches used above, showing that a modified version of the Hochster--Yao invariant defined via relative Hilbert--Kunz multiplicity agrees with Sannai's dual $F$-signature defined via the maximal numbers of surjections. To that end, we introduce herein \emph{the relative F-rational signature} $\csig(R)$ of $R$:
\[
\csig(R) = \inf_{\langle \ul{x} \rangle \subset I} 
\frac{\ehk(\langle \ul{x} \rangle) - \ehk(I)}
{\length (R/\langle \ul{x} \rangle) - \length (R/I)},
\]
where the infimum is taken over all systems of parameters $\ul{x}$ 
and all ideals $I$ properly containing $\langle \ul{x} \rangle$. Our first main results can be summarized as follows.

\begin{theorem*}[Corollary~\ref{cor dual and relative}, Theorem~\ref{thm dual semicontinuous}]
The limit defining the dual F-signature exists and equals the relative F-rational signature. 
Furthermore, the dual F-signature is lower semicontinuous and therefore satisfies 
\emph{all} five core properties (a)-(e) of F-signature listed above.
\end{theorem*}

\noindent
In the course of showing these results, the equivalence of different perspectives on the $F$-signature plays a prominent role. 
The equivalence $\dsig(R) = \csig(R)$ stated in the theorem
itself requires developing an intricate linear algebra result, 
based on a number of nontrivial matrix computations and inductions, 
which we separated out from the main body of the article in the appendix.
The third perspective arises as a certain dual interpretation of the splitting ideals $I_e$, which were originally defined by Aberbach--Enescu (\cite{AberbachEnescu2}) and Yao (\cite{Yao}),
developed in Section~\ref{dualizing} using the Cartier formalism in the sense of Blickle \cite{Blickle}.


To the extent possible, our goal is to present unified theory of F-rational and dual $F$-signature that is fully parallel to the established theory of $F$-signature. 
The newly observed properties of the dual $F$-signature immediately lead to a number of novel perspectives on previously established results. For example, the core properties imply that the set $\{\mf p \mid \dsig(R_\mf p) > 0\}$ is equal to the F-rational locus of $R$ and is open, so we recover a result of V\'elez \cite[Theorem~1.11]{Velez}. However, we are also able to leverage these properties further --
the equality $\dsig(R) =\csig(R)$ in particular opens the door to the use of sophisticated uniform convergence techniques from Hilbert--Kunz theory to establish a number of new and important results.

\begin{theorem*}
The dual F-signature satisfies the following properties.
\begin{itemize}

\item (Corollary~\ref{cor relative deformation}) the dual F-signature deforms, i.e., 
$\dsig(R) \geq \dsig(R/xR)$ for a regular element $x \in R$. 

\item (Corollary~\ref{flat map}, Corollary~\ref{cor equality geometrically regular}) $\dsig(R) \geq \dsig(S)$ 
for a faithfully flat local map $R \to S$, with equality if the closed fiber is geometrically regular. 


\item (Theorem~\ref{t global}) The global dual F-signature in the sense of \cite{DSPY} exists
and is equal to the minimum of localizations $\dsig(R_\mf p)$.

\item (Theorem~\ref{t second coefficient}) $b_e(R)$ admits a second coefficient, i.e., 
there exists a constant $\beta$ such that
\[
b_e(R) = \dsig(R)p^{ed} + \beta p^{e(d-1)} + O(p^{e(d-2)}).
\]
\end{itemize}
\end{theorem*}

\noindent
In particular, while many of these results again parallel the theory of $F$-signature, in some cases we see that the behavior of the dual $F$-signature is even better. Indeed, $F$-signature (as well as strong $F$-regularity) fails to deform without additional assumptions, and moreover the last result on the second coefficient remains an important open question for the $F$-splitting numbers.

Finally, note that while the dual $F$-signature is undefined when $R$ is not assumed to be $F$-finite, the definitions of both the $F$-rational and relative $F$-rational signature are well-posed. While not the primary aim of this article, we additionally explore and verify a number of desirable properties of the relative $F$-rational signature in this setting.

\begin{theorem*}
\begin{enumerate}
\item $\csig(R) \geq 0$, and $\csig(R) > 0$ if and only if $R$ is F-rational
(Corollary~\ref{cor csig independent}.
\item $\csig(R) \leq 1$, and $\csig(R) = 1$ if and only if $R$ is regular
(Proposition~\ref{singularity}).
\item $\csig(R) \leq \csig(R_\mf p)$ for every prime $\mf p$ (Proposition~\ref{relative localizes}).
\item $\csig(R) \leq \csig(S)$ for a faithfully flat local map $R \to S$ (Corollary~\ref{flat map}).
\item $\csig$ deforms (Corollary~\ref{cor relative deformation}). 
\end{enumerate}
\end{theorem*}

\subsection{Structure of the paper}
After setting up definitions in Section~\ref{preliminaries}, we define and study the  relative F-rational signature in Section~\ref{relative}. 
The results of this section are developed without the F-finite assumption. In particular, 
the analogues of properties (\ref{regularity}, \ref{fregularity}, \ref{localization}) of F-signature are shown to hold for the relative F-rational signature (Corollary~\ref{cor relative vs Frational}, Propositions~\ref{singularity},\ref{relative localizes}). We also present an appropriate inequality for flat extensions in Proposition~\ref{flat map}
and for deformation in Corollary~\ref{cor relative deformation}.
The main technical result of this section is Proposition~\ref{go socle}, which allows one to restrict the computation of the relative F-rational signature to socle ideals for a given system of paramaters. In turn, this also allows us to utilize the results of Hochster--Yao and show independence of a given system of parameters as well (Corollary~\ref{cor csig independent}). Combined with the semicontinuity result of \cite{SmirnovAffine}, we also deduce here that the relative F-rational signature is a minimum (rather than infimum) of relative Hilbert--Kunz multiplicities. This gives a new proof independent of \cite[Section~3]{HochsterYao} showing that the positivity of either $\rsig(R)$ or $\csig(R)$ detect F-rationality.

In Section~\ref{dualizing} we give an equivalent definition of the relative $F$-rational signature based on the Cartier operator on the canonical module 
(Definition~\ref{def Cartier}) in the $F$-finite setting. The importance of this perspective comes in part by allowing one to view the relative $F$-rational signature as the minimum of numerical function on the $k$-rational points on a Grassmann variety. This leads to further study of a technical extension of the relative rational signature that takes into account the relative Hilbert--Kunz multiplicities determined by the non $k$-rational points as well, with a view towards properties such as semicontinuity where it is natural to consider behavior at all geometric points.
We will show existence, uniform convergence, and semicontinuity of this generalized invariant. These proofs are somewhat novel and are likely of independent interest. Inspired by \cite{Smirnov}, our method is roughly based on two steps:
we first apply uniform convergence methods introduced in \cite{Tucker} to translate the problem to a more tractable invariant, and second we use the semicontinuity of the rank of a continuous matrix-valued function on a vector bundle. Restricting back to $k$-rational points gives a number of alternative arguments in the non-generalized setting as well. For instance, an alternative proof of the fact that the relative F-signature achieves its minimum via these methods is given (Corollary~\ref{infimum achieved}). 

In Section~\ref{semi} we apply the intricate linear algebra machinery of Appendix~\ref{linear algebra} to show that all of
the different points of view give equal invariants, and in particular 
$\csig(R) = \dsig(R)$. 
This allows one to exploit all of our techniques together at once and establish the fundamental properties of the dual F-signature highlighted above: its existence as a limit (Corollary~\ref{cor dual exists}), its semicontinuity (Theorem~\ref{thm dual semicontinuous}), existence of the second coefficient (Theorem~\ref{t second coefficient}), invariance 
under regular morphisms (Corollary~\ref{cor equality geometrically regular}), 
and also the local-to-global property (Theorem~\ref{t global}). 
We note that our proof of the existence of the global dual F-signature
utilizes semicontinuity more efficiently than \cite{DSPY} and can be used to significantly shorten it. 

Finally, in Section~\ref{Toric} we present an approach for computing the F-rational signature of a toric singularity, and we finish the paper with some of the remaining open questions. The appendix contains the crucial linear algebra machinery. 

\subsection*{Acknowledgements}
We are indebted to 
Pham Hung Quy for finding and correcting a mistake in an earlier version of Proposition~\ref{relative localizes}. 
We thank Mel Hochster and Yongwei Yao for many fruitful discussions.

\section{Preliminaries}\label{preliminaries}

\subsection{F-rational singularities}

In \cite{FedderWatanabe}  Fedder and Watanabe defined a local ring $(R, \mf m)$ to be \emph{F-rational} if every parameter ideal is tightly closed, in parallel to the notion of \emph{weakly F-regular}, due to Hochster and Huneke in \cite{HochsterHuneke1}, that asks every ideal to be tightly closed. 
F-rational rings are normal and an F-rational ring which is an image of Cohen-Macaulay ring is Cohen-Macaulay itself.
The  theory of F-rational rings was further developed in \cite{HochsterHuneke2}.

It was shown in \cite[Theorem 8.17]{HochsterHuneke1} 
that tight closure is determined by Hilbert--Kunz multiplicity under mild assumptions.
Thus we can restate F-rationality using so called \emph{relative Hilbert--Kunz multiplicity}. We provide a proof to illustrate where the assumptions are used.

\begin{proposition}\label{prop HK frat}
Let $(R, \mf m)$ be a local ring. If $\ehk(\langle \ul{x} \rangle) > \ehk(I)$ for every system of parameters $\ul{x}$ and every ideal $\ul{x} \subsetneq I$,
then $R$ is F-rational. 
Moreover, if $R$ is excellent (more generally, it suffices that $R$ and $\widehat{R}$ have a common parameter test element; see also \cite[Proposition~3.2.2]{DattaTucker}) then the converse holds. 
\end{proposition}
\begin{proof}
If $R$ is not F-rational, then there exists a system of parameters which is not tightly closed, say $a \in \langle \ul{x} \rangle^* \setminus \langle \ul{x} \rangle$.
Hence, $\ehk(\langle \ul{x} \rangle) = \ehk(I)$ by \cite[Theorem 8.17]{HochsterHuneke1}, which requires no assumptions. 

If $R$ is complete and F-rational, then, since an F-rational ring is a domain, 
the assumptions of the converse in \cite[Theorem 8.17]{HochsterHuneke1}
are satisfied, thus $\ehk(\langle \ul{x} \rangle) > \ehk(I)$ for all 
$\ul{x} \subsetneq I$. 
Since Hilbert--Kunz multiplicity is not affected by completion, 
we obtained that the relative Hilbert--Kunz condition determines whether $\widehat{R}$
is F-rational. The excellence (or the weaker assumption) is needed to descend F-rationality from $\widehat{R}$ to $R$, see the discussion in \cite[Proof of Theorem~4.1]{HochsterYao}.
\end{proof}

The definitions of F-rational signature \cite{HochsterYao} and normalized F-rational signature (Section~\ref{relative})
are motivated by this equivalence. In \cite{Smith} Karen Smith restated F-rationality using tight closure in local cohomology,
as observed by Hochster--Yao the relative Hilbert--Kunz multiplicity can be also moved to local cohomology, 
see Proposition~\ref{HYprop}. 

Motivated by the notion of \emph{strong F-regularity}, Vel{\' e}z gave definition of 
\emph{strong F-rationality} \cite{Velez}. However, while the equivalence of strong and weak F-regularity 
is a long-standing conjecture, Velez  showed (\cite[Lemma~1.3, Proposition~1.6]{Velez}) that the two versions of F-rationality agree for 
F-finite domains, the assumption that we will impose from Section~\ref{dualizing}. Note that an F-finite ring is excellent \cite{Kunz2}
and is a quotient of a regular ring by \cite{Gabber},
thus it has a dualizing module: a maximal Cohen-Macaulay module 
of finite injective dimension and Cohen-Macaulay type $1$. We refer to the Bruns--Herzog book \cite[3.3]{BrunsHerzog} for properties of dualizing (canonical) modules. 

\begin{theorem}[Vel{\' e}z]
Let $(R, \mf m)$ be a Cohen-Macaulay F-finite reduced local ring and $\omega_R$ be its dualizing module. 
Then $R$ is F-rational if and only if for every $0 \neq c \in R$ 
there is $e \geq 1$ and a homomorphism $\phi\colon F_*^e \omega_R \to \omega_R$ such that
$\phi(F_*^e (c \omega_R)) = \omega_R$. 
\end{theorem}

\subsection{Semicontinuity}
We recall  that function $f \colon X \to \mathbb R$ on a topological space $X$ 
is lower semicontinuous if for any $a \in \mathbb R$ the set $
\{x \in X \mid f(x) > a\}
$
is open.
Semicontinuity is an essential property of a singularity invariant for multiple reasons
and we will present multiple consequences of semicontinuity in this paper. 
Let us start with several fundamental properties. 

\begin{theorem}\label{thm semicontinuity properties}
Let $f$ be a lower semicontinuous function on a topological space $X$. The following properties hold. 
\begin{enumerate}
    \item\label{sem prop 1} If $X$ is quasi-compact, then $f$ has a minimum.
    \item\label{sem prop 2} If $X$ is Noetherian, then a lower semicontinuous function satisfies a descending chain condition, i.e., the set of its values does not contain infinite strictly decreasing sequences.
    \item\label{sem prop 3} In particular, if $X$ is Noetherian, then the minimum of $f$ is separated, i.e., there exists $\varepsilon > 0$ such that for any $y \in X$ if $f(y) - \min_{x \in X} f(x) < \varepsilon$,
    then $f(y) = \min_{x \in X} f(x)$.
    \item\label{sem prop 4} If $X$ is a Noetherian $k$-scheme then $f$ attains a minimum on $\ell$-rational points for any $k \subseteq \ell$. This minimum is also separated. 
\end{enumerate}
\end{theorem}
\begin{proof}
(\ref{sem prop 1}) The ordered family of sets $\{x \in X \mid f(x) > a\}$ forms an open cover of $X$, thus 
the quasi-compactness assumption implies 
that there exists a minimal $a_0$ such that 
$\{x \in X \mid f(x) > a\}\neq X$. This $a_0$ is the minimum of $f$.

(\ref{sem prop 2}) Suppose that $a_1 = f(x_1) > a_2 = f(x_2) > \cdots > a_i = f(x_i) > \cdots$ is an infinite decreasing sequence. Then $X_i := \{x \in X \mid f(x) >  a_i\}$ form an increasing chain of open sets, but any such chain must stabilize
because $X$ is Noetherian.  Hence, $x_{i} \in X_i = X_{i - 1}$ for $i \gg 0$, a contradiction.

(\ref{sem prop 3}) Since there is no infinite decreasing sequences, there is the second smallest value. 

(\ref{sem prop 4}) The set of values on $\ell$-rational points has a minimum because it cannot contain an infinite decreasing sequence.
\end{proof}

See Corollary~\ref{cor csig independent} for an important application of (\ref{sem prop 4}).

\subsubsection{An important example}
A standard example of an upper semicontinuous function (e.g., \cite[Lemma~2.2]{PolstraTucker}) is the minimal number of generators of a finitely generated $R$-module $M$:  
$\mf p \mapsto \dim_{k(\mf p)} (M \otimes_R k(\mf p))$ defines an upper semicontinuous function on $\Spec R$.
From this example we can build more. 
For example, if $A$ is an artinian $k$-algebra of finite length 
and $J$ is an ideal of $A[\ul{X}]$, where $\ul{X}$ is a \emph{set} of variables, then $\mf p \mapsto \length (A[\ul{X}]/J \otimes_{A[\ul{X}]} k(\mf p))$
defines an upper semicontinuous function on $\Spec k[\ul{X}]$.

Later, we will need a semicontinuity result on the Grassmannian. This was observed in \cite[Remark~4.17]{SmirnovAffine}, but its uniform convergence machinery requires $R$ to be a finitely generated $k$-algebra. Instead, we may use that the Grassmannian parametrizes ideals in $R$, so we get uniform convergence in much easier way. 

\begin{theorem}\label{thm: concrete grassmannian}
Let $(R, \mf m)$ be a local ring of characteristic $p > 0$ and $I$ be an $\mf m$-primary ideal.  Let $V = (I :_R \mf m)/I$ be the socle of $I$ and for any subspace $U \subseteq V$ denote $J_U = (I, U)$, the corresponding socle ideal. 
Then the function 
\[
U \mapsto \frac{1}{\dim_k U} \left(\ehk (I) - \ehk (J_U)\right) 
\]
is lower semicontinuous on the $k$-rational points of the Grassmannian of $X$. 
\end{theorem}
\begin{proof}
First, let us fix $e$ and show that the function 
\[
U \mapsto \length \left (R/J_U^{[p^e]}\right ) 
\]
is upper semicontinuous on the $k$-rational points of the Grassmannian of rank $r$ subspaces of $V$. This is a local question, so we may cover the Grassmannian by affine patches. After choosing a basis $e_1, \ldots, e_N$ of $V$, 
a patch is given by setting a fixed maximal minor of the generic $r \times N$-dimensional matrix to be the identity, the remaining $r \times (N - r)$ entries are coordinates. Let us organize these entries in a generic matrix $X$. 
Without loss of generality, the non-vanishing minor is the top one:
\[\begin{bmatrix}
 1 & 0 & \cdots  & 0\\
 0 & 1 & \cdots  & 0\\
 \vdots & \vdots & \vdots & \vdots\\
 0 & 0 & \cdots & 1\\
 X_{1, 1} & X_{1,2} & \cdots  & X_{1,r}\\
 \vdots & \vdots & \vdots &  \vdots\\
 X_{N-r, 1} & X_{N-r,2} & \cdots  & X_{N-r,r}
\end{bmatrix}.
\]

Let $A$ denote the artinian ring $R/(I^{[p^e]})$. We may choose a coefficient field $k$ for $A$. 
The ideal $J_U^{[p^e]}$ is obtained by specializing $X_{i, j}$
in the ideal
\[
A[X] \supset J^{[p^e]} = \left (e_1 + \sum_{i = 1}^{N - r} X_{1, i}^{p^e} e_{r + i}, \ldots, e_r + \sum_{i = 1}^{N-r} X_{r, i}^{p^e} e_{r + i} \right).
\]
Thus the function $U \mapsto \length \left (R/J_U^{[p^e]}\right )$
is upper semicontinuous, note that the restriction to $k$-rational closed points is still semicontinuous since
this set is equipped with the induced topology. 

Second, we recall that \cite[Theorem~3.6]{Tucker} gives uniform convergence: there exists a constant $C > 0$ and a positive integer $e_0$  such that for any ideal $J \supseteq I$ and any positive integer $e$ 
\[
\left |\ehk (J) -  \frac{1}{p^{(e + e_0)d}}\length (R/J^{[p^{e + e_0}]})\right | 
\leq C/p^e.
\]
This implies uniform convergence of 
\[\length  \left (\frac{J_U^{[p^e]}}{I^{[p^e]}} \right ) = 
\length (R/I^{[p^e]}) - 
\length (R/J_U^{[p^e]})
\]
to the relative Hilbert--Kunz multiplicity. 
It follows that 
\[
U \mapsto \frac{1}{\dim_k U} \left(\ehk (I) - \ehk (J_U)\right) 
\]
is lower semicontinuous as the uniform convergent limit of lower semicontinuous functions. 
\end{proof}

\section{Relative F-rational signature}
\label{relative}
\begin{definition}
Let $(R, \mf m)$ be a local ring.  
\begin{enumerate}
    \item The \emph{F-rational signature} of $R$
is defined as 
\[
\rsig(R) = \inf_{\ul{x} \subset I} \left\{ \ehk(\langle \ul{x}\rangle ) - \ehk(I)\right\}
\]
where the infimum is taken over all systems of parameters $\ul {x}$ and ideals $\langle\ul{x}\rangle \subset I$. 
\item
The \emph{relative F-rational signature} of $R$ is 
\[
\csig(R) = \inf_{\ul{x} \subset I} 
\frac{\ehk(\langle\ul{x}\rangle) - \ehk(I)}{\length (R/\langle\ul{x}\rangle) - \length(R/I)}
\]
where the infimum is taken over all systems of parameters $\ul {x}$ and ideals $\langle\ul{x}\rangle \subset I$. 
\end{enumerate}

\end{definition}

F-rational signature was defined by Hochster and Yao in \cite{HochsterYao}.
Clearly, if $\rsig(R) > 0$ or $\csig(R) > 0$, then $R$ is F-rational by Proposition~\ref{prop HK frat}. However, as its name indicates, the converse also holds
under the assumptions of Proposition~\ref{prop HK frat} (\cite[Theorem~4.1]{HochsterYao}).  
In (a) above, it is enough to consider any fixed system of parameters \cite[Theorem~2.4]{HochsterYao}
and it is easy to see that one can restrict to socle ideals $I = \langle \ul{x}, u \rangle$, 
where $\langle\ul{x}\rangle: u = \mf m$ (\textit{cf.} \Cref{cor csig independent} for a similar result for relative F-rational signature). Though the difference in the two definitions would seem small, the additional normalizing factor in the definition of relative F-rational signature is quite useful and leads to a number of desirable properties that are unknown (if not false) for F-rational signature.

\begin{example}\label{different}
It is easy to find examples where $\rsig(R) \neq \csig(R)$ in the toric case, see Section~\ref{Toric}. Explicitly, let $V_n$ be the $n$th Veronese subring of $k[[x,y]]$.
Hochster and Yao (\cite[Example~7.4]{HochsterYao}) 
computed that $\rsig(V) = 1 - \frac 1n$.
On the other hand, one can see that the relative Hilbert--Kunz multiplicity 
for the entire socle is $1/2$, so $\rsig(V_n) > \csig(V_n)$ for $n \geq 3$.

Namely,  if we take a system of parameters $x^n, y^n$ then the whole maximal ideal $x^n, x^{n-1}y, \ldots, y^n$ is the socle. We compute the Hilbert--Kunz multiplicity by passing to $S = k[[x,y]]$:
\[
\ehk (\mf m, V_n) = \frac{\length (S/\langle x,y\rangle^n)}{[S: V_n]} = \frac{\binom{n + 1}{2}}{n} = \frac{n + 1}{2}.
\]
Since
$
\ehk (\langle x^n, y^n\rangle, V_n) = \length (V/\langle x^n, y^n\rangle ) = n,
$
we get that
\[
\frac{\ehk (\langle x^n, y^n\rangle, V_n) - \ehk (\mf m, V_n)}{\length (\mf m/\langle x^n, y^n\rangle )} 
= \frac{1}{2}.
\]

In \cite[Example~3.17]{Sannai} Sannai has computed that the dual F-signature of any Veronese
subring of $k[[x,y]]$ is $1/2$. 
Using Theorem~\ref{thm just comparison}, this will show that, in fact,  
 $\csig(V_n) = 1/2$.
\end{example}

\subsection{Measuring singularities}

As a first step, we record that the relative F-rational signature is normalized so as to detect singularity. The original F-rational signature does not detect singularities (Remark~\ref{no HY}). It is also not known whether it is bounded above by $1$ (which we suspect may be false).

\begin{proposition}
\label{singularity}
If $(R, \mf m)$ is a formally unmixed local ring, then $\csig(R) \leq 1$ with equality if and only if $R$ is regular.
\end{proposition}
\begin{proof}
Suppose that $\csig(R) \geq 1$. It follows from the definition that 
\[
\frac{\ehk(\langle\ul{x}\rangle) - \ehk(\mf m)}{\length (R/\langle\ul{x}\rangle) - \length (R/\mf m)} \geq 1.
\]
for any system of parameters $\ul{x}$.
Since $\length (R/\mf m) = 1$,
we obtain that
\[
1 \geq \ell(R/\langle \ul{x} \rangle ) - \ehk(\langle \ul{x} \rangle) + \ehk(\mf m) \geq \ehk(\mf m).
\]
Because $\ehk (\mf m) \geq 1$ and
$\ehk(\langle\ul{x}\rangle) = \eh(\langle\ul{x}\rangle) \leq 
\length (R/\langle\ul{x}\rangle)$, where the former holds by Lech's lemma \cite[Theorem~2]{Lech} and the latter by \cite[Proposition~11.1.10]{HunekeSwanson},
the above inequality implies that $\ehk (\mf m) = 1$ and $R$ is regular by a result of Watanabe and Yoshida 
\cite[Theorem~1.5]{WatanabeYoshida}. 

The converse follows by noting that $\ehk(I) = \length (R/I)$ for any $\mf m$-primary ideal $I$ of a regular local ring $R$, so $\csig(R) = 1$.
\end{proof}


The same idea can be also used to show that $R$ has mild singularities assuming $\csig(R)$ is sufficiently close to one.

\begin{corollary}\label{mild singularity}
Let $(R, \mf m)$ be a formally unmixed local ring with an infinite residue field.
\begin{enumerate}
\item 
If $\csig(R) \geq 1 - \frac{\max\left\{ \frac 1{d!}, \frac 1{\eh(R)}\right\}}{\eh(R) - 1}$,
then $R$ is weakly F-regular.
\item 
If $\csig(R) \geq 1 - \frac 1{(\eh(R) - 1)^2}$, then
$R$ is Gorenstein and F-regular.
\end{enumerate}
\end{corollary}
\begin{proof}
Take a system of parameters $\ul{x}$ that forms a minimal reduction of $\mf m$, so that $\eh (\langle\ul{x}\rangle) = \eh(R)$.
Suppose that $\csig(R) \geq 1 - \varepsilon$ for some $\varepsilon > 0$. It follows from the definition that
\[
\frac{\ehk(\langle\ul{x}\rangle) - \ehk(\mf m)}{\length (R/\langle\ul{x}\rangle) - \length (R/\mf m)} \geq 1 - \varepsilon.
\]
Following the method of proof in \Cref{singularity}, we obtain
\[
\ehk (R) \leq 1 + \varepsilon (\eh(R) - 1).
\]
The desired result now follows from that of Aberbach and Enescu \cite[Corollaries~3.5, 3.6]{AberbachEnescu1}, which makes use of the expressions for $\varepsilon$
appearing in statements (1) and (2).
\end{proof}

The residue field assumption can be removed once we establish Corollary~\ref{c relative transcendental}.

\subsection{Reduction to socle ideals}

We want to prove that it is enough to take only socle ideals in the definition of $\csig(R)$.
This reduction is at the core of the theory and will allow 
to fix a system of parameters in the definition of $\csig(R)$.

\begin{proposition}\label{go socle}
Let $(R, \mf m)$ be a local ring of characteristic $p > 0$. 
Then for any $\mf m$-primary ideal $I$,
any ideal $J \supsetneq I$,
and any element $x \in \mf m$
there exists an ideal $J'$ such that $I \subsetneq J' \subseteq J$, $xJ' \subseteq I$, 
and
\[
\frac{\ehk (I) - \ehk(J)}{\length (R/I) - \length (R/J)} \geq 
\frac{\ehk (I) - \ehk(J')}{\length (R/I) - \length (R/J')}.
\]
\end{proposition}
\begin{proof}
Since both ideals are $\mf m$-primary, there is 
an integer $m$ such that $x^{m + 1} J \subseteq I$.
We will prove the claim by induction on $m$, with the trivial base case of $m = 0$.

By our assumption the multiplication by $x$
induces the exact sequence
\[
0 \to \frac{I :_J x}{I}
\to \frac{J}{I}
\to \frac{I + xJ}{I}
\to 0.
\]
By applying the containment
$(I :_J x)^{[p^e]} \subseteq 
\frq{I} :_{\frq{J}} x^{[p^e]}
$ to the exact sequence
\[
0 \to 
\frac{\frq{I} :_{\frq{J}} x^{[p^e]}}{\frq{I}}
\to \frac{\frq{J}}{\frq{I}}
\to \frac{\frq{I} + x^{p^e}\frq{J}}{\frq{I}}
\to 0
\]
we get after taking limits that 
\[
\ehk (I) - \ehk(J) \geq
\ehk (I) - \ehk(I :_J x) + 
\ehk (I) - \ehk(I + xJ).
\]
Thus it follows from the inequality 
$\frac{a + c}{b +d }\geq \min \left (\frac ab, \frac cd \right)$ that
\begin{align*}
\frac{\ehk (I) - \ehk(J)}{\length (R/I) - \length (R/J)}
&\geq 
\frac{\ehk (I) - \ehk(I + xJ) + \ehk (I) - \ehk(I :_J x)}{\length (R/I) - \length (R/J)}\\
&\geq \min \left\{ 
\frac{\ehk (I) - \ehk(I + xJ)}{\length (R/I) - \length (R/(I + xJ))}, 
\frac{\ehk (I) - \ehk(I :_J x)}{\length (R/I) - \length(R/I :_J x)}
\right\}.
\end{align*}
Depending on the minimizer, either $J' = I +  xJ$ satisfies the assertion or we apply the induction hypothesis to
$I :_J x$ and find $J \subset I :_J x \subset J$ such that $xJ' \subseteq I$.
\end{proof}

We will now show that one can fix a system of parameters in the definition of $\csig(R)$
using the machinery built by Hochster and Yao. 
First, recall that the Peskine--Szpiro functor of a module $M$ is defined 
as a module such that 
$F_*^e F^e(M) \cong M \otimes_R F_*^e R$ as $F_*^e R$-modules. 
If $L \subseteq H$, we will use $L_H^{[p^e]}$ to denote the image of 
$F^e (L)$ in $F^e (H)$. 
The following proposition combines \cite[Proposition~2.3]{HochsterYao} and the proof of  
\cite[Theorem 2.4]{HochsterYao}.

\begin{proposition}\label{HYprop}
Let $(R, \mf m)$ be a local Cohen-Macaulay ring of prime
characteristic $p > 0$. If we denote $H = \lc^d_\mf m (R)$,
then
\begin{enumerate}
\item for every system of parameters $\ul x$ of $R$ and ideal $I$ such that 
$\langle \ul x\rangle \subset I$ 
there exists a submodule $L$ of $H$ isomorphic to $I/\langle\ul{x}\rangle$,
\item given a finite length submodule $L$ of $H$ there always exists a system of parameters 
$\ul x$ and ideal $I$, $\langle\ul x\rangle \subseteq I$, such that $I/\langle\ul {x}\rangle \cong L$.
\item if $L \subseteq 0:_H \mf m$, the socle of the top local cohomology, 
then such $I$ exists for any system of parameters $\ul{x}$.
\end{enumerate}

Moreover, via this identification we have
$\length (L_H^{[p^e]}) = \length (\frq{I}/\frq{\langle\ul{x}\rangle})$. In particular, 
\[
\lim_{e \to \infty} \frac{\length (L_H^{[p^e]})}{p^{e\dim R}} = \ehk (\langle\ul{x}\rangle) - \ehk(I).
\]
\end{proposition}

This proposition allows to consider F-rational signature as an invariant of the top local cohomology module
\[
\rsig(R) = \inf_{L \subseteq H} \left( \lim_{e \to \infty} \frac{\length (L_H^{[p^e]})}{p^{e\dim R}}\right),
\]
where the infimum is taken over all nonzero finite length submodules of $H$ and immediately shows 
that $\rsig(R)$ is independent
of a system of parameters. 
Using Proposition~\ref{go socle} we will apply the same argument to $\csig(R)$.

\begin{corollary}\label{cor csig independent}
Let $(R, \mf m)$ be a formally equidimensional local ring of characteristic $p > 0$. Then
for any system of parameters $\ul{x}$ we have
\[
\csig(R) = \inf \left \{\frac{\ehk (\langle\ul{x}\rangle) - \ehk(I)}{\length (R/\langle\ul x\rangle ) - \length (R/I)} \mid \langle\ul{x}\rangle \subset I \subseteq \langle \ul{x}\rangle :\mf m \right \}.
\]

Moreover, the infimum in the definition is, in fact, a minimum. Hence, 
if $R$ is excellent, then $\csig (R) > 0$ if and only if $R$ is F-rational.
\end{corollary}
\begin{proof}
We note that the right-hand side does not change under completion, while the left-hand side can only decrease.
Hence we assume that $R$ is complete. If $R$ is not Cohen-Macaulay, then it is enough to show that the 
right-hand side is $0$. But by the colon-capturing \cite[Theorem~7.15(a)]{HochsterHuneke1}
$\langle\ul{x}\rangle$ is not tightly closed, so we can use $I = \langle\ul{x}\rangle^*$ to get $0$ by \cite[Theorem 8.17]{HochsterHuneke1}.

Let $J$ be an arbitrary ideal containing a system of parameters $\ul{x}$.
If $\langle m_1, \ldots, m_k \rangle = \mf m$, then, after applying Proposition~\ref{go socle} $k$ times, we obtain the ideal $I$ such that
$\mf m I = \langle m_1, \ldots, m_k\rangle I \subseteq \langle \ul{x} \rangle$
and 
\[
\frac{\ehk (\langle\ul{x}\rangle) - \ehk(J)}{\length (R/\langle\ul x\rangle) - \length (R/J)} \geq 
\frac{\ehk (\langle\ul{x}\rangle) - \ehk(I)}{\length (R/\langle\ul x\rangle) - \length (R/I)}.
\]
This reduces $\csig(R)$ to socle ideals. However, after fixing $\ul{x}$, Proposition~\ref{HYprop} gives that for $H = \lc^d_\mf m (R)$ we have
\[  
\inf \left \{\frac{\ehk (\langle\ul{x}\rangle) - \ehk(I)}{\length (R/\langle\ul x\rangle ) - \length (R/I)} \mid 
\langle\ul{x}\rangle \subset I \subseteq \langle \ul{x}\rangle :\mf m \right \} 
= \inf \left \{ \frac{1}{\length (L)} \lim_{e \to \infty} \frac{\length (L_{H}^{[p^e]})}{p^{e\dim R}}
\mid 0 \neq L \subseteq 0 :_H \mf m
\right\},
\]
so the left side is independent of $\ul{x}$ and must be equal to $\csig(R)$.

Since $\ul{x}$ can be fixed, the existence of the minimum follows from semicontinuity in Theorem~\ref{thm: concrete grassmannian}
and Theorem~\ref{thm semicontinuity properties}.
We may now use Proposition~\ref{prop HK frat} to characterize F-rationality.
\end{proof}

\begin{remark}
It is natural to expect that the minimum is achieved at the entire socle. However, 
Example~\ref{ex toric} shows that the minimum might be given by a proper subspace.
\end{remark}

\begin{corollary}\label{cor relative deformation}
Let $(R, \mf m)$ be a local ring of characteristic $p > 0$
which is a homomorphic image of a Cohen-Macaulay ring.
If $x$ is a regular element, then $\csig(R) \geq \csig(R/xR)$.
\end{corollary}
\begin{proof}
We may assume that $\csig(R/xR) > 0$, so $R/xR$ is F-rational 
and $R$ must be Cohen-Macaulay.
We complete $x$ to a system of parameters $x, \ul{y}$. 
Because $R$ is Cohen-Macaulay,
\[
\ehk(\langle x,\ul{y}\rangle) = \length (R/\langle x,\ul{y}\rangle) = \ehk(\langle\ul{y}\rangle R/xR).
\]
However, for an arbitrary ideal $I$ containing $x$ we only have an inequality $\ehk(I) \leq \ehk(IR/xR)$
(\cite[Proposition~2.13]{WatanabeYoshida}).
Since $\csig(R)$ (resp. $\csig(R/xR)$) can be computed on a fixed system of parameters $x,\ul{y}$ (resp. $\ul{y}$), 
the inequality now follows.  
\end{proof}

As a corollary we obtain that F-rationality deforms 
(this was proven in \cite[Theorem~4.2(h)]{HochsterHuneke2} without excellence).

\begin{corollary}\label{F-rationality deforms}
Let $(R, \mf m)$ be an excellent local ring of characteristic $p > 0$
which is a homomorphic image of a Cohen-Macaulay ring and $x \in \mf m$ be a regular element.
If $R/xR$ is F-rational, then so is $R$.
\end{corollary}
\begin{proof}
By Corollary~\ref{cor relative deformation} and Corollary~\ref{cor csig independent}
we have
$\csig(R) \geq \csig(R/xR) > 0$, so $R$ is F-rational.
\end{proof}

We can also derive an easy inequality that connects $\csig(R)$ with $\rsig(R)$.

\begin{corollary}\label{cor relative vs Frational}
Let $(R, \mf m)$ be a Cohen-Macaulay local ring of characteristic $p > 0$.
Then
\[
\rsig(R) \geq \csig(R) \geq \frac{\rsig(R)}{\type R},
\]
where $\type(R)$ is the dimension of the socle of any system of parameters. 
\end{corollary}
\begin{proof}
The first inequality is clear from the definition. For the second, we note that 
for any $J$ such that $\mf mJ \subseteq \ul{x}$ 
\[
\frac{\ehk (\langle\ul{x}\rangle) - \ehk(J)}{\length (R/\langle\ul x\rangle) - \length (R/J)}
\geq \frac{\ehk (\langle\ul{x}\rangle) - \ehk(J)}{\length (R/\langle\ul x\rangle) - \length (R/\langle\ul x\rangle:\mf m)}.
\]
A classical result of Northcott (\cite{Northcott}) asserts that 
the denominator is independent of $\ul x$.  The statement now follows after taking the infimums.
\end{proof}

\subsection{Localization and flat extension}
Another benefit of the normalized F-rational signature is that it satisfies the expected localization inequality $\csig (R) \leq \csig(R_\mf p)$
which is not known to hold for the original definition of Hochster and Yao. Namely, it is only known (\cite[Proposition~5.8]{HochsterYao})
that $\rsig (R) \leq \rsig(R_\mf p) \alpha(\mf p)$, where $\alpha (\mf p) \geq 1$
with equality if and only if $R/\mf p$ is regular.

The following version of the proof was suggested to us by Pham Hung Quy. In
F-finite case the proof is much easier, see Theorem~\ref{thm dual semicontinuous}.

\begin{proposition}\label{relative localizes}
Let $(R, \mf m)$ be a local ring of characteristic $p > 0$ and $\mf p$ be a prime ideal.
Then $\csig(R) \leq \csig(R_\mf p)$.
\end{proposition}
\begin{proof}
By induction on $\dim R/\mf p$, we may also assume that $\dim R/\mf p = 1$.

Let $\ul{x}$ be elements in $R$ such that the images of $\ul{x}$
in $R_\mf p$ form a system of parameters. 
Let $\mf p, \mf p_1, \ldots, \mf p_k$ be minimal primes of $\ul{x}$. 
By prime avoidance, we may  choose 
$u \in (\cap_{i = 1}^k \mf p_i) \setminus \mf p$
and $v \in \mf p \setminus \cup_{i = 1}^k \mf p_i$.
For every $n \geq 1$ let $y_n = u^n + v$,
then $\ul{x}, y_n$ is a system of parameters.
By the associative formula for multiplicity we have for every integer $m \geq 1$
 \[
\length(R/\langle \ul{x}, y_n^m\rangle )
\geq \eh(\langle y_n^m \rangle, R/\ul{x})  = nm \eh(\langle u \rangle, R/\mf p)\length_{R_\mf p}(R_\mf p/\langle \ul{x}\rangle) 
+ \sum_{i = 1}^k
m\eh(\langle v \rangle, R/\mf p_i)\length_{R_{\mf p_i}}(R_{\mf p_i}/\langle \ul{x}\rangle).
\]
By \cite[Proposition~4.4]{equi} we also have 
\[
\lim_{m \to \infty} \frac 1m \ehk (\langle \ul{x}, y_n^m \rangle) =
n \eh(\langle u \rangle, R/\mf p)\ehk (\langle \ul{x}\rangle R_\mf p) 
+ \sum_{i = 1}^k
\eh(\langle v \rangle, R/\mf p_i) \ehk (\langle \ul{x}\rangle R_{\mf p_i}).
\]
Let $J$ be an arbitrary ideal in $R_\mf p$ such that $\ul{x} \subset J$. Since $J \cap R$ is $\mf p$-primary we similarly obtain that for
\[
\length (R/\langle J\cap R, y_n^m\rangle)
= \eh(\langle y_n \rangle, R/J\cap R) = mn\eh(\langle u \rangle, R/\mf p)\length_{R_\mf p}(R_\mf p/J)
\]
and
\[
\lim_{m \to \infty} \frac 1m  \ehk(\langle J\cap R, y_n^m\rangle)
 = n\eh(\langle u \rangle, R/\mf p)\length_{R_\mf p}(R_\mf p/J).
\]

For brevity, let us denote 
$C = \sum_{i = 1}^k
\eh(\langle v \rangle, R/\mf p_i)\length_{R_{\mf p_i}} (R_{\mf p_i}/\langle \ul{x}\rangle)$.
It follows from the above discussion that
\[
\csig (R) \leq
\frac{
\ehk(\langle \ul x, y_n^m\rangle) - \ehk(\langle J \cap R, y_n^m \rangle)}{\length(R/\langle \ul{x}, y_n^m\rangle) - \length (R/\langle J\cap R, y_n^m\rangle)} 
\leq 
\frac{
\ehk(\langle \ul x, y_n^m\rangle) - \ehk(\langle J \cap R, y_n^m \rangle)}
{nm\eh(u, R/\mf p)\left (\length (R_\mf p/\langle \ul{x}\rangle) - \length (R_\mf p/J) \right ) + mC}. 
\]
Therefore, after taking the limit as $m \to \infty$ we obtain that for any $n$
\[
\csig (R) \leq
\frac{
n\eh(u, R/\mf p)\left (\ehk(\ul xR_\mf p) - \ehk(JR_\mf p) \right ) + C}{n\eh(u, R/\mf p)\left (\length (R_\mf p/\langle \ul{x}\rangle) - \length (R_\mf p/J) \right ) + C}.
\]
Therefore, after taking the limit as $n \to \infty$,
we must have 
\[
\csig (R) \leq
\inf_{\ul x \subsetneq J}
\frac{
\eh(u, R/\mf p)\left (\ehk(\ul xR_\mf p) - \ehk(JR_\mf p) \right )}{\eh(u, R/\mf p)\left (\length (R_\mf p/\langle \ul{x}\rangle) - \length (R_\mf p/J) \right )}
=
\inf_{\ul x \subsetneq J}
\frac{\ehk(\ul xR_\mf p) - \ehk(JR_\mf p)}{\length_{R_\mf p}(R_\mf p/\ul{x}R_\mf p) - \length_{R_\mf p}(R_\mf p/JR_\mf p)}
= \csig(R_\mf p).
\]
\end{proof}

Next we study the behavior under flat extensions.

\begin{proposition}\label{flat map}
Let $(R, \mf m)$ and $(S, \mf m_S)$ be local  rings of characteristic $p > 0$
with a flat local map $R \to S$. Then $\csig(R) \geq \csig(S)$.
\end{proposition}
\begin{proof}
First, we can take a minimal prime ideal $Q$ of $\mf mS$ and observe that 
$\csig(S) \leq \csig(S_Q)$ by Proposition~\ref{relative localizes}.
Thus we assume that $\mf mS$ is primary to the maximal ideal of $S$, i.e., the two rings have same dimension. 

By flatness, we can tensor a composition series and get that, for an $\mf m$-primary ideal $I$, 
$\length (S/IS) = \length (R/IR) \length (S/\mf mS)$.
Thus
\[
\frac{\ehk (\langle\ul{x}\rangle) - \ehk(I)}{\length (R/\langle\ul x\rangle) - \length (R/I)} = 
\frac{\ehk (\langle\ul{x}\rangle S) - \ehk(IS)}{\length (S/\langle\ul x\rangle S) - \length (S/IS)}.
\]
Thus, because there are more ideals in $S$, we obtain  that
\[
\csig(R) = \inf_{\ul{x} \subset I} \frac{\ehk (\langle\ul{x}\rangle) - \ehk(I)}{\length (R/\langle\ul x\rangle ) - \length (R/I)}
\geq \inf_{\ul{x}S \subset J} \frac{\ehk (\langle\ul{x}\rangle S) - \ehk(J)}{\length (S/\langle\ul x\rangle S) - \length (S/J)}  = \csig(S).
\]
\end{proof}

By using that a system of parameters can be fixed we have the following result. 

\begin{corollary}\label{c relative transcendental}
Let $(R, \mf m)$ be a local ring of characteristic $p > 0$. 
Then $\csig(R) = \csig(R[[t]]) = \csig(R(t))$.
\end{corollary}
\begin{proof}
By Proposition~\ref{flat map}, 
$\csig(R) \geq \csig(R(t))$, so we need to show the opposite inequality. 

Let $S = (R[t])_{(\mf m,t)}$. 
As a first step, we observe that $\csig(R) = \csig(S)$. Namely, we take an arbitrary system of 
parameters $\ul{x}$ of $R$ and observe that the map $I \mapsto (I, t)$ forms a
bijection between socle ideals of 
$(\ul{x})$ and $(\ul{x}, t)$. Since $\length (R/I) = \length (S/(I,t)S)$, we can derive that this extension preserves 
the relative Hilbert--Kunz multiplicities, so the claim follows from Corollary~\ref{cor csig independent}. The same argument holds for $R[[t]]$.
Now, $\csig(S) \leq \csig(S_{\mf mS})$ by Proposition~\ref{relative localizes}
and the assertion follows since $R(t) = S_{\mf mS}$.
\end{proof}

Last, we give the following comparison result between F-signature and relative F-rational signature.

\begin{proposition}\label{p csig vs Fsig}
Let $(R, \mf m)$ be a Cohen-Macaulay local ring, then $\csig(R) \geq \fsig(R)$, 
where the latter is defined for non F-finite rings as in \cite[Definition~2.2]{Yao2}. 

Moreover, if $\csig(R) > 0$ then $\csig(R) = \fsig(R)$ if and only if $R$ is Gorenstein.
\end{proposition}
\begin{proof}
We see that $\csig(R) \geq \fsig(R)$ using \cite[Theorem~1.3(3)]{Yao2}.
If $R$ is Gorenstein, then, by the proof of \cite[Theorem~11]{HunekeLeuschke}, 
$\fsig(R) = \ehk(\langle \ul{x} \rangle) - \ehk(\langle \ul{x}, u\rangle)$
where $u$ generates the socle $(\langle \ul{x} \rangle: \mf m)/\langle \ul{x} \rangle$.
However, Corollary~\ref{cor csig independent} shows that 
$\csig(R) = \ehk(\langle \ul{x} \rangle) - \ehk(\langle \ul{x}, u\rangle)$.

For the converse, \cite[Remark~2.3]{Yao2} allows to extend the residue field without changing the F-signature, so we 
let $S$ be such an extension, an F-finite faithfully flat $R$-algebra such that $S/\mf mS$ is its residue field. 
Thus, by Proposition~\ref{flat map}  $\fsig(R) = \csig(R) \geq \csig(S) \geq \fsig(S)$, so $\csig(S) = \fsig(S)$.
Note that it is now enough to show the statement for $S$ which is F-finite. 
The F-finite case will follow from \cite[Proposition~3.10]{Sannai} 
after we will prove 
Corollary~\ref{cor dual exists}.
\end{proof}

\subsection{Graded rings}
We also want to remark that the F-rational signature of a graded ring can be computed 
using only homogeneous ideals.

\begin{proposition}\label{p go graded}
Let $R$ be a $\mathbb N$-graded ring
over a local ring $(R_0, \mf m)$ and $M = \mf m \oplus \bigoplus_{n > 0} R_{n}$. 
Then $\csig(R_M) = \csig(R)$, where the latter is computed in the graded category.
\end{proposition}
\begin{proof}
We can choose a homogeneous system of parameters $\ul{x}$ of $R$ to compute $\csig(R_M)$. Now, for any finite colength ideal $I \subset R_M$, it is known that $I$ and its initial form ideal 
\[
\In I = (I + R_{>0}) \cap R_0 \oplus  (I + R_{>1}) \cap R_1 \oplus \cdots
\]
have equal colengths.  
Moreover, one can easily see that 
$\frq{(\In I)} \subseteq \In \frq{I}$, 
so $\ehk(\In(I)) \geq \ehk(I)$. 
Thus, by comparing $I$ with its initial ideal, we see that 
\[
\csig(R_M) = 
\inf \left \{\frac{\ehk(\langle\ul{x}\rangle) - \ehk(I)}{
\length( R/\langle\ul{x}\rangle) - \length(R/I)} \mid \ul{x} \subset I \right\}
\]
can be computed using only homogeneous ideals, so it is then equal to $\csig(R)$.
\end{proof}

By the same technique we derive an expected inequality with the associated graded ring.

\begin{proposition}\label{p associated graded}
Let $(R, \mf m)$ be a local ring of characteristic $p > 0$
and let $I$ be an $\mf m$-primary ideal.
Then
$\csig(R) \geq \csig(\Gr_{I} (R))$, where the latter is computed in the graded category.
\end{proposition}
\begin{proof}
We may assume that $\Gr_{I} (R)$ is F-rational, otherwise the 
statement is trivial. Note that $\Gr_{I} (R)$ is a quotient of a polynomial 
ring over $R/I$ which is Cohen-Macaulay, so $\Gr_{I} (R)$ is 
Cohen-Macaulay and so is $R$ by \cite[Theorem~4.11]{HochsterRatliff}.

We may use Corollary~\ref{c relative transcendental} to assume that the residue field is infinite. Then we may choose a regular sequence on $\Gr_{I} (R)$ 
such that its lift to $R$ is a minimal reduction of $I$ 
(e.g., by the proof of \cite[Corollary 8.6.2, Theorem~8.6.3]{HunekeSwanson}). 
The multiplicity of an ideal generated by a regular sequence 
is equal to its colength, thus the multiplicities of
this common system of parameters in $R$ and in the associated graded ring are equal. Hence, by comparing any socle ideal $J$ with its initial ideal as in the previous proof, we see that 
\[
\csig(R) = 
\inf \left \{\frac{\ehk(\langle\ul{x}\rangle) - \ehk(J)}{
\length(R/\langle\ul{x}\rangle) - \length(R/J)} \mid \ul{x} \subset J \right\} 
\geq 
\inf \left \{\frac{\ehk(\langle\ul{x}\rangle) - \ehk(\In I)}{
\length(\Gr_{I} (R)/\langle \ul{x} \rangle) - \length(\Gr_{I} (R)/\In J)} \mid \ul{x} \subset J \right\}
\]
and the latter is clearly greater or equal to $\csig(\Gr_{I} (R))$.
\end{proof}


\begin{corollary}\label{cor Rees}
Let $(R, \mf m)$ be a local ring of characteristic $p > 0$,
let $I$ be an $\mf m $-primary ideal, and let $S = R[It, t^{-1}]$.
Then for any prime ideal $\mf p$ of $S$ 
$\csig(S_\mf p) \geq \csig(\Gr_{I} (R))$, where the latter is computed in the graded category.
\end{corollary}
\begin{proof}
The statement is trivial if $\csig(\Gr_{I} (R)) = 0$, so 
as in the proof of Proposition~\ref{p associated graded}, we 
may assume that $\Gr_{I} (R)$ is Cohen-Macaulay. 
Since $\Gr_{I} (R) = S/t^{-1}S$, for any $\mf p \in \Spec S$ 
that contains $t^{-1}$ the localization $S_{\mf p}$ is Cohen-Macaulay.
Thus, we use the inequalities 
from Corollary~\ref{cor relative deformation} and Proposition~\ref{relative localizes} to get that
\[
\csig(S_\mf p) \geq \csig((S/t^{-1}S)_\mf p)
\geq \csig(\Gr_{I} (R)).
\]
For primes that do not contain $t^{-1}$, we note that 
$S_{t^{-1}} \cong R[t, t^{-1}]$ 
and then use Proposition~\ref{relative localizes}, Corollary~\ref{c relative transcendental}, and Proposition~\ref{p associated graded}.
\end{proof}

\subsection{Finite extensions}

The following result recovers F-rationality of direct summands of 
regular rings. 

\begin{proposition}
Let $(R, \mf m, k)$ be a local domain of characteristic $p > 0$ and let $(S, 
\mf n, \ell)$ be a module-finite domain extension of $R$.
Then\footnote{As stated, the proposition does not hold if $\ell$ is not separable over $k$; see \cite[Example~4.4, Proposition~4.5]{SmirnovTucker2} for a counter-example and the proof of the separable case. We thank Junquan Qin for bringing our attention to this proof.} $[S : R] \csig(R) \geq [\ell : k] \csig(S)$.
\end{proposition}
\begin{proof}
Let $\ul{x}$ be a system of parameters in $R$, then $\ul{x}$ is also a system of parameters in $S$ because $S$ is module-finite.
By the formula for Hilbert--Kunz multiplicity in finite extensions \cite[Theorem~2.7]{WatanabeYoshida}
$[S : R] \ehk(I) = [\ell : k] \ehk(IS)$ for any $\mf m$-primary ideal. 
Because there can be ideals in $S$ which are not extended from an ideal in $R$, it follows that
\[
\csig(R) = \inf_{\ul{x} \subset I} \frac{\ehk (\langle\ul{x}\rangle) - \ehk(I)}{\length (R/\langle\ul x\rangle ) - \length (R/I)} = \inf_{\ul{x} \subset I} \frac{[\ell : k]}{[S:R]} \frac{\ehk (\langle\ul{x}\rangle S) - \ehk(IS)}{\length (S/\langle\ul x\rangle S) - \length (S/IS)}  \geq \frac{[\ell : k]}{[S:R]} \csig(S).
\]

\end{proof}

\section{F-signature theories in the dualizing module}
\label{dualizing}

The goal of this section is to take our considerations to the 
dualizing module, where we can see generalizations of all three perspectives 
on F-signature. In this section we will focus on a generalization of the definition of F-signature via so-called F-splitting or degeneracy ideals. 

In order to relate this definition to the relative F-rational signature, we first need to transfer $\csig(R)$ to the dualizing module.  Since the results of Hochster--Yao already transferred the invariant to the top local cohomology module, it only remains to dualize 
Proposition~\ref{HYprop} to obtain a theory of F-signature of the dualizing module based on the Cartier trace operator. This definition can be further generalized in the framework of Cartier modules introduced by Blickle (\cite{Blickle}), but we will not pursue such generalization in this work.

\begin{definition}
Let $R$ be a ring of positive characteristic $p > 0$.  
A Cartier module $(M, \phi)$ is a finitely generated module $M$ equipped with a $p^{-1}$-linear map $\phi\colon M \to M$.
Equivalently, $\phi$ can be thought of as an $R$-module homomorphism $F_* M \to M$. 
\end{definition}

\begin{remark}\label{trace remark}
The canonical module $\omega_R$ of a Cohen-Macaulay ring
is naturally a Cartier module via the trace map which is constructed 
as follows.
By applying $\Hom_R (\bullet, \omega_R)$ to the Frobenius map
we obtain the trace map
$\trace^e \colon F_*^e \omega_R \cong \Hom (F_*^e R, \omega_R) \to \omega_R $
by the evaluation $\trace^e (\alpha) = \alpha (1)$.
Since $F_*R$ is a maximal Cohen-Macaulay module, 
we can dualize again and obtain that any map from $F_*^e \omega_R \cong \Hom (F_*^e R, \omega_R)$  to $\omega_R$ is a precomposition of trace, $\trace^e (F_*^e r \times \bullet)$.
\end{remark}

The next definition naturally extends the definition of 
F-signature used by Tucker in \cite{Tucker} 
and based on prior work of Yao (\cite{Yao}) and Aberbach--Enescu (\cite{AberbachEnescu2}).

\begin{definition}\label{def Cartier}
Let $(R, \mf m, k)$ be a local F-finite ring of characteristic $p > 0$ with a dualizing module $\omega_R$
and $W \neq 0$ be a quotient of $\omega_R/\mf m\omega_R$. 
For $e \geq 1$ we define the submodule of $F_*^e\omega_R$
\[
Z_e(W) = \bigcap_{r \in R} 
\ker \left [
F_*^e \omega_R\xrightarrow{\trace^e (F_*^e r \times \bullet)} \omega_R \to W
\right ].
\]
The Cartier signature of $W$ is then defined as
\[
\Wsig (W) = \lim_{e \to \infty} \frac{\dim_k ((F_*^e\omega_R)/Z_e(W))}{[k : k^{p^e}]p^{e\dim R} }. 
\]
\end{definition}

It is clear that $\frq{\mf m}\omega_R \subseteq 	Z_e(W)$ for all $W$, so 
$\dim_k ((F_*^e\omega_R)/Z_e(W))$ is finite.

\begin{remark}\label{rmk rank}
If $R$ is a domain, then, in the denominator, $p^{e\dim R} [k : k^{p^e}] = \rank F_*^e R = \rank F_*^e \omega_R$, 
where the first equality holds due to \cite[Proposition~2.3]{Kunz2} and the second because 
$\rank \omega_R = 1$.
\end{remark}

\begin{remark}\label{rmk splitting def}
An alternative interpretation of $Z_e(W)$ is closer to the standard 
definition of the splitting ideals $I_e$:
for $N \subset \omega_R$ consider a submodule 
\[
Z_e(N) = \{x \in F_*^e \omega_R \mid \trace^e_R (F_*^e Rx) \subseteq N\}.
\]
It is easy to see that
if $\length (\omega_R/N) < \infty$ then 
$Z_e(N) = Z_e(\omega_R/N)$.
\end{remark}

\begin{remark}\label{rmk trace}
Because $\trace^e$ generates $\Hom (F_*^e \omega_R, \omega_R)$,
$Z_e(W)$ consists of elements that belong to the kernel of any map 
$F_*^e \omega_R \to \omega_R$.
\end{remark}

We now define a notion of F-signature in $\omega_R$, it 
will be later revisited in Definition~\ref{def extended}. 

\begin{definition}
Let $(R, \mf m, k)$ be a local F-finite ring of characteristic $p > 0$.
Then the (small) Cartier signature of $R$ is 
\[
\sWsig(R) := \inf_{W}
\frac{\Wsig (W)}{\dim_{k} W},
\]
where the infimum is taken over all nonzero quotients $W$ of $\omega_R/\mf m \omega_R$.
\end{definition}

This definition is chosen so that it coincides with the dual of Proposition~\ref{HYprop}. In order to prove this, we first 
recall that an F-finite Cohen-Macaulay ring always has a dualizing module
because Gabber (\cite[Remark~13.6]{Gabber}) showed that an F-finite ring is an image of a regular ring.

\begin{lemma}\label{go omega}
Let $(R, \mf m, k)$ be an F-finite Cohen-Macaulay local ring and let $\omega_R$ be a dualizing module. 
Let $\ul{x}$ be a system of parameters and 
$\ul{x} \subsetneq I \subseteq \langle \ul{x} \rangle : \mf m$. 
In the notation of Proposition~\ref{HYprop} take $W = \Hom_R (L, E)$, where $E$ is the injective hull of the residue field.
Then 
\[
\dim_k \left((F_*^e\omega_R)/Z_e(W) \right) =  [k : k^{p^e}] \length \left(\frq{I}/\frq{\langle\ul{x}\rangle}\right).
\]
Therefore, $\ehk (\langle \ul{x}\rangle ) - \ehk(I) = \Wsig(W)$ and $\csig(R) = \sWsig(R)$.
\end{lemma}
\begin{proof}
Denote $M^\vee := \Hom_R (M, E)$.
First, $L \cong I/\langle \ul {x} \rangle$
is a vector space, so 
$W$ is naturally a quotient of 
$(\lc_{\mf m}^d (R))^\vee \otimes_R k \cong \omega_R \otimes_R k$. Hence $\Wsig(W)$ is defined. 

By Proposition~\ref{HYprop}, we are interested in $\length (L_H^{[p^e]})$.
Since
$\length (L_H^{[p^e]})[k:k^{p^e}] = \length (F_*^e L_H^{[p^e]})$, it will be more convenient 
to work with $F_*^e L_H^{[p^e]}$, i.e., the image of $L \otimes_R F_*^e R \to H \otimes_R F_*^e R$.
Because tensor product is right-exact, we have an exact sequence
\[
0 \to \Hom_R (F_*^e R \otimes_R H/L, E) \to \Hom_R (F_*^e R \otimes_R H, E) \to \Hom_R (F_*^e L_H^{[p^e]}, E) \to 0.
\]
By the Hom-tensor adjunction, we obtain that
\[\Hom_R (H \otimes F_*^e R, E) \cong \Hom_R(F^e_* R, \Hom_R (H, E)) 
\cong \Hom_R (F_*^eR, \omega_R) \cong \omega_{F_*^e R} \cong F_*^e \omega_R. \]
Similarly $\Hom_R (F_*^e R \otimes_R H/L, E) \cong \Hom_R (F_*^e R, (H/L)^\vee)$ 
is a submodule of $F_*^e \omega_R$, so it is enough to show that $Z_e(W) = \Hom_R (F_*^e R, (H/L)^\vee)$.

By Remark~\ref{trace remark} we identify $\trace^e (F_*^e r \times \bullet)$ with the evaluation of $\Hom (F_*^e R, \omega_R)$ at $F_*^e r$. Then $Z_e(W) = \{\phi \in \Hom(F_*^e R, \omega_R) \mid \image \phi \subseteq (H/L)^\vee\}$,
which clearly coincides with $\Hom_R (F_*^e R, (H/L)^\vee)$.
\end{proof}

\subsection{New perspective on existing results}
Our new interpretation gives a more transparent proof of the localization 
property (Proposition~\ref{relative localizes}) in the F-finite case.

\begin{proposition}\label {relative localizes again}
Let $(R, \mf m)$ be an F-finite local ring of characteristic $p > 0$ and $\mf p$ be a prime ideal.
Then $\sWsig(R) \leq \sWsig(R_\mf p)$.
\end{proposition}
\begin{proof}
If $R$ is not  F-rational, then $\sWsig(R) = \csig(R) = 0$ and the claim is trivial.
Thus we assume that $R$ is Cohen-Macaulay and let $\omega_R$ be its dualizing module.
By induction on $\dim R/\mf p$,  we may assume that $\dim R/\mf p = 1$.
Let $x$ be a parameter modulo $\mf p$. 
Since $(\omega_R)_{\mf p} = \omega_{R_\mf p}$,
for any submodule $\mf p \omega_{R_\mf p} \subseteq N \subset \omega_{R_\mf p}$
we may associate a submodule $N' = N \cap \omega_R + y\omega_R$ of finite colength. 
Because $N \cap \omega_R$ is $\mf p$-primary, $x$ is a regular element on $\omega_R/N\cap \omega_R$ and, using that multiplicity is additive, we may compute that
\[
\length (\omega_R/N')
= \eh (x, \omega_R/N \cap \omega_R)
= \eh(x, R/\mf p) \length_{R_\mf p} (\omega_{R_\mf p}/N).
\]

If $L \subset \omega_R$
and $Z_e(L)$ are defined as in Remark~\ref{rmk splitting def},
then one can check that
$Z_e(L) :_{F_*^e \omega_R} x = Z_e(L :_{\omega_R} x)$.
In particular, $Z_e(N \cap \omega_R)$
is still $\mf p$-primary. 
We also note that 
$\trace$ localizes, so 
$(Z_e (N \cap \omega_R))_\mf p = Z_e (N) \subseteq F_*^e \omega_{R_\mf p}$. 
It is straightforward 
to check that 
$x F_*^e \omega_R + Z_e (N \cap \omega_R) \subseteq Z_e(N')$,
thus
\begin{align*}
\length (F_*^e \omega_R/Z_e(N'))
\leq & \length (F_*^e \omega_R/(xF_*^e\omega_R + Z_e (N \cap \omega_R)))
= \eh (x, F_*^e\omega_R/Z_e(N \cap \omega_R)) =
\\ &\eh(x, R/\mf p) \length_{R_\mf p} (F_*^e \omega_{R_\mf p}/Z_e(N)).
\end{align*}
It follows that
$
\Wsig(\omega_R/N') \leq \eh(x, R/\mf p)
\Wsig (\omega_{R_\mf p}/N),  
$
hence by Lemma~\ref{go omega}
\[
\sWsig(R_{\mf p}) = \inf_{\mf p\omega_{R_\mf p} \subseteq N}
\frac{\Wsig (\omega_{R_\mf p}/N)}{\length_{R_\mf p}(\omega_{R_\mf p}/N)}
\geq \inf_{\mf p\omega_{R_\mf p} \subseteq N}
\frac{\Wsig (\omega_{R}/N')}{\length(\omega_{R}/N')}
\geq 
\inf_{\mf p\omega_{R_\mf p} \subseteq N}
\frac{\Wsig (\omega_{R}/N')}{\length(\omega_{R}/N')}
\geq \sWsig(R), 
\]
where the last inequality holds
by Proposition~\ref{HYprop}
and the proof of 
Lemma~\ref{go omega}, because every such $N'$
gives a system of parameters
$\ul{x}$ and an ideal
$I \supset \ul{x}$.
\end{proof}

It is also more convenient to derive a deformation statement in the new language.

\begin{proposition} \label{deformation}
Let $(R, \mf m, k)$ be an F-finite Cohen-Macaulay local ring, $x \in \mf m$ be a parameter, and $\omega_R$ be the canonical module of $R$.
Then
$\sWsig (R) \geq \sWsig(R/xR)$.
\end{proposition}
\begin{proof}
Since $\omega_{R/xR} \cong \omega_R/x\omega_R$, 
we have the following commutative diagram
\begin{equation}\label{equation for deformation}
\begin{tikzcd}
0 \arrow{r}{} &
F_*^e \omega_R \arrow{r}{\times F_*^e x} \arrow[swap]{d}{\trace_{\omega_R}} & F_*^e \omega_R \arrow{d}{\trace_{\omega_R} (F_*^e x^{p^e - 1} \times \bullet)}
\arrow{r}{} & F_*^e \omega_{R/xR}
\arrow{d}{\trace_{\omega_{R/xR}}}\arrow{r}{} & 0
 \\
0 \arrow{r}{} &
\omega_R  \arrow{r}{\times x} & \omega_R
\arrow{r}{} & \omega_{R/xR} \arrow{r}{} & 0
\end{tikzcd}
\end{equation}
that allows us to think about the trace map on
$\omega_{R/xR}$ as a precomposition of the trace on $\omega_R$.
For any quotient $W$ of  $k \otimes_{R/xR} \omega_{R/xR} \cong k \otimes_R \omega_R$
and an element $r \in R$ 
we obtain the induced diagram
\[
\begin{tikzcd}[column sep=large]
F_*^e \omega_R \arrow{d}{\trace_{\omega_R} (F_*^e x^{p^e - 1}r\times \bullet)}
\arrow{r}{\alpha} & F_*^e \omega_{R/xR}
\arrow{d}{\trace_{\omega_{R/xR}}(F_*^e r \times \bullet )}\arrow{r}{} & 0
 \\
\omega_R  \arrow[two heads]{d}{}\arrow{r}{} & \omega_{R/xR} \arrow[two heads]{d}{}\arrow{r}{} & 0\\
W \arrow[equal]{r}{}& W &
\end{tikzcd}
\]
which easily shows that $\alpha(Z_e (\omega_R, W):_{F_*^e \omega_R} F_*^e x^{p^e-1}) \subseteq Z_e (\omega_{R/xR}, W)$.
Hence 
\[
\dim_k \frac{F_*^e \omega_R}{
Z_e(\omega_R, W) :_{F_*^e \omega_R} (F_*^e x^{p^e-1}) 
+ F_*^e x\omega_R }
\geq \dim_k \frac{F_*^e \omega_{R/xR}}{ Z_e (\omega_{R/xR}, W)}.
\]
Since $xF_*^e \omega_R \subseteq Z_e(\omega_R, W)$, we may filter
\begin{align*}
\dim_{k} F_*^e \omega_R/Z_e(W)
&= \sum_{n = 1}^{p^e} \dim_{k} 
\frac{Z_e(W) + F_*^e x^{n-1}\omega_R}{Z_e(W) + F_*^e x^{n}\omega_R}=
\sum_{n = 1}^{p^e} \dim_{k} 
\frac{F_*^e x^{n-1}\omega_R}{Z_e(W) \cap F_*^e x^{n-1}\omega_R + F_*^e x^{n}\omega_R}\\ &= 
\sum_{n = 1}^{p^e} \dim_{k} 
\frac{F_*^e \omega_R}{Z_e(W) :_{F_*^e \omega_R} F_*^e x^{n-1}\omega_R + F_*^e x\omega_R}\\
&\geq p^e\dim_{k} 
\frac{F_*^e \omega_R}{Z_e(W) :_{F_*^e \omega_R} F_*^ex^{p^e-1}\omega_R + F_*^e x\omega_R}
\geq p^e\dim_{k} 
\frac{F_*^e \omega_{R/xR}}{Z_e(W, \omega_{R/xR})}.
\end{align*}
Hence $\Wsig (\omega_R, W) \geq \Wsig(\omega_{R/xR}, W)$ and the assertion follows.
\end{proof}

\subsection{Relative Hilbert--Kunz multiplicity on the Grassmannian}
The Grassmann functor of a coherent sheaf $E$ on a scheme $X$ of rank $n$ associates 
to any $X$-scheme $Y$ the set  of all equivalence classes 
$E \times_X Y \to F$ where $F$ is locally free on $Y$ of rank $n$.
The Grassmannian scheme $\pi_n \colon \Grass (E, n) \to X$ represents the functor as 
$\Hom_X (Y, \Grass(E,n))$. The representing scheme is projective over $S$. 
We refer to \cite{Nitsure} and \cite{FGA} for further background. 

If we consider $\omega_R$ as a coherent shear on $\Spec R$, then 
a point $x \in \Grass (\omega_R, n)$ can be thought of as a pair $\{k(x), W_x\}$ consisting of 
a field extension $k(x)$ of $\pi_n(x)$ and a rank $n$ quotient $W_x$ of $\omega_{R} \otimes k(x)$.
This forces us to extend $\Wsig$ to such quotients. 

\begin{definition}\label{def extended}
Let $(R, \mf m, k)$ be a local F-finite ring of characteristic $p > 0$ with a dualizing module $\omega_R$.
If $\ell$ is a field extension of $k$
and $W \neq 0$ is a quotient of $\omega_R \otimes_R \ell$, for $e \geq 1$ we define the submodule of $F_*^e\omega_R \otimes_R \ell$
\[
Z_e(W) = \bigcap_{r \in R} 
\ker \left [
\ell \otimes_R F_*^e \omega_R\xrightarrow{1 \otimes \trace^e (F_*^e r \times \bullet)} \ell \otimes_R \omega_R \to W
\right ].
\]
The Cartier signature of $W$ is then defined as
\[
\Wsig (W) = \lim_{e \to \infty} \frac{\diml ((\ell \otimes_R F_*^e\omega_R)/Z_e(W))}{[k : k^{p^e}]p^{e\dim R} }. 
\]
\end{definition}

\begin{definition}
Let $(R, \mf m, k)$ be a local F-finite ring of characteristic $p > 0$ with a dualizing module $\omega_R$.
Then the Cartier signature of $R$
is 
\[
\eWsig(R) := \inf \left \{\frac{\Wsig (W)}{\dim_{\ell} W}\mid k \subseteq \ell \text{ is finite and } W \text{ is a nonzero quotient of }\ell \otimes_R \omega_R \right\}.
\]
\end{definition}

Before showing that this definition makes sense, i.e., that the limit in the definition exists, we want to make several useful observations.

The next observation is the key to the semicontinuity. 

\begin{lemma}\label{l finite intersection}
Let $(R, \mf m, k)$ be a local F-finite ring of characteristic $p > 0$ with a dualizing module $\omega_R$,
$\ell$ be a field extension of $k$, and $\pi \colon \omega_R \otimes_R \ell \to W$ be a nonzero surjection of vector spaces over $\ell$.
If $r_1, \ldots, r_m \in R$ are such that  $\{F_*^e r_i\}$ generate $F_*^e R$ as an $R$-module for some $e \geq 1$, then 
\begin{align*}
Z_e(W) &= \bigcap_{i = 1}^m 
\ker \left [
\ell \otimes_R F_*^e \omega_R\xrightarrow{1 \otimes \trace^e (F_*^e r_i \times \bullet)} \ell \otimes_R \omega_R \xrightarrow{\pi} W
\right ]\\
&= \ker 
\left[\ell \otimes_R F_*^e \omega_R \xrightarrow{\sum 1 \otimes \trace (F_*^e r_i \times \bullet)} \bigoplus^m \ell \otimes_R \omega_R \xrightarrow{\oplus  \pi}  \bigoplus^m W \right].
\end{align*}
\end{lemma}

While we defined $\eWsig(R)$ as the infimum over all finite extensions, this was done merely for convenience. 

\begin{lemma}\label{go algebraic closure}
Let $(R, \mf m, k)$ be a local ring. Then 
\[
\eWsig(R) = \inf \left \{\frac{\Wsig(W)}{\fdim{\ell} W} \mid \ell \otimes_R \omega_R \to W \to 0, \text{ $\ell$ is algebraic} \right\}
= \inf \left\{\frac{\Wsig(W)}{\fdim{\overline{k}} W} \mid \overline{k} \otimes_R \omega_R \to W \to 0\right \}.
\] 
\end{lemma}
\begin{proof}
If $\ell$ is algebraic (in particular, finite) over $k$, from 
a given surjection $\pi \colon \ell \otimes_R \omega_R \to W$ 
we may obtain a surjection $1 \otimes \pi \colon \overline{k} \otimes_R \omega_R \to \overline{k} \otimes_{\ell}  W$ by tensoring.
Let $W' =  \overline{k} \otimes_{\ell}  W$. We claim that $\Wsig(W) = \Wsig(W')$. 

Let $F_*^e r_1, \ldots, F_*^e r_m$ generate $F_*^e R$ as an $R$-module. Then by Lemma~\ref{l finite intersection}
\[
\fdim{\ell} 
(\ell \otimes_R F_*^e \omega_R)/Z_e(W) = \rank_{\ell}
\left[\ell \otimes_R F_*^e \omega_R \xrightarrow{\sum 1 \otimes \trace (F_*^e r_i \times \bullet)} \bigoplus^m \ell \otimes_R \omega_R \xrightarrow{\oplus  \pi}  \bigoplus^m W \right].
\]
Because $\otimes_\ell \overline{k}$ is exact, we obtain that
$\fdim{\ell} 
(\ell \otimes_R F_*^e \omega_R)/Z_e(W) = \fdim{\overline{k}} 
(\overline{k} \otimes_R F_*^e \omega_R)/Z_e(W')$ and the claim follows.

The claim easily implies the assertion. 
First, $\eWsig(R) \geq \inf \{\Wsig(W)/(\fdim{\ell} W) \mid \ell \otimes_R \omega_R \to W \to 0, \text{ $\ell$ is algebraic}\}
\geq \inf \{\Wsig(W)/(\fdim{\overline{k}} W) \mid \overline{k} \otimes_R \omega_R \to W \to 0\}$, 
where the first inequality holds since we have more extensions and the second because 
of the claim. It remains to observe that for any surjection $\pi \colon \overline{k} \otimes_R \omega_R \to W$ 
we can find, by taking a basis of $W$, a finite extension $\ell$ of $k$ and a surjection $\sigma\colon \ell \otimes_R \omega_R \to V$ 
such that $W = \overline{k} \otimes_{\ell} V$ and $\pi = 1 \times \sigma$.
Thus, $\eWsig(R) \geq \inf \{\Wsig(W) \mid \overline{k} \otimes_R \omega_R \to W \to 0\}$.

\end{proof}

In fact, it will follow from Corollary~\ref{infimum achieved} that even including arbitrary extensions will not change the invariant.

\subsection{Existence and uniform convergence}
We will now show that our definitions make sense, i.e., the dimensions are finite and the limits exist. 
Furthermore, we will show that the convergence in Definition~\ref{def extended} is uniform.

\begin{lemma}\label{lem length is finite}
Let $(R, \mf m, k)$ be a local F-finite ring of characteristic $p > 0$ with a dualizing module $\omega_R$.
If $\ell$ is a field extension of $k$, then for any nonzero quotient $W$ of 
$\omega_R \otimes_R \ell$
\[
\fdim\ell (\ell \otimes_R F_*^e \omega)/Z_e(W) \leq [k:k^{p^e}]\length_R (\omega_R/\mf m^{[p^e]} \omega_R)
< \infty.
\]
\end{lemma}
\begin{proof}
Observe that $\ell \otimes_R F_*^e \omega_R \cong \ell \otimes_{k} k \otimes_R F_*^e \omega \cong \ell \otimes_{k} F_*^e \omega_R/\frq{\mf m}\omega_R$.
Thus 
\[
\fdim{\ell} (\ell \otimes_R F_*^e \omega /Z_e(W))
\leq \fdim{\ell} \ell \otimes_R \left (F_*^e \omega_R/\mf m^{[p^e]}\omega_R \right) = \dim_k F_*^e \omega_R/\mf m^{[p^e]} \omega_R
\]
and the latter length is finite since $\omega_R$ is finitely generated and $R$ is F-finite.
\end{proof}

\begin{theorem}\label{Cartier exists}
Let $(R, \mf m, k)$ be an F-finite reduced Cohen-Macaulay local ring of dimension $d$ with a dualizing module $\omega_R$.
The limit in the definition of $\Wsig (W)$ exists and the convergence is uniform, i.e., there exists 
a constant $C$ such that for any $e \geq 1$, any extension $\ell$ of $k$, and any nonzero quotient $W$ of 
$\ell \otimes_R \omega_R$
\[
\left |
\frac{\dim_\ell \, (\ell \otimes_R F_*^e\omega_R)/Z_e(W)}{[k : k^{p^e}]p^{ed} }
-\Wsig(W)
\right | < \frac C {p^e}.
\]
\end{theorem}
\begin{proof}
Let $\alpha_p = [k:k^p]p^{d}$. 
Since a Cohen-Macaulay local ring is equidimensional, by \cite[Proposition~2.3]{Kunz2}, the ranks of $F_* \omega_{R_\mf p}$
and $\oplus^{\alpha_p} \omega_{R_\mf p}$ agree at any minimal prime ideal $\mf p$. 
As in the proof of \cite[Lemma~3.3]{Tucker}
this gives us the exact sequence
\[
\bigoplus^{\alpha_p}
\omega_R  \to F_* \omega_R \to T \to 0,
\]
where $\dim T < d$ (i.e., $T = 0$ if $d = 0$). 
Thus the sequence 
$
\bigoplus^{\alpha_p}
F_*^e\omega_R  \xrightarrow{\gamma_e} F_*^{e+1}  \omega_R \to F_*^e T \to 0
$ 
is exact and, by tensoring it with $\ell$, we obtain an exact sequence
\[
\bigoplus^{\alpha_p}
\ell \otimes_R F_*^e\omega_R  \xrightarrow{1_{\ell} \otimes \gamma_e} \ell\otimes_R F_*^{e+1}  \omega_R \xrightarrow{\pi_e} \ell\otimes_R F_*^e T \to 0.
\]

\begin{claim}\label{1st inclusion}
$(1_{\ell} \otimes \gamma_e) (\oplus^{\alpha_p}  Z_e(W)) \subseteq Z_{e+1}(W)$.
\end{claim}
\begin{proof}
By restricting $\gamma_e$ to a summand
and composing with an arbitrary multiple of $\trace^{e+1}$, we obtain the
diagram
\[
F_*^e\omega_R 
\to  F_*^{e+1} \omega_R \xrightarrow{\trace^{e+1}( F_*^{e+1} r \times \bullet)} \omega_R.
\]
Since the resulting map $F_*^e \omega_R \to \omega_R$ is necessarily 
a premultiple of $\trace^e$ by the main property of the trace, the kernel of the induced map
\[
\ell\otimes_R F_*^e\omega_R
\to  \ell\otimes_R F_*^{e+1} \omega_R \xrightarrow{1 \otimes \trace^{e+1}( F_*^{e+1} r \times \bullet) } \ell\otimes_R \omega_R  \to W
\]
contains $Z_e(W)$ by the definition. Since $r$ was arbitrary, we see that $(1_{\ell} \otimes \gamma_e)(Z_e(W)) \subseteq  Z_{e+1}(W)$.
\end{proof}
The claim gives us the exact sequence
\[
\bigoplus^{\alpha_p}
\frac{\ell\otimes_R F_*^e\omega_R}{Z_e(W)} \xrightarrow{1_{\ell} \otimes \gamma_e} \frac{\ell\otimes_R F_*^{e+1} \omega_R}{Z_{e+1}(W)} \xrightarrow{\pi_e} \frac{\ell\otimes_R F_*^e T} {\pi_e(Z_{e+1}(W))} \to 0,
\]
which, as in the proof of Lemma~\ref{lem length is finite}, gives us the bound
\[
\dim_\ell \left(\frac{\ell\otimes_R F_*^{e+1} \omega_R}{ Z_{e+1}(W)} \right )
- \alpha_p
\dim_\ell \left(\frac{\ell\otimes_R F_*^{e} \omega_R}{Z_{e}(W)} \right )
\leq \dim_\ell \left (\ell\otimes_R F_*^e T \right)
\leq 
[k : k^{p^{e}}]\length_R (T/\mf m^{[p^{e}]} T).
\]

For the second step, we consider the analogously obtained exact sequence
\[
\ell \otimes_R F_*^{e+1} \omega_R  \xrightarrow{1_{\ell} \otimes \delta_e}
\bigoplus^{\alpha_p} \ell \otimes_R F_*^e\omega_R \xrightarrow{\rho_e} \ell \otimes_R F_*^e U \to 0.
\]
\begin{claim}
$(1_\ell \otimes \delta^e) (F_*^{e+1} Z_{e+1}(W)) \subseteq \oplus^{\alpha_p} F_*^e Z_e(W)$.
\end{claim}
\begin{proof}
It is enough to show that if we compose 
$(1_\ell \otimes \delta_e)$ with the projection on one of the summands,
then the image of $Z_{e+1}(W)$
is in $Z_e(W)$.
Following the proof of the first claim, this reduces to the fact that 
$F_*^{e+1} \omega_R  \xrightarrow{\delta_e}
\bigoplus^{\alpha_p} F_*^e\omega_R
\to F_*^e \omega_R$ 
is necessarily a premultiple of the trace again. 
\end{proof}

Thus we have the exact sequence
\[
\frac{\ell\otimes_R F_*^{e+1} \omega_R}{Z_{e+1}(W)} \to
\bigoplus^{\alpha_p}
\frac{\ell\otimes_R F_*^e\omega_R}{Z_e(W)} \to \frac{\ell\otimes_R F_*^e U}{\rho(Z_{e}(W))} \to 0,
\]
which by Lemma~\ref{lem length is finite} gives us the bound
\[
\alpha_p
\dim_\ell \left(\frac{\ell\otimes_R F_*^{e} \omega_R }{Z_{e}(W)} \right ) - 
\dim_\ell \left(\frac{\ell\otimes_R F_*^{e+1} \omega_R}{ Z_{e+1}(W)} \right )
\leq 
\dim_\ell (\ell\otimes_R F_*^e U)
= [k : k^{p^{e}}]\length_R (U/\mf m^{[p^e]} U).
\]

After combining and dividing the inequalities by $[k:k^{p^{e+1}}]p^{(e+1)d}$ we get that
\begin{equation}\label{equation bound}
\left |
\frac{\diml (\ell\otimes_R F_*^{e} \omega_R)/Z_{e}(W)}{[k:k^{p^e}]p^{ed}}
 - 
\frac{\diml (\ell\otimes_R F_*^{e+1} \omega_R)/Z_{e+1}(W)}{[k:k^{p^{e+1}}]p^{(e+1)d}}
\right|
\leq 
\frac{\max \{\length_R (T/\mf m^{[p^e]} T), \length_R (U/\mf m^{[p^e]} U)\}}{[k:k^p] p^{(e+1)d}}.    
\end{equation}
By \cite[Lemma~1.1]{Monsky} the right-hand side is bounded above by $D/p^e$ for some constant $D \geq 0$. The theorem then follows from \cite[Lemma~3.5]{PolstraTucker}.
\end{proof}

The proof also shows uniform, independent of a prime ideal, convergence 
on $\Spec R$.

\begin{corollary}\label{Spec uniform}
Let $R$ be an F-finite reduced Cohen-Macaulay ring of dimension $d$.
If $\Spec R$ is connected, then there exists a constant $C$ such that 
for all $\mf p \in \Spec (R)$,
all  field extensions $k(\mf p) \subseteq \ell$, all nonzero quotients $W$ of $\omega_R \otimes_R \ell$, and all $e \geq 1$, we have
\[
\left |
\frac{\dim_\ell \, (\ell \otimes_R F_*^e\omega_R)/Z_e(W)}{[k(\mf p) : k(\mf p)^{p^e}]p^{e\hght \mf p} }
-\Wsig(W)
\right | < \frac C {p^e}.
\]
\end{corollary}
\begin{proof}
By \cite[Corollary~2.7]{Kunz2} $\alpha_p = [k(\mf p):k(\mf p)^p]p^{\hght \mf p}$ is independent of $\mf p$.
Hence, as in the proof of Theorem~\ref{Cartier exists}
we may choose the exact sequences 
\[
\bigoplus^{\alpha_p}
\omega_R  \to F_* \omega_R \to T \to 0 \text{ and }
F_* \omega_R \to 
\bigoplus^{\alpha_p}
\omega_R  \to U \to 0
\]
where $T_\mf q = U_\mf q = 0$ for every minimal prime $\mf q$.
By \cite[Proposition~3.3]{Polstra}
we can find a constant $D > 0$ such that for all $\mf p \in \Spec R$
\[
\frac{\max \{\length_{R_\mf p} (T_\mf p/\mf p^{[p^e]} T_\mf p), \length_{R_\mf p} (U_\mf p/\mf p^{[p^e]} U_\mf p)\}}{[k(\mf p):k(\mf p)^p] p^{(e+1)\hght p}}
< \frac{D}{p^e}
\]
and then use this bound in (\ref{equation bound}) of Theorem~\ref{Cartier exists}.
\end{proof}

For our main result it will be important to interchange the infimum and the limit.

\begin{corollary}\label{cor infimum uniform}
Let $R$ be an F-finite reduced Cohen-Macaulay ring of dimension $d$.
If $\Spec R$ is connected, then there exists a constant $C$ such that 
for all $\mf p \in \Spec (R)$ and all $e \geq 1$, we have for 
$\alpha_e(\mf p) = [k(\mf p) : k(\mf p)^{p^e}]p^{e\hght \mf p}$
\[
\left |\frac{1}{\alpha_e(\mf p)}
\inf \left\{ \frac{\fdim{\ell}  (\ell \otimes_R F_*^e\omega_{R_\mf p})/Z_e(W)}{\dim_{\ell} W} \mid 
[\ell : k(\mf p)] < \infty, W \neq 0, \ell \otimes_R  \omega_R \to W \to 0 \right \}
-\eWsig(R) \right | < \frac C {p^e}.
\]
\end{corollary}
\begin{proof}
A standard application of uniform convergence shows that  in a local ring $R_{\mf p}$
we can interchange the infimum and the limit:
\begin{align*}
\Wsig(R) &= \inf_{\ell \otimes_R  \omega_{R} \to W} \frac{\Wsig(W)}{\dim_{\ell} W} 
= 
\inf_{\ell \otimes_R \omega_{R} \to W} \frac 1 {\dim_{\ell} W}
\lim_{e \to \infty}
\frac{\dim_{\ell} ((\ell \otimes_R F_*^e\omega_\mf p)/Z_e(W))}{[k(\mf p):k(\mf p)^{p^e}]p^{e\hght \mf p}}
\\
&= \lim_{e \to \infty}
\frac{1}{[k(\mf p):k(\mf p)^{p^e}]p^{e\hght \mf p}}
\inf_{\ell \otimes_R \omega_{R} \to W} \frac 1 {\dim_{\ell} W}\dim_{\ell} ((\ell \otimes_R  F_*^e\omega_\mf p )/Z_e(W)).
\end{align*}
But since the appearing constants are independent of $\mf p$ by Corollary~\ref{Spec uniform}, 
we get that this convergence is also uniform in $\mf p$.
\end{proof}

\subsection{Semicontinuity}
We extended $\Wsig$ to a function on  $\Grass (\omega_R, n)$
and will now show its semicontinuity which will imply several other good properties. 

\begin{theorem}\label{t semi on Grassmannian}
Let $R$ be an F-finite reduced Cohen-Macaulay ring of characteristic $p > 0$ with a connected spectrum and $\omega_R$ be a dualizing module.
Let $\pi_n \colon \cB_n \to \Spec R$ be the rank $n$ Grassmannian of the coherent sheaf $\omega_R$.
Then $\Wsig \colon \cB_n \to \mathbb R$ is a lower semicontinuous function.
\end{theorem}
\begin{proof}
Let $\mathcal Q$ be the universal quotient bundle of $\cB_n$.
Let $r_1, \ldots, r_\mu \in R$ be such that they generate $F_*^e R$ as an $R$ module.
Then we may define $\phi_i \colon F_*^e \omega_R \to \omega_R$ by $x \mapsto \trace^e (F_*^e r_ix)$ and consider 
\[
g_n \colon \pi_n^* F_*^e \omega_R \xrightarrow{\sum \pi_n^* \phi_i} \bigoplus_{i = 1}^\mu \pi_n^* \omega_R \to 
\bigoplus_{i = 1}^\mu \mathcal Q, 
\]
where the last map is given by the construction of $\mathcal Q$. 
The rank of the image of the composition is a lower semicontinuous function (e.g., because non-vanishing of a minor is an open condition). If $x \in \cB_n$ is a point such that $\pi_n (x) = \mf p$, then $k(x)$ is a field extension of $k(\mf p)$ and $x$
represents a rank $n$ quotient $W_x$
of $\omega_R \otimes_R k(x)$. 
Thus at $x$ we have the map
\[
g_{n,e}(x) \colon
F_*^e \omega_R \otimes_R k(x) 
\xrightarrow{\sum \phi_i \otimes 1}
\bigoplus_{i = 1}^\mu \omega_R \otimes_R k(x)
\to \bigoplus_{i = 1}^\mu W_x,
\]
which coincides with Lemma~\ref{l finite intersection}. Note that $\Wsig(x)$ is defined as $\Wsig(W_x)$.

Furthermore, let $\alpha_e(x) = [\pi_n(x):\pi_n(x)^{p^e}] p^{e\hght \pi_n(x)}$ 
and note that 
$\rank g_{n,e}(x)/\alpha_e(x)$ 
then coincides with the sequence used in the definition of $\Wsig(W_x)$. 
Hence, Theorem~\ref{Spec uniform}
establishes its uniform convergence independent of $x$. 
Because $\Spec R$ is connected, by \cite[Corollary~2.7]{Kunz2} $\alpha_e(x)$ is a constant, it does not depend on $x$.
Thus $\rank g_{n,e}(x)/\alpha_e(x)$
is a lower semicontinuous function. 
Therefore, 
\[
\Wsig(x) = \lim_{e \to \infty} \frac{\rank g_{n,e}(x)}{\alpha_e(x)}.
\] 
is lower semicontinuous because it is the uniform limit of lower semicontinuous functions. 
\end{proof}

\begin{corollary}\label{infimum achieved}
Let $(R, \mf m, k)$ be an F-finite Cohen-Macaulay local ring and $\omega_R$ be a dualizing module.
Then the infimum in the definition of $\eWsig(R)$ is achieved. 
\end{corollary}
\begin{proof}
If $R$ is not F-rational, then $\Wsig(W) = 0$ by the tight closure characterization and Lemma~\ref{go omega}. 
Hence, we may assume that $R$ is a domain. By Theorem~\ref{t semi on Grassmannian},  $\Wsig$ is lower semicontinuous on $\cB_n$ for each $n$. Thus, $\Wsig$ has a minimum on $\cB_n$ and this minimum is achieved at a closed 
point $x$. Because $\pi_n$ is projective, $\pi_n(x) = \mf m$. Furthermore, it follows from Nullstallensatz 
that $k(x)$ is a finite extension of $k$. Therefore, 
\[
\eWsig(R) = \min \left \{ \frac 1 n \min_{x \in \cB_n} \Wsig(x) \mid 1 \leq n \leq \dim_k \omega_R/\mf m\omega_R \right \}.
\]
\end{proof}

\begin{remark}\label{rem gaps on the socle}
Lemma~\ref{go omega} allows to view $\csig(R)$ as the infimum of the generalized F-signature function $\eWsig$ on $k$-rational points
of the Grassmannian. Hence, by (\ref{sem prop 4}) of Theorem~\ref{thm semicontinuity properties} we obtain a different proof of  Corollary~\ref{cor csig independent} in the F-finite case.  
It should be noted that for non $k$-rational points the two functions are different:
we will show in the next section 
(Theorem~\ref{thm just comparison} and Corollary~\ref{cor dual exists})
that $\eWsig(R) = \sWsig(R)$, i.e., the minimum on $k$-rational points 
is equal to the global minimum, while Corollary~\ref{cor csig independent}
considers relative Hilbert--Kunz multiplicity after the field extension. 

Semicontinuity also implies that the minimum is separated, i.e, there is the second smallest value.  This result also holds for non F-finite rings by using 
\cite{SmirnovAffine} as in Corollary~\ref{cor csig independent}.
\end{remark}

\begin{corollary}\label{Cartier semicontinuous}
Let $R$ be an F-finite reduced Cohen-Macaulay ring such that $\Spec R$ is connected.
Then 
\[
\mf p \mapsto \eWsig(R_\mf p) := 
\inf \left\{ \frac{\Wsig(W)}{\dim_\ell W} \mid k(\mf p) \subseteq \ell \text{ is finite and }
W \neq 0 \text{ is a quotient of } \omega_{R_\mf p} \otimes_{R_\mf p} \ell \right \}
\]
is a lower semicontinuous function.
\end{corollary}
\begin{proof}
We need to show that $\{\mf p \mid \eWsig(R_\mf p) \leq a\}$ is closed for all $a \in \mathbb{R}$.
Let $\pi_n \colon \cB_n \to \Spec R$ be the rank $n$ Grassmannian. 
Because $\pi_n$ is projective and $\Wsig$ is lower semicontinuous, the set 
$Z_n (\leq a) := \pi_n (\{x \in \cB_n \mid \Wsig(x) \leq a\})$ is closed for all $a \in \mathbb R$. 
Clearly, we have $Z_n (\leq a) = \left \{\mf p \in \Spec R \mid \inf \left \{ \Wsig(x)  \mid \mf p = \pi_n (x) \right\} \leq a   \right \}$.
By the definition of Grassmannian, the points $x \in \pi_n^{-1}(\mf p)$ parametrize all possible extensions of $k(\mf p)$ and 
all possible quotients of $\omega_R \otimes_R k(x)$ of rank $n$. Furthermore, from the proof of Corollary~\ref{infimum achieved}, we know that this infimum is achieved at $x$ such that $k(x)$ is finite over $k(\mf p)$. 

Thus, if we let $N$ be such that $\omega_R$ can be generated by $N$ elements as an $R$-module, then  
$\cup_{1 \leq n \leq N} Z_n (\leq na)$ is closed and coincides with $\{\mf p \mid \eWsig(R_\mf p) \leq a\}$
due to the equality
\begin{align*}
\bigcup_{1 \leq n \leq N} Z_n (\leq na) &= \left \{\mf p \in \Spec R \mid \frac{\Wsig(W)}{\diml(W)} \leq a  \text{ for some } [\ell : k(\mf p)] < \infty \text{ and } \omega_R \otimes_R \ell \to W \to 0\right\}.
\end{align*}
\end{proof}

\section{New properties of the dual F-signature}
\label{semi}

In this section we proceed to study the dual F-signature.  
The powerful linear algebra machinery of the appendix 
will show that all three perspectives on F-rational signature are equivalent.
By combining the available techniques, this will allow to greatly advance the theory of dual F-signature, in particular, due to 
the powerful uniform convergence techniques of Hilbert--Kunz theory available for the relative F-rational signature. 

Let us start by recalling the definition given by Sannai in \cite{Sannai}.

\begin{definition}\label{Sannai dual}
Let $(R, \mf m, k)$ be an F-finite Cohen-Macaulay local ring.
Let $\omega_R$ be the dualizing module of $R$. 
For any $e$ let $b_e(R)$ be the largest integer $N$ 
such that there exists a surjection
\[
F_*^e\omega_R \to \bigoplus^{N} \omega_R\to 0.
\]
Then the dual F-signature of $R$ is defined as
\[
\dsig(R) = \limsup_{e \to \infty}\frac{b_e(R)}{p^{e\dim R} [k : k^{p^e}]}.
\]
\end{definition}

\begin{remark}\label{rmk must be reduced}
In \cite{Sannai} the dual F-signature of $R$ was defined under the assumption that $R$ is reduced. This restriction is not essential, because $R$ must be reduced if $\dsig(R) > 0$. 

Namely, suppose there is a surjection 
$F_*^e \omega_R \to \omega_R \to 0$. If $a$ is a nilpotent element such that $a^{p^e} = 0$, 
then $a F_*^e \omega_R = 0$. It follows that $a\omega_R = 0$ which is 
a contradiction with faithfulness of $\omega_R$ (\cite[(1.8)]{Aoyama}).
\end{remark}

\begin{remark}\label{embed}
Sannai observed in \cite[Lemma~3.6]{Sannai} that there is
a useful one-to-one correspondence, arising from duality, between surjections 
$
F_*^e \omega_R \to \oplus^{b_e} \omega_R\to 0
$
and injections
\[
0 \to \oplus^{b_e} R \to R^{1/p^e} \to M \to 0
\]
where $M$ is maximal Cohen-Macaulay.
In particular, this shows that $\dsig(R) \geq \fsig(R)$.
\end{remark}

We now easily get inequalities connecting the theories of F-rational signature. 

\begin{theorem}\label{thm just comparison}
Let $(R, \mf m)$ be an F-finite Cohen-Macaulay local ring.
Then
\[
\rsig(R) \geq \csig(R) = \sWsig(R) \geq \eWsig(R) \geq \dsig(R) \geq \fsig(R).
\]
\end{theorem}
\begin{proof}
The last inequality was established in \cite[Proposition~3.8]{Sannai}.
The inequalities $\rsig(R) \geq \csig(R)$ and $\sWsig(R) \geq \eWsig(R)$ follow from the definitions. It was proved Lemma~\ref{go omega} in that $\csig(R) = \sWsig(R)$, so it remains to show the last inequality.
 
Let $\omega_R$ be the dualizing module. By tensoring the definition of $b_e(R)$, for any field extension $\ell$ and any quotient $W$ of $\ell \otimes_R \omega_R$ there is a surjection
\[
\ell \otimes_R F_*^e \omega_R
\to \bigoplus^{b_e(R)} \ell \otimes_R \omega_R
\to \bigoplus^{b_e(R)} W.
\]
Since the original $b_e(R)$ surjective
maps were necessarily multiples of $\trace^e$ by Remark~\ref{trace remark},
the map induces a surjection
$
(\ell \otimes_R F_*^e \omega_R)/Z_e(W)
\to
\bigoplus^{b_e(R)} W \to 0
$
and the inequality $\eWsig(R) \geq \dsig(R)$ follows. 
\end{proof}

We refine the theorem in the following uniform relation needed both for showing the existence and semicontinuity of the dual F-signature.

\begin{theorem}\label{thm global dual and Cartier}
Let $R$ be an F-finite ring and $\omega_R$ be its dualizing module.
There is a constant $C$ such that for all $\mf p \in \Spec R$ and for all $e \geq 1$ we have
\[
b_e(R_\mf p) + C\geq \min \left \{\frac{\dim_{k(\mf p)} (F_*^e \omega_{R_\mf p}/Z_e(W))}{\dim_{k(\mf p)} W} \mid W \neq 0, \omega_R \otimes_R k(\mf p) \to W \to 0  \right\} \geq b_e(R_\mf p).
\]
\end{theorem}
\begin{proof}
The second inequality was observed in the proof of Theorem~\ref{thm just comparison},
so it remains to show the first inequality. 
Let us denote 
\[
N_e(\mf p) = \min \left \{\frac{\dim_{k(\mf p)} (F_*^e \omega_{R_\mf p}/Z_e(W))}{\dim_{k(\mf p)} W} \mid W \neq 0, \omega_R \otimes_R k(\mf p) \to W \to 0  \right\}.
\]

As first step, we assume that $R$ is a local ring with the maximal ideal $\mf m$ and the residue field $k$.
Let $X = F_*^e \omega_R/\mf m^{[p^e]}\omega_R$ and  $Y = \omega_R/\mf m\omega_R$.
Note that any map $F_*^e \omega_R \to \omega_R/\mf m\omega_R$ factors through $F_*^e \omega_R/\mf m^{[p^e]}\omega_R$, so we let
$H \subseteq \Hom (X, Y)$ to consist of homomorphisms induced by $\Hom_R (F_*^e \omega_R, \omega_R)$. 
By Corollary~\ref{cor:linalgsurjections}, by taking 
$C = P(\dim_k Y)$ for the polynomial $P(T) = T^2(T^2 - 1)/6$
we can build a surjection 
\[
F_*^e \omega_R/\mf m^{[p^e]}\omega_R \to \bigoplus^{N_e(\mf m) - C} \omega_R/\mf m\omega_R \to 0
\]
which descended from $R$-module maps. By Nakayama's lemma, it 
can be lifted to a surjection $F_*^e \omega_{R} \to \oplus^{N_e(\mf m) - C} \omega_{R}$. 
Thus $b_e(R) \geq N_e(\mf m) - C$. 

Second, in all other cases, we let $\nu$ be any integer such that there are $\nu$ elements that generate $\omega_R$. Since 
$\dim_{k(\mf p)} \omega_{R_\mf p}/\mf p\omega_{R_\mf p} \leq \nu$ for all $\mf p$,
and $P(T)$ is monotone by Corollary~\ref{cor:linalgsurjections}, the theorem follows from the first case with $C = P(\nu)$.
\end{proof}

Combining the theorem with  Lemma~\ref{go omega} we obtain a connection with relative Hilbert--Kunz multiplicities.

\begin{corollary}\label{cor dual and relative}
Let $R$ be an F-finite ring. 
There is a constant $C$ such that for any $\mf p \in \Spec R$ and any system of parameters $\ul{x_\mf p}$ of $R_\mf p$
we have
\[
b_e(\mf p) + C\geq [k(\mf p):k(\mf p)^{p^e}]
\min
\left \{\frac{\length_{R_\mf p} \left(\frq{I}/\frq{\langle\ul{x_\mf p}\rangle}\right)}{\length_{R_\mf p} (I/\langle\ul{x_\mf p}\rangle)} 
\mid  \langle\ul{x_\mf p}\rangle \subset I \subseteq \langle\ul{x_\mf p}\rangle :_{R_\mf p} \mf p \right\} \geq b_e(\mf p).
\]
\end{corollary}

\begin{remark}
From the optimal criterion for two-dimensional vector spaces in Theorem\ref{thm:dimtwocriterion}, by appropriately modifying 
Corollary~\ref{cor:linalgsurjections} and Theorem~\ref{thm global dual and Cartier}, we obtain  
the exact equality
\[
b_e(\mf p) = [k(\mf p):k(\mf p)^{p^e}]
\min
\left \{\frac{\length_{R_\mf p} \left(\frq{I}/\frq{\langle\ul{x_\mf p}\rangle}\right)}{\length_{R_\mf p} (I/\langle\ul{x_\mf p}\rangle)} 
\mid  \langle\ul{x_\mf p}\rangle \subset I \subseteq \langle\ul{x_\mf p}\rangle :_{R_\mf p} \mf p \right\}
\]
whenever $\type R(\mf p) = 2$. 
\end{remark}

We will combine these results with the following uniform convergence result that easily follows from  \cite[Theorem~3.6]{Polstra}.

\begin{theorem}\label{Polstra}
Let $R$ be an F-finite ring. There exists a constant $D$ such that for any $\mf p \in \Spec R$ and
any $\mf p$-primary ideal $I$ we have 
\[
\left |
\frac{1}{p^{e\hght \mf p} } 
\min \left \{ 
\frac{\length_{R_\mf p} \left (\frq{J}R_\mf p/\frq{I}R_\mf p\right)}{\length_{R_\mf p} \left (JR_\mf p/IR_\mf p\right)} 
\mid I \subset J \subseteq \mf p 
\right \}
- \inf \left \{
\frac{\ehk (IR_\mf p) - \ehk (JR_\mf p)}{\length_{R_\mf p} \left (JR_\mf p/IR_\mf p\right)} 
\mid I \subset J \subseteq \mf p 
\right\}
\right| 
\leq \frac{D}{p^{e}}.
\]
\end{theorem}
\begin{proof}
By setting $q_2 \to \infty$ in \cite[Theorem~3.6]{Polstra}, we obtain a constant $D > 0$ 
such that for any $\mf p \in \Spec R$ and 
any pair of $\mf p$-primary ideals $I \subseteq J$ 
\[
\left |
\frac{1}{p^{e\hght \mf p}}\length \left(\frac{\frq{J}R_\mf p}{\frq{I}R_\mf p} \right) - \left(\ehk (IR_\mf p) - \ehk (JR_\mf p)\right)
\right| \leq \frac{D}{p^{e}} \length \left(\frac{JR_\mf p}{IR_\mf p} \right).
\]
In order to finish the proof, it remains to remove the absolute value and
take the infimums:
\[
\inf_{I \subset J}
\frac{\ehk (IR_\mf p) - \ehk (JR_\mf p)}{\length_{R_\mf p} \left (JR_\mf p/IR_\mf p\right)} 
- \frac{D}{p^e} \leq
\inf_{I \subset J} 
\frac{\length_{R_\mf p} \left (\frq{J}R_\mf p/\frq{I}R_\mf p\right)}{p^{e\hght \mf p}\length_{R_\mf p} \left (JR_\mf p/IR_\mf p\right)} 
\leq \inf_{I \subset J}
\frac{\ehk (IR_\mf p) - \ehk (JR_\mf p)}{\length_{R_\mf p} \left (JR_\mf p/IR_\mf p\right)} 
+ \frac{D}{p^e}.
\]
\end{proof}

\subsection{Dual F-signature exists and is semicontinuous}
Now, we can easily show that the dual F-signature exists.

\begin{corollary}\label{cor dual exists}
Let $(R, \mf m)$ be an F-finite Cohen-Macaulay local ring and $\omega_R$ be its dualizing module.
Then 
\[\dsig(R) = \lim_{e \to \infty} \frac{b_e(R)}{\rank F_*^e \omega_R}\]
exists and equals to $\csig(R)$.
\end{corollary}
\begin{proof}
By Corollary~\ref{cor dual and relative}
it is enough to show that for a system of parameters $\ul{x}$
\[
\lim_{e \to \infty} \frac{1}{p^{e\dim R}}
\min
\left \{\frac{\length \left(\frq{I}/\frq{\langle\ul{x}\rangle}\right)}{\length (I/\langle\ul{x_\mf p}\rangle)} 
\mid  \langle\ul{x}\rangle \subset I \subseteq \langle\ul{x}\rangle : \mf m \right\}
\]
exists. This follows from Theorem~\ref{Polstra}.
\end{proof}

Combining the corollary with Theorem~\ref{thm just comparison} and Corollary~\ref{Cartier semicontinuous} shows that the dual F-signature 
defines a lower semicontinuous function on the spectrum
$\mf p \mapsto \dsig(R_\mf p)$. However, it is easy to give a direct proof 
avoiding Corollary~\ref{Cartier semicontinuous}.


\begin{theorem}\label{thm dual semicontinuous}
Let $R$ be an F-finite Cohen-Macaulay locally equidimensional ring.
If $\Spec R$ is connected, then the convergence of 
\[
\frac{b_e(\mf p)}{p^{e\dim R_\mf p}[k(\mf p):k(\mf p)^{p^e}]} \to \dsig(R_\mf p)
\]
is uniform on $\Spec R$ and $\mf p \mapsto \dsig(R_\mf p)$
is lower semicontinuous.
\end{theorem}
\begin{proof}
We may assume that $\Spec R$ is connected because semicontinuity can be checked on components. 
Furthermore, since F-rationality coincides with regularity for Artinian rings, we may 
assume that $\dim R > 0$.

We start by proving lower semicontinuity of the function
\[
\mf p \mapsto b_e(\mf p) := \max \{N \mid F_*^e \omega_{R_\mf p} \to \bigoplus^N \omega_{R_\mf p} \to 0 \text{ is exact}\}.
\]
We can lift a surjection by collecting denominators: for any $\mf p$ 
there is an element $s \notin \mf p$  such that $F_*^e \omega_{R_s} \to \oplus^{b_e(\mf p)} \omega_{R_s} \to 0$ is exact. 
Thus $b_e(\mf q) \geq b_e(\mf p)$ for any $\mf q \in D(s)$.  
Therefore, for any $\mf p$ such that $b_e(\mf p) > a$ there is an open set $\mf p \in D(s)$
satisfying the same inequality, hence the set $\{\mf q \mid b_e(\mf q) > a\}$ is open.

Because $\Spec R$ is connected, by \cite[Corollary~2.7]{Kunz2} for any $e \geq 1$ 
the function $\mf p \mapsto \alpha_e(\mf p) := p^{e\dim R_\mf p}[k(\mf p):k(\mf p)^{p^e}]$ is constant on $\Spec R$. 
Clearly, $\alpha_e(\mf p) \geq p^e$. Thus, by Corollary~\ref{cor dual and relative} there is a constant $C$ such that for all $\mf p$ 
\[
\frac{b_e(\mf p)}{\alpha_e(\mf p)} + \frac{C}{p^e}\geq
\frac{b_e(\mf p) + C}{\alpha_e(\mf p)}\geq \frac{1}{p^{e\dim R_\mf p}} \min
\left \{\frac{\length \left(\frq{I}/\frq{\langle\ul{x_\mf p}\rangle}\right)}{\length (I/\langle\ul{x_\mf p}\rangle)} 
\mid  \langle\ul{x_\mf p}\rangle \subset I \subseteq \langle\ul{x_\mf p}\rangle : \mf p \right\} \geq \frac{b_e(\mf p)}{\alpha_e(\mf p)}.
\]
Hence, Theorem~\ref{Polstra} and Corollary~\ref{cor dual exists} imply that
\[
\left |\frac{b_e(\mf p)}{p^{e\dim R_\mf p}[k(\mf p):k(\mf p)^{p^e}]} - \dsig(R_\mf p) \right | \leq \frac{C + D}{p^e}.
\] 
This finishes the proof because the uniform limit of semicontinuous functions is semicontinuous. 
\end{proof}

Note that semicontinuity is a vast generalization of  \cite[Theorem~1.11]{Velez} where it was shown that the F-rational locus, i.e., the set $\{\mf p \in \Spec R \mid \dsig(R_\mf p) > 0\}$, is open.

\subsection{Global dual F-signature}
De Stefani, Polstra, and Yao defined global versions of F-signature and Hilbert--Kunz multiplicity in \cite{DSPY}. 
A similar definition can be made for the dual F-signature: if $R$ is a Cohen-Macaulay F-finite ring with a dualizing module $\omega_R$, then we define $b_e(\omega_R)$ by the formula in Definition~\ref{Sannai dual}.
It seems that $b_e(\omega_R)$ may depend on the choice of $\omega_R$, 
but this does not affect the dual F-signature:
it follows from a result of Baidya, see (\ref{eq Baidya}) in the next proof, that $|b_e(\omega_R) - b_e(\omega_R')| \leq \dim R$ for any two dualizing modules $\omega_R, \omega_R'$.

We will now give an analogue of the main result of \cite{DSPY}. 
Our treatment is based on a deeper use of semicontinuity and greatly shortens \cite{DSPY}, since we do not need to show the existence of the global dual F-signature separately.

\begin{theorem}\label{t global}
Let $R$ be a Cohen-Macaulay F-finite domain with a dualizing module $\omega_R$. 
Then 
\[
\dsig(R) := \lim_{e \to \infty} \frac{b_e(\omega_R)}{\rank F_*^e \omega_R} = \min\{\dsig(R_\mf p) \mid \mf p \in \Spec R\}.
\]
In particular, the limit defining $\dsig(R)$ exists and does not depend on the choice of $\omega_R$.
\end{theorem}
\begin{proof}
By Theorem~\ref{thm semicontinuity properties}, $s = \min \{\dsig(R_\mf p) \mid \mf p \in \Spec R\}$ exists due to semicontinuity, so the right-hand side is defined. Similarly, the function $\mf p \mapsto b_e(R_\mf p)$ 
also has a minimum by semicontinuity. Note that $\dim R < \infty$ by \cite[Proposition~1.1]{Kunz2}.
Then 
\begin{equation}\label{eq Baidya}
\min \{b_e(R_\mf p) \mid \mf p \in \Spec R\} - \dim R \leq b_e(\omega_R) \leq \min \{b_e(R_\mf p) \mid \mf p \in \Spec R\},    
\end{equation}
where the first inequality holds by \cite[Theorem~1.1]{Baidya} and the second holds by localizing the definition.
Using the inequalities, it is enough to show that  
$\frac{\min \{b_e(R_\mf p) \mid \mf p \in \Spec R\}}{\rank F_*^e R}$ converges to $s$.
 We will derive this using semicontinuity and uniform convergence of the dual F-signature obtained in Theorem~\ref{thm dual semicontinuous}.

By Theorem~\ref{thm semicontinuity properties} there exists $\varepsilon > 0$ such that for every $\mf p$ either $\dsig(R_\mf p) = s$ or 
$\dsig(R_\mf p) > s + \varepsilon$.
By uniform convergence established in Theorem~\ref{thm dual semicontinuous}, for all $e \gg 0$ we have 
\[
\left | \dsig(R_\mf p) - \frac{b_e(R_\mf p)}{\rank F_*^e R}\right| < \frac\varepsilon 2,
\]
note that $\rank F_*^e R = p^{e\dim R_\mf p} [k(\mf p) : k(\mf p)^{p^e}]$ by \cite[Proposition~2.3]{Kunz2}.
Therefore, if $\dsig(R_\mf p) \neq s$, then for any $\mf m$ such that $\dsig(R_\mf m) = s$ we have
\[
\frac{b_e(R_\mf p)}{\rank F_*^e R} > \dsig(R_\mf p) - \frac\varepsilon 2 > s + \frac\varepsilon 2 > 
\frac{b_e(R_\mf m)}{\rank F_*^e R}.
\]
Thus 
$
\min \{b_e(R_\mf p) \mid \mf p \in \Spec R\} = \min \{b_e(R_\mf p) \mid \mf p \in \Spec R, \dsig(R_\mf p) = s\}
$ 
for all $e \gg 0$.
Then by uniform convergence
\[
\lim_{e \to \infty} \frac{b_e(\omega_R)}{\rank F_*^e R} = \lim_{e \to \infty} \frac{\min \{b_e(R_\mf p) \mid \mf p \in \Spec R\}}{\rank F_*^e R} =  \lim_{e \to \infty} \frac{\min \{b_e(R_\mf p) \mid \mf p, \dsig(R_\mf p) = s\}}{\rank F_*^e R}
= s.
\]
\end{proof}

\subsection{Geometrically regular fibers}

As observed in Proposition~\ref{flat map} F-rational signature does not increase in flat extensions and it is natural to search for conditions that ensure equality. 
This seems to be a difficult question, perhaps due to the lack of complete understanding of the conditions that guarantee that F-rationality 
passes from $R$ to $S$. We will present a generalization of a result of V\'elez (\cite[Theorem~3.1]{Velez}), asserting that F-rationality 
is preserved when $R \to S$ is smooth, by proving that $\dsig(R) = \dsig(S)$ if the closed fiber is \emph{geometrically regular}. Note that $R \to S$ is flat with a geometrically regular fiber if and only 
if $R \to S$ is formally smooth.

There are further results in the literature that concern the transfer of rationality from $R$ to $S$ (\cite{Enescu, Hashimoto,Velez}). In particular, Aberbach and Enescu \cite{AberbachEnescu3} relaxed the assumption to requiring \emph{geometric F-rationality} of the closed fiber.

\begin{lemma}\label{lemma p-degree}
Let $(R, \mf m, k) \to (S, \mf n, \ell)$ be a flat local homomorphism 
of F-finite rings such that the closed fiber $k \to S/\mf m S$ is geometrically regular. 
Then $[k: k^p]=[\ell: \ell^p]$.
\end{lemma}
\begin{proof}
Let $L$ be the fraction field of the regular domain $S/\mf m S$.
Because $L$ is geometrically regular over $k$, it is separable. 
Thus $[L:L^p]=[k: k^p] p^{\trdeg_k L}$.
On the other hand, $\trdeg_k L = \dim S/\mf mS$, so 
$[L : L^p] = [\ell: \ell^p] p^{\trdeg_k L}$ by \cite[2.2]{Kunz2}.
\end{proof}

\begin{theorem}\label{thm dual geometrically regular}
Let $(R, \mf m, k) \to (S, \mf n, \ell)$ be a flat local homomorphism 
of F-finite rings such that the closed fiber $k \to S/\mf m S$ is geometrically regular.
Then $b_e(R) p^{e \dim S/\mf mS} \leq b_e(S)$.
\end{theorem}
\begin{proof}
By a theorem of Andr{\' e} \cite[Page 297]{Andre1},
the homomorphism $R \to S$ is 
regular, i.e., \emph{all} fibers are geometrically regular. Thus, we may apply the Radu--Andr{\' e} theorem (\cite{Andre2, Radu}) 
to learn that the relative Frobenius map $S \otimes F^e_* R \to F^e_* S$ 
is faithfully flat. Note that $S \otimes_R F^e_* R$ is still a local ring due to $F_*^e R$ being purely inseparable 
and that $F^e_* S$ is a finite module over $S \otimes_R F^e_* R$ because $S$ is F-finite. 
It follows that $F^e_* S$ is a free module over $S \otimes_R F^e_* R$.
Its rank can be found after first tensoring with $\otimes_{F_*^e R} F_*^e k$,
which yields the map $S/\mf mS \otimes_k F^e_* k \to F^e_* S/\mf mS$, 
and further tensoring with the quotient field $L$ of $S/\mf mS$ to see that it is enough to compute the rank of $L^{1/p^e}$ over 
$L \otimes_k  k^{1/p^e}$. Since $L$ is separable over $L$, 
this rank is equal to $p^{e\trdeg_k L}$.

By the first paragraph, there is an isomorphism $\oplus^{p^{e\dim S/\mf m S}}(S \otimes_R F^e_* R) \cong F_*^e S$, which implies that
\[
F_*^e \omega_S \cong \Hom_S (F_*^eS, \omega_S )
\cong \oplus^{p^{e\dim S/\mf m S}} \Hom_S (S \otimes_R F^e_* R, \omega_S).
\]
Since $S$ is flat and $S/\mf mS$ is Gorenstein, we have 
$\omega_S = \omega_R \otimes_R S$. This leads to a further isomorphism 
$\Hom_S (S \otimes_R F^e_* R, \omega_S) \cong S \otimes_R F_*^e \omega_R$.
Thus, we can build a surjection 
$F_*^e \omega_S \to \oplus^{b_e(R) p^{e\dim S/\mf m S}} \omega_S$
by tensoring  $F_*^e \omega_R \to \oplus^{b_e(R)} \omega_R \to 0$
with $S$ and taking an appropriate direct sum. 
This finishes the proof.
\end{proof}

\begin{corollary}\label{cor equality geometrically regular}
Let $(R, \mf m, k) \to (S, \mf n, \ell)$ be a flat local homomorphism 
of F-finite rings such that the closed fiber $k \to S/\mf m S$ is geometrically regular.
Then $\dsig(R) = \dsig(S)$.
\end{corollary}
\begin{proof}
We combine Theorem~\ref{thm dual geometrically regular}, 
Proposition~\ref{flat map}, and Corollary~\ref{cor dual exists}.
\end{proof}

\begin{remark}
A special case of the corollary is when the residue field extension 
$k \to \ell$ is separable.  Recall, that a field extension $k \subset L$ is \emph{separable} if $L \otimes_k k^{1/p}$ is still a field 
(hence, equivalently, $L$ is geometrically reduced over $L$) and is \emph{separably generated} if it can be presented as a separable algebraic extension of a purely transcendental extension of $k$. 
The second notion is due to Mac Lane \cite{MacLane} who showed that every finitely generated separable 
extension is separably generated. However, Mac Lane also observed that $L = k(t^{1/p^{\infty}})$ is separable,
but is not separably generated. 
\end{remark}

\subsection{A transformation rule}

In \cite{CRST} it was established 
that for an extension $R \to S$ 
which is {\'e}tale in codimension one
there is a transformation rule connecting F-signatures.
Such rule is impossible for F-rational signature because an example of Singh (\cite[Examples~6.5, 6.6]{Singh} and \cite[Theorem~4.2]{Singh} for the background) shows that F-rationality may not transfer from $R$ to $S$. 

\subsection{Second coefficient}
It was shown in \cite{HMM} that Hilbert--Kunz function has a second coefficient 
in an excellent normal local ring with a perfect residue field.
Subsequent works have shown this result holds with somewhat weaker assumptions: an unpublished manuscript of Hochster and Yao demonstrates that, essentially, only Serre's $(R1)$ condition is needed. For rings over a perfect field, this was also independently shown in \cite{ChanKurano}. 

Using the interpretation via relative Hilbert-Kunz multiplicity, we will now prove a similar result for the dual F-signature. In order to do so, we will follow Huneke's alternative proof from \cite{HunekeSurvey} of the main result in \cite{HMM}. While original stated for normal rings with perfect residue field, a close inspection shows that the treatment in \cite{HunekeSurvey} also applies for merely F-finite rings. In the following result, we also track the dependency of a number of constants in the proof so as to give uniform control over the correction terms. Our careful handling is further motivated by the proof of \cite[Lemma~7.5]{HunekeSurvey}: it appears that it has a small inaccuracy, which we explain and fix, and its last part is left as an exercise, which we believe requires a mild generalization of \cite[Lemma~7.2]{HunekeSurvey}.

\begin{theorem}[Huneke]\label{t Huneke 7.8}
Let $(R, \mf m)$ be an F-finite local normal domain of dimension $d \geq 2$
and characteristic $p> 0$
and let $\ul{x}$ be a fixed system of parameters. 
For any torsion module $N$ there exist a positive constant $C(N)$ with the following property:
for all ideals $\ul{x} \in I$ there exists $\gamma (I, N) \in \mathbb{R}$ such that for all $e \geq 1$
\[
\left |\length (\Tor_1 (N, R/\frq{I})) - \gamma (I, N)p^{e(d-1)}\right | <  C(N)p^{e(d-2)}.
\]
\end{theorem}
\begin{proof}
The result essentially follows from the proof of \cite[Theorem~7.8]{HunekeSurvey}, but the constant $C(N)$ needs to be chosen to work uniformly for all $\ul{x}\in I$. We now carefully trace through the proof and in particular  the preceding results of \cite[Section~7]{HunekeSurvey} to verify this claim.

\begin{step}[Uniform and extended {\cite[Lemma~7.2]{HunekeSurvey}}] 
We will show that for any finitely generated $R$-module torsion module $T$
and a finite generated $R$-module $M$ 
there exists a constant $C_2(T, M)$ such that for all ideals $I$ containing $\ul{x}$ 
\[\length (\Tor_1 (T, M/\frq{I}M)) \leq C_2(T, M) p^{e\dim T}.\]

We start with the case of $M = R$ covered by Huneke. 
In the proof, he shows that
for any finitely generated $R$-module $T$ and any ideal $I$ containing $\ul{x}$ there is a bound
\[
\length (\Tor_1 (T, R/\frq{I})) \leq C(T, \ul{x})p^{e\dim T} + \length (I/{\langle \ul{x} \rangle})\length(T/\frq{\mf m}T)
\leq C(T, \ul{x})p^{e\dim T} + \length (R/{\langle \ul{x} \rangle})\length(T/\frq{\mf m}T),
\] 
where $C(T, \ul{x})$ is given by applying \cite[Theorem~7.3]{HunekeSurvey} 
to the Koszul complex of $\ul{x}$ and does not depend on $I$. 
Since the Hilbert--Kunz function converges \cite{Monsky}, it follows from the above equation that there is a constant $C_2(T)$ 
such that for all ideals $I$ containing $\ul{x}$ we have
\[\length (\Tor_1 (T, R/\frq{I})) \leq C_2(T) p^{e\dim T}.\]

For an arbitrary $M$, we tensor an exact sequence $0 \to \Omega \to \oplus^N R \to M \to 0$ with $R/\frq{I}$ to get an exact sequence
$0 \to \Omega/(\frq{I}\Omega + A_e) \to \oplus^N R/\frq{I} \to M/\frq{I}M \to 0$. After further tensoring with $T$ we now estimate the $\Tor$-module of interest as
\begin{align*}
\length (\Tor_1 (T, M/\frq{I}M)) &\leq N\length (\Tor_1 (T, R/\frq{I}))
+ \length (T \otimes_R \Omega/\frq{I}\Omega)\\
&\leq N C_2(T) p^{e\dim T} + 
\length (T \otimes_R \Omega/\frq{\langle \ul{x}\rangle}\Omega)
\leq N C_2(T) p^{e\dim T} + C p^{e\dim T},
\end{align*}
where the last bound is given by Hilbert--Kunz theory because $\dim \Omega \otimes_R T \leq \dim T$.
\end{step}

\begin{step}[Uniform and extended version of {\cite[Lemma~7.4]{HunekeSurvey}}]
We will show that for any finitely generated $R$-module $T$
of dimension at most $d - 2$ 
and a finitely generated $R$-module $M$ 
there exists a constant $C_4(T, M)$ such that for all ideals $I$ containing $\ul{x}$ 
\[\length (\Tor_2 (T, M/\frq{I}M)) \leq C_4(T, M) p^{e(d-2)}.\]

Following Huneke's proof we deduce 
for any $R$-module $T$ annihilated by a regular sequence $x,y$
and generated by $N$ elements
a bound
\begin{align*}
\length (\Tor_2 (T, M/\frq{I}M)) &\leq N\length (\Tor_2 (R/\langle x,y \rangle, M/\frq{I}M)) + \length (\Tor_1 (T', M/\frq{I}M)), 
\end{align*}
where $T'$ is an $R/\langle x, y\rangle$-syzygy of $T$.
Uusing the Koszul resolution of $R/(x,y)$ we bound 
\begin{align*}
\length (\Tor_2 (T, M/\frq{I}M)) &\leq
N\length (\Tor_1 (R/\langle x,y\rangle, M/\frq{I}M))  + \length (\Tor_1 (T', M/\frq{I}M)). 
\end{align*}
The result of the first step bounds the length of $\Tor$-modules 
and shows that we may take $C_4(T, M) = NC_2(R/\langle x,y \rangle, M) + C_2(T', M)$. 
\end{step}

\begin{step}[Uniform version of {\cite[Lemma~7.5]{HunekeSurvey}}]
We need to show that for any torsion-free finitely generated $R$-module
there exists a constant $C_5(M)$ such that for all ideals $I$ containing $\ul{x}$ 
\[\length (\Tor_1 (M, R/\frq{I})) \leq C_5(M) p^{e(d-2)}.\]

Huneke's proof first passes to the double dual $M^{**}$ by observing that
\begin{align*}
\length (\Tor_1 (M, R/\frq{I})) &\leq \length (\Tor_1 (M^{**}, R/\frq{I})) + \length (\Tor_1 (M^{**}/M, R/\frq{I}))  + \length (\Tor_2 (M^{**}/M, R/\frq{I})) \\
&\leq \length (\Tor_1 (M^{**}, R/\frq{I})) + \left (C_2(M^{**}/M) + C_4(M^{**}/M)\right ) p^{e(d-2)}.
\end{align*}
Thus, we may assume that $M$ is reflexive.

Let $\to F_1 \to F_0 \to M \to 0$ be a part of a free resolution of $M$. 
Let $Z_e$ is the kernel of the induced map
$(F_1 \to F_0) \otimes R/\frq I$ and $B_e$
be the image of the induce map $(F_2 \to F_1) \otimes R/\frq I$.
Because $\Tor_1 (M, R/\frq{I}) = Z_e/B_e$, we derive an exact sequence
\begin{equation}\label{eq Tor formula}
0 \to \Tor_1 (M, R/\frq{I}) \to F_1/B_e \to F_1/Z_e \to 0.    
\end{equation}
By tensoring the exact sequence 
defining the first syzygy $\Omega_1$ of $M$, $F_2 \to F_1 \to \Omega_1 \to 0$,
with $R/\frq{I}$ we identify\footnote{The proof in \cite{HunekeSurvey} seems to claim that $F_1/Z_e \cong \Omega_1/\frq{I}\Omega_1$.}  $F_1/B_e \cong \Omega_1/\frq{I}\Omega_1$.

As explained in Huneke's proof, one can choose a regular sequence $x,y$ 
so that $\langle x,y \rangle$ annihilates all $\Tor_1 (M, \bullet)$.
It follows that tensoring $(\ref{eq Tor formula})$ with $R/\langle x, y\rangle$
yields the bound 
\begin{align*}
\length (\Tor_1 (M, R/\frq{I})) &\leq \length (\Tor_1 (R/\langle x,y \rangle, F_1/Z_e)) + \length (\Omega_1/\langle x,y, \frq{I} \rangle\Omega_1) \\
&\leq \length (\Tor_1 (R/\langle x,y \rangle, F_1/Z_e)) + \length (\Omega_1/\langle x,y, \frq{\langle \ul{x}\rangle } \rangle\Omega_1) \\
&\leq \length (\Tor_1 (R/\langle x,y \rangle, F_1/Z_e)) + C p^{e(d-2)}
\end{align*}
from Hilbert--Kunz theory.
We estimate the remaining $\Tor$-module by tensoring 
$
0 \to F_1/Z_e \to F_0/\frq{I}F_0 \to M/\frq{I}M \to 0$
with $R/\langle x,y \rangle$ and obtain that 
\begin{align*}
\length (\Tor_1 (R/\langle x,y \rangle, F_1/Z_e))
&\leq \length (\Tor_2 (R/\langle x,y \rangle, M/\frq{I}M))
+ \length (\Tor_1 (R/\langle x,y \rangle, F_0/\frq{I}F_0))\\
&\leq 
C_4(R/\langle x,y \rangle, M)p^{e(d-2)} + C_2(R/\langle x,y \rangle, F_0)p^{e(d-2)}.
\end{align*}
The assertion follows. 
\end{step}

\begin{step}[Uniform {\cite[Corollary~7.6]{HunekeSurvey}}]
The assertion follows by replacing \cite[Lemma~7.5]{HunekeSurvey} by its uniform version. Hence
for any $R$-module $M$ and any $i \geq 2$ there exists a constant $C_{6, i}$ such that 
$\length (\Tor_i (M, R/\frq{I})) \leq C_{6,i}(M) p^{e(d-2)}$ for all $\ul{x} \in I$.
\end{step}

\begin{step}[Uniform {\cite[Corollary~7.7]{HunekeSurvey}}]
The proof shows that for an exact sequence
$0 \to T_1 \to T_2 \to T_3 \to 0$ of torsion modules we can bound 
\[
\left |\sum_{i = 1}^3 (-1)^{i+1} \length (\Tor_1 (T_i, R/\frq{I})) \right| \leq 
\length (\Tor_2 (T_3, R/\frq{I})) + \sum_{i = 1}^3 (-1)^{i + 1} \length (T_i/\frq{I}T_i).
\]
Since $R$ is a domain we can find $c \neq 0$ that annihilates $T_1, T_2, T_3$. 
Because Hilbert--Kunz multiplicity is additive in short exact sequences and converges uniformly (\cite[Theorem 3.6]{Tucker}),
by working in $R/(c)$ we can find a constant $D$ such that for all $I$ containing $\ul{x}$ there is a bound 
\[\left |\length (T_3/\frq{I}T_3) + \length (T_1/\frq{I}T_1) - \length (T_2/\frq{I}T_2)\right | < Dp^{e(d-2)}.\] 
Therefore, we take $C_7 = C_{6,2}(T_3) + D$ to bound 
$\left |\sum_{i = 1}^3 (-1)^{i+1} \length (\Tor_1 (T_i, R/\frq{I})) \right| \leq C_7 p^{e(d-2)}$.
\end{step}

\begin{step}[The proof of the assertion]
Last, we trace the proof \cite[Theorem~7.8]{HunekeSurvey} to show that it works in the F-finite case and provides a uniform constant. First, it is explained 
that we may reduce to $N = R/Q$ where $Q$ is a height one prime. Observe that $[k(Q) : k(Q)^{p}] = p^{d-1}[k:k^p]$ by \cite[Corollary~2.7]{Kunz2}, 
so \cite[(10) in the proof of Theorem~7.8]{HunekeSurvey} can be replaced with 
\begin{equation}\label{new (10)}
\left|p^{d-1}[k:k^p] \length\big (\Tor_1 (R/Q, R/\frq{I})\big) - \length \big(\Tor_1 ((R/Q)^{1/p}, R/\frq{I})\big) \right| 
\leq (C_2(T) + C_4(T))p^{e(d-2)}.
\end{equation}

From the long exact sequence for the tensor product we derive 
that the quantity 
\begin{align*}
\length \big (\Tor_1 ((R/Q)^{1/p}, R/\frq{I}) \big) - \length (Q^{1/p} \otimes_R R/\frq{I}) +
\length (R^{1/p} \otimes_R  R/\frq{I}) - \length ((R/Q)^{1/p} \otimes_RR/\frq{I}) 
\end{align*}
is non-negative and is bounded above by 
$\length (\Tor_1 (R^{1/p},  R/\frq{I}))$.
Thus,
\[
\left|
\frac{\length (\Tor_1 ((R/Q)^{1/p}, R/\frq{I}))}{[k: k^p]}
- \length (Q/I^{[p^{e+1}]}Q) +
\length (R/I^{[p^{e+1}]}) - \length (R/\langle Q, I^{[p^{e+1}]}\rangle )
\right| \leq \frac{C_5(R^{1/p})}{[k: k^p]}p^{e(d-2)}.
\]
He notes that the alternating sum of lengths can be computed 
by tensoring $0 \to Q \to R \to R/Q \to 0$ with $R/I^{[p^{e+1}]}$, 
so it follows that
\[
\left|
\frac{\length \big (\Tor_1 ((R/Q)^{1/p}, R/\frq{I}) \big)}{[k: k^p]}
- \length \big(\Tor_1 (R/Q, R/I^{[p^{e+1}]}) \big)
\right| \leq \frac{C_5(R^{1/p})}{[k: k^p]}p^{e(d-2)}.
\]
By plugging this into (\ref{new (10)}) we get that
\[
\left |
p^{d-1} \length\big (\Tor_1 (R/Q, R/\frq{I})\big) -  \length \big(\Tor_1 (R/Q, R/I^{[p^{e+1}]}) \big) \right |
\leq \frac{C_5(R^{1/p}) + C_2(T) + C_4(T)}{[k:k^p]}p^{e(d-2)},
\]
which allows to invoke \cite[Lemma~3.5(iii)]{PolstraTucker}
and deduce the existence of 
\[
\gamma(I, R/Q) := 
\lim_{e \to \infty} \frac{1}{p^{e(d-1)}}\length\big (\Tor_1 (R/Q, R/\frq{I})\big)
\]
and estimate the convergence rate by
\[
\left|\length\big (\Tor_1 (R/Q, R/\frq{I})\big) - p^{e(d-1)} \gamma (I, R/Q)
\right | \leq 2\frac{C_5(R^{1/p}) + C_2(T) + C_4(T)}{[k:k^p]}p^{e(d-2)}.
\]
\end{step}

\end{proof}

\begin{corollary}[{\cite[Proposition~7.9, Corollary~7.10]{HunekeSurvey}}]\label{c Huneke 7.10}
Let $(R, \mf m)$ be an F-finite local normal domain of dimension $d \geq 2$
and characteristic $p > 0$ and $\ul{x}$ be a system of parameters.
Let $M$ be a finitely generated torsion-free $R$-module. 
There exists a constant $C_9(M) \in \mathbb R$ such that 
for any ideal $I$ that contains $\ul{x}$ 
there exists a constant $\gamma(I, M)$ such that 
\[
\left |\length (\Tor_0 (M, R/\frq{I})) - r\length (\Tor_0 (R, R/\frq{I})) - \gamma (I, M)p^{e(d-1)}\right | <  C_9(M)p^{e(d-2)}.
\]

In particular, for any ideal $I$ that contains $\ul{x}$ 
there exists $\gamma(I, R^{1/p}) \in \mathbb R$ such that 
\[
\left |[k:k^p]\length (\Tor_0 (R, R/I^{[p^{e+1}]})) - p^{d}[k:k^p]\length (\Tor_0 (R, R/\frq{I})) - \gamma (I, R^{1/p})p^{e(d-1)}\right | <  C_9(R^{1/p})p^{e(d-2)}.
\]

\end{corollary}
\begin{proof}
Huneke's proof of \cite[Proposition~7.9]{HunekeSurvey} applies verbatim to the first statement by replacing his references to \cite[Lemma~7.5, Theorem~7.8]{HunekeSurvey} by the uniform versions obtained in Theorem~\ref{t Huneke 7.8} and his appeal to the convergence of the Hilbert--Kunz sequence 
by the uniform convergence estimate from \cite[Theorem 3.6]{Tucker}.

The second statement is obtained by taking $M = R^{1/p}$
and noting that its rank is $p^{d}[k:k^p]$ and that 
$
\length (\Tor_0 (R, R^{1/p} \otimes_R R/I^{[p^{e}]}))
= \length (R^{1/p} \otimes_R R/I^{[p^{e}]}) = 
[k : k^p] \length (R/I^{[p^{e+1}]}).
$

\end{proof}

\begin{theorem}\label{t Huneke 7.11}
Let $(R, \mf m)$ be an F-finite local normal domain of dimension $d \geq 2$ and $\ul{x}$ be a system of parameters.
Then there exists a constant $C \geq 0$ such that for every ideal $\ul{x} \in I$ there exists $\beta (I)$ such that
\[
|\length (R/\frq{I}) - \ehk(I)p^{ed} - \beta(I)p^{e(d-1)}| < Cp^{e(d-2)}.
\]
\end{theorem}
\begin{proof}
The assertion is a uniform version of \cite[Theorem~7.11]{HunekeSurvey} and is obtained from its proof by replacing 
$\gamma(R^{1/p})$ with $\gamma(R^{1/p})/[k:k^p]$ in the definition of $\epsilon_q$, replacing the reference to \cite[Corollary~7.10]{HunekeSurvey}
by the uniform estimate in Corollary~\ref{c Huneke 7.10}, 
and quantifying the geometric series trick through \cite[Lemma~3.5(iii)]{PolstraTucker}.
\end{proof}

\begin{corollary}\label{c Huneke bound}
Let $(R, \mf m)$ be an F-finite local normal domain of dimension $d \geq 2$ and $\ul{x}$ be a system of parameters.
Then there exists a constant $C \geq 0$ such that for every ideal $\ul{x} \in I$ there exists $\beta (I)$ such that
\[
|\length (\frq{I}/\frq{\langle \ul x \rangle}) - \left (\ehk(\langle \ul{x} \rangle) - \ehk(I) \right )p^{ed} + \beta(I)p^{e(d-1)}| < 2Cp^{e(d-2)}.
\]
\end{corollary}
\begin{proof}
The statement follows from Theorem~\ref{t Huneke 7.11}
by using the estimates for $\langle \ul{x}\rangle$ and $I$.
\end{proof}

\begin{theorem}\label{t second coefficient}
Let $(R, \mf m, k)$ be an F-finite F-rational local domain of dimension $d \geq 2$.
Then there exists a constant $\beta$ such that
\[
\frac{b_e(R)}{[k:k^{p^e}]}   = \dsig(R)p^{ed} + \beta p^{e(d-1)} + O(p^{e(d-2)}).
\]
\end{theorem}
\begin{proof}
As discussed in Remark~\ref{rem gaps on the socle},
semicontinuity of the relative Hilbert--Kunz multiplicity 
on the Grassmannian of the socle implies that
there exists $\varepsilon > 0$
such that whenever
$\ehk (\langle \ul{x} \rangle) - \ehk(I) < (\csig(R) + \varepsilon) \length (I/\langle \ul{x} \rangle)$
for an ideal $\ul{x} \in I \subseteq \langle \ul{x} \rangle: \mf m$ then
$\ehk (\langle \ul{x} \rangle) - \ehk(I) = \csig(R) \length (I/\langle \ul{x} \rangle)$.

Due to the uniform convergence (Theorem~\ref{Polstra}), there exists $e_0 > 0$ such that for all socle ideals $I$ and all $e \geq e_0$
\[
\left |\frac{\length (\frq{I}/\frq{\langle \ul x \rangle})}{p^{ed}} - \left (\ehk(\langle \ul{x} \rangle) - \ehk(I) \right)p^{ed} + \beta(I)p^{e(d-1)}\right | < \frac{\varepsilon}2.
\]
Hence, if $I, J$ are arbitrary socle ideals such that $\ehk (\langle \ul{x} \rangle) - \ehk(I) = \csig(R) \length (I/\langle \ul{x} \rangle)$
and $\ehk (\langle \ul{x} \rangle) - \ehk(J) > \csig(R) \length (J/\langle \ul{x} \rangle)$, then for all $e \geq e_0$
\[
\frac{\length (\frq{J}/\frq{\langle \ul x \rangle})}{p^{ed} \length (J/\langle \ul{x} \rangle)} > 
\frac{\ehk (\langle \ul{x} \rangle) - \ehk(J)}{ \length (J/\langle \ul{x} \rangle)} - \frac{\varepsilon}2
> \csig(R) +  \frac{\varepsilon}2 > \frac{\length (\frq{I}/\frq{\langle \ul x \rangle})}{p^{ed} \length (I/\langle \ul{x} \rangle)}.
\]
Therefore, by Corollary~\ref{cor dual and relative} for all $e \geq e_0$
\[
b_e(R) = [k:k^{p^e}] \min \left \{\frac{\length (\frq{I}/\frq{\langle \ul x \rangle})}{\length (I/\langle \ul{x} \rangle)} \mid \langle \ul{x} \rangle \subset I \subseteq \langle \ul{x} \rangle: \mf m, \frac{\ehk (\langle \ul{x} \rangle) - \ehk(I)}{\length (I/\langle \ul{x} \rangle) } = \csig(R) \right\} + O(1).
\]

Now, for any socle ideal $I$ such that $\ehk (\langle \ul{x} \rangle) - \ehk(I) = \csig(R) \length (I/\langle \ul{x} \rangle)$
consider the sequence 
\[
c_e(I) = \frac{1}{p^{e(d-1)}}\left (\frac{\length (\frq{I}/\frq{\langle \ul x \rangle})}{\length (I/\langle \ul{x} \rangle)}
- p^{ed} \csig(R)  \right).
\]
By Corollary~\ref{c Huneke bound} this sequence converges uniformly, at the rate $2C/p^e$ independent of $I$, to $\beta(I)/\length (I/\langle \ul{x} \rangle)$. Thus by taking the infimum in the inequality 
\[
c_e(I) - 2C/p^e \leq \beta(I)/\length (I/\langle \ul{x} \rangle) \leq c_e(I) + 2C/p^e
\]
we obtain that $\inf_I c_e(I)$ converges, at the same rate, to $\inf_I \beta(I)/\length (I/\langle \ul{x} \rangle)$. 
Therefore, 
\begin{align*}
\frac{b_e(R)}{[k:k^{p^e}]}  
= \csig(R) p^{ed} + \inf \left \{\frac{\beta(I)}{\length (I/\langle \ul{x} \rangle)}\mid \frac{\ehk (\langle \ul{x} \rangle) - \ehk(I)}{\length (I/\langle \ul{x} \rangle) } = \csig(R) \right\} p^{e(d-1)} + O(p^{e(d-2)}).
\end{align*}
\end{proof}

\begin{corollary}
Let $(R, \mf m, k)$ be an F-finite $\mathbb Q$-Gorenstein F-rational local domain of dimension $d \geq 2$ with a perfect residue field.
Then 
\[
b_e(R) = \dsig(R)p^{ed} + O(p^{e(d-2)}).
\]
\end{corollary}
\begin{proof}
By \cite{Kurano}, the second coefficient $\beta(I)$ is zero for every $\mf m$-primary ideal $I$.
\end{proof}

\section{A formula for toric varieties}\label{Toric}
Besides Gorenstein examples where dual F-signature and F-signature 
coincide, we do not have many examples where the dual F-signature was computed.
In \cite{Sannai} Sannai computed dual F-signature of the Veronese subrings of $k[x,y]$ (Example~\ref{different}). More generally, the results of Nakajima in \cite{Nakajima} can be used for computations in cyclic quotients of $k[x,y]$.
In \cite{HashimotoDual} Hashimoto studied the dual F-signature of invariant subrings and was able to characterize vanishing of $\dsig (\omega_{R})$ representation-theoretically even in the non Cohen-Macaulay case. 
In particular, he showed that $\dsig (\omega_{R^G}) > 1/|G|$ whenever it is positive. 

We suspect that it might be easier to work with the Hilbert--Kunz definition and will now  
discuss F-rational signature in the toric case. Note that cyclic quotient surface singularities are exactly two-dimensional toric singularities (\cite[2.2]{Fulton}). 
Hilbert--Kunz multiplicity of monomial ideals in toric rings  was computed combinatorially by Watanabe (\cite{WatanabeToric}).
We extend his idea and will start with a recipe for computing $\Wsig(\omega_R/N)$ where $\mf m \omega_R \subseteq N$ 
is torus-invariant.

Specifically, consider a lattice $L$ (i.e., a group isomorphic to $\mathbb Z^n$) and a convex rational polyhedral cone $\sigma \subset L \otimes_{\mathbb Z} \mathbb R$. We can always assume that $L = \mathbb Z^n$, but sometimes it is more convenient with work with a proper sublattice of $\mathbb Z^n$. 
Let $M = \Hom (L, \mathbb Z)$ be the dual lattice and the dual cone is defined as 
\[\sigma^\vee = \{u \in M \otimes_{\mathbb Z} \mathbb R \mid \langle u, v \rangle \geq 0, v \in \sigma\},\]
and we let $R = k[\sigma^\vee \cap M]$ to be a monomial subring of a Laurent polynomial ring $k[x_1^{\pm 1}, \ldots, x_n^{\pm 1}] = k[\mathbb Z^n]$. It is invariant under the torus action $x_i \mapsto t_ix_i$ with $t_i \in k^\circ$. 

We say that $\sigma^\vee$ is pointed if $\sigma^\vee \cap -\sigma^\vee = 0$,
or, equivalently, if $\sigma$ spans $L \otimes_{\mathbb Z} \mathbb R$. In this case, $R = k[\sigma^\vee \cap M]$
has a distinguished maximal ideal generated by all non-trivial monomials. It corresponds to the unique 
fixed point of the torus action on the toric variety $\Spec R$.
We will denote this ideal by $\mf m$.

By the work of Hochster \cite{HochsterTori}, $R$ is Cohen-Macaulay and 
it is also known that one can choose a torus-invariant dualizing ideal (\cite{Danilov}) corresponding to the interior $(\sigma^\vee)^\circ$ of the cone. From now on, we will use $\omega_R$ to denote this ideal.


Suppose that $N \subset \omega_R$ is a monomial ideal, 
we identify $N$ with the finitely many monomials in its complement $\omega_R\setminus N$. 
Following the interpretation of $\Wsig(\omega_R/N)$ in Remark~\ref{rmk trace}, 
we are searching for monomials $x^u \in \omega_R^{1/p^e}$ 
such that $\trace_R (r^{1/p^e} x^u) \notin N$ for some $r \in R$.
It is known (see \cite[page 1780]{HsiaoSchwedeZhang}) that $\trace^e$ is a projection on the lattice: for $u \in 1/{p^e} M$
\[
\trace^e (x^u) = \begin{cases}
x^u &\text{ if } u \in M,\\
0 &\text{ otherwise.}
\end{cases}
\]
Hence, we need $x^{u} \in \omega_R^{1/p^e}$ such that $x^{u + v} \in \omega_R \setminus N$ for some $x^v \in R^{1/p^e} = k[1/p^e M \cap \sigma^\vee]$, i.e., 
\[u \in P_e := \bigcup \left \{\frac{1}{p^e}M \cap (\sigma^\vee)^\circ \cap \left (a  + \frac{1}{p^e}M \cap (-\sigma^\vee) \right) \mid a \in \omega_R\setminus N \right\}.\]
Thus, we see that 
\[
\Wsig(\omega_R/N) = \lim_{e \to \infty}
\frac{\dim_{k^{1/p^e}} \omega_R^{1/p^e}/Z_e(I)}{p^{e\dim R}} = \lim_{e \to \infty} \frac{|P_e|}{p^{e\dim R}}.
\]
Since $\dim R = \dim \sigma^\vee$, from the Ehrhart theory we obtain that the limit is the normalized volume of a region:
\begin{align}\label{toric volume equation}
\Wsig(\omega_R/N) =  \vol \left (\bigcup \left \{\sigma^\vee \cap \left (a  - \sigma^\vee \right) \mid a \in \omega_R\setminus N \right\} \right)/\vol(M),
\end{align}
where $\vol(M)$ is the Eucledian volume of an elementary parallelepiped of the lattice.

We will now show that F-rational signature can be computed from the toric quotients only.

\begin{proposition}\label{prop go toric}
If $R = k[\sigma^\vee \cap M]$ is an affine pointed toric ring, then 
\[
\csig(R) = \frac{1}{\vol(M)}\inf_S \frac{1}{|S|}
\vol \left (\bigcup \left \{\sigma^\vee \cap \left (a  - \sigma^\vee \right) \mid a \in S \right\} \right),
\]
where the minimum ranges through all finite non-empty subsets $S$ of points
$a \in (\sigma^\vee)^\circ$. 
Moreover, we may restrict to subsets $S$ consisting of points
$a \in (\sigma^\vee)^\circ$ such that $a - m \notin (\sigma^\vee)^\circ \text{ for all } 0 \neq m \in \sigma^\vee$
(corresponding to monomials in $\omega_R \setminus \mf m \omega_R$).
\end{proposition}
\begin{proof}
It follows from (\ref{toric volume equation}) and the discussion preceding it  that the invariant on the right-hand side 
is the infimum computed over a subset of quotients of $\omega_R$, so 
it is clearly no less than $\sWsig(R) = \csig(R)$.
Furthermore, the right-hand side is independent of the ground field by (\ref{toric volume equation}).
Since $\csig(R)$ cannot increase when the ground field is extended, it is enough to show the equality at some field extension. 
Thus, we may assume that $k$ is algebraically closed. 

Since $\omega_R$ is $T$-invariant, $T$ acts on the Grassmannian of rank $n$ quotients, so 
we follow the construction in Theorem~\ref{t semi on Grassmannian} to show that $\Wsig(\omega_R/N)= 
\Wsig(\omega_R/t \cdot N)$ for any $t \in T$.
Namely, choose a $T$-invariant set $\{r_i\}_{i = 1}^{m_e}$ such that $F_*^e r_i$ generate $F_*^e R$ as an $R$-module and let
\[
g_e(N) \colon F_*^e \omega_R/\frq{\mf m}\omega_R \xrightarrow{\oplus \trace^e(F_*^e r_i \cdot \bullet)} \bigoplus \omega_R/\mf m\omega_R \to \bigoplus \omega_R/N.
\]
Since $T$ acts invertibly and the first map is $T$-invariant, one can easily check that $\ker g_e(t \cdot N) = t \cdot \ker g_e(N)$. Thus $\rank g_e(N) = \rank g_e(t \cdot N)$, so $\Wsig(N) = \Wsig(t \cdot N)$. 

Because $k$ is algebraically closed, the infimum in $\sWsig(R)$ is attained by Corollary~\ref{infimum achieved}. 
Then the locus of the Grassmannian that minimizes $\Wsig(N)$ is $T$-invariant and, because the Grassmannian is projective, we may apply the Borel fixed point theorem (\cite[Theorem 10.4]{Borel}) to conclude that there is a $T$-invariant minimizer. It remains to show that it is monomial.

Now, $k^\circ$ is an infinite cyclic group and we can use its generator $g$ to adapt the  usual argument showing 
that a $T$-invariant ideal $I$ of the Laurent ring is monomial. For any $f \in I$ we can write 
$
f = \sum_{i = -a}^b x_1^i p_i(x_2, \ldots, x_n).
$
For the weight $t = (t_1, 1, \ldots, 1)$ we see that 
\[
t \cdot f = \sum_{i = -a}^b t_1^i x_1^i p_i(x_2, \ldots, x_n) \in I.
\]
We can choose $t_1 = g$, then its powers are all distinct. Because $t^2 \cdot f, \ldots, t^{a + b} \cdot f \in I$, we can use the Vandermonde determinant to show that each term $x_1^i p_i(x_2, \ldots, x_n) \in I$. Repeating the process for $x_2, \ldots, x_n$ we obtain that $I$ is monomial. 
\end{proof}

In the second case, there are only finitely many ideals to check, so we compute the F-rational signature by evaluating all possible subsets.  The volume of the union in (\ref{toric volume equation}) can be computed from the basic volumes $\vol \left( \sigma^\vee \cap \left (a  - \sigma^\vee \right) \right)$ by inclusion-exclusion, but 
this does not give an efficient algorithm.

\begin{example}
It is sometimes easier to work with a sublattice. We can consider the $n$th Veronese subring of the polynomial ring in $d$ variables
as a toric variety for the sublattice $L \subset\mathbb Z^d$ formed by vectors whose sum of components is divisible by $n$
and the positive orthant as the cone. If $e_i$ are a standard basis of $\mathbb Z^d$, then 
$ne_1, e_1 - e_i, i \geq 2$ form a basis of $L$, so we can easily compute that $\vol (L) = n$. 

Integral points in $\omega_V$ are vectors $a = (a_1, \ldots, a_d)$
such that all components are positive and their sum is divisible by $n$.  
One can further see that an integral point $a \in \omega_V \setminus \mf m\omega_V$ is 
such that $\sum a_i = \left \lceil \frac{d}{n} \right \rceil n$, since any smaller sum will have some $a_i = 0$ (so it is not in $\omega_R$) and any larger sum can be decreased without violating positivity of $a_i$ (so it is in $\mf m\omega_R$).
For any such point we can easily compute the volume 
$\vol (\sigma^\vee \cap \left (a  - \sigma^\vee \right)) = \prod a_i$. 
This volume is minimized when all but one components are equal to $1$, 
giving the minimum of $\left\lceil \frac{d}{n} \right\rceil n - d + 1$. 
See Hochster--Yao \cite[Example~7.4]{HochsterYao} for another approach.

It is well-known that there are $\binom{N -1}{d-1}$ integer points such that $a_i > 0, \sum a_i = N$.
If we use all of them in (\ref{toric volume equation}), we obtain a shape which is best described as 
a ``building block pyramid''. The volume of it is equal to the number of integer points with positive coordinates 
such that $\sum a_i \leq N$, which can be seen by identifying each integer unit cube with its vertex with the largest coordinates (e.g., top-right for a square). Hence, the volume of that region is 
\[
\sum_{k = 1}^N \binom{N-1}{d-1} = \binom{N}d.
\]
Thus, using $N =  \lceil \frac{d}{n} \rceil n$, we obtain 
from (\ref{toric volume equation})
that the relative Hilbert--Kunz multiplicity of the entire socle is 
\[
\frac{\binom{N}d}{n \binom{N-1}{d-1}}
= \frac{N}{dn} = \frac{1}{d} \left\lceil \frac{d}{n} \right\rceil.
\]

Note that $\lceil \frac d n\rceil n \geq d$ with equality if and only if $d$ is divisible by $n$, 
which is itself equivalent, by the socle formula, to $V$ being Gorenstein. 
Hence, $\rsig(V) = \csig(V)$ if and only if $V$ is Gorenstein, because 
$\rsig(V) = (\lceil \frac d n\rceil n - d + 1)/n \geq  \frac{1}{d} \left\lceil \frac{d}{n} \right \rceil = \csig(V)$.
\end{example}

\begin{remark}\label{no HY}
The $2$nd Veronese of $k[x,y,z]$ is an example of a singular ring such that $\rsig(V) = 1$.
\end{remark}

\begin{question}
Is $\csig(V_n) = \frac{1}{d} \left\lceil \frac{d}{n} \right \rceil$?
\end{question}

The equality is easy to verify when the Cohen-Macaulay type is close to $d$. 
For example, in the simplest non-Gorenstein case, such as $n = 2$ or $n = d+1$, we have $\lceil \frac{d}{n} \rceil n = d + 1$.
Then $\omega_R \setminus \mf m \omega_R$ has $d$ integer points and they have the form $(1, \ldots, 1, 2, \ldots, 1)$.
By symmetry, any collection of $k$ points will have the volumes $2 + (k-1)$, so the relative Hilbert--Kunz multiplicity is $(k + 1)/(kn)$, which is  minimized for $k = d$. With slightly more effort, one can also verify combinatorially 
the next case $\lceil \frac{d}{n} \rceil n = d + 2$.

\begin{example}
Suppose $C \subseteq \mathbb{R}^3$ is the strongly convex rational polyhedral cone with rays through the points $[0,0,1],[0,2,1],[3,0,1]$ and $[1,-1,1]$. In other words, $C$ is the cone over the polytope pictured below
\begin{center}
\includegraphics[]{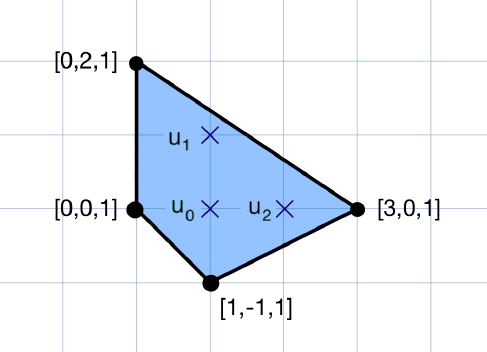}
\end{center}
in the $z = 1$ plane. Let $k$ be any $F$-finite field of characteristic $p > 0$ and consider $R = k[C \cap \mathbb{Z}^n]$ and $\mathfrak{m}$ the homogeneous maximal ideal. We view $\omega_R$ as $k[C^\circ \cap \mathbb{Z}^n]$, the ideal of $R$ generated by the monomials corresponding the interior lattice points of $C$. One checks that $u_0 = [1,0,1], u_1 = [1,1,1], u_2 = [2,0,1]$ correspond to a minimal set of generators for $\omega_R$, \textit{i.e.} the images of the corresponding monomials give a $k$-basis for the vector space $\omega_R / \mathfrak{m} \omega_R$. In particular, we see that the type of $R$ is three.

For each subset of indices $\emptyset \neq S \subseteq \{0,1,2\}$, let $N_S$ be the $R$-submodule of $\omega_R$ generated by $\mathfrak{m} \omega_R$ and $x^{u_i}$ for $i \in \{0,1,2\} \setminus S$. We have that $W_S = \omega_R / N_S$ is then a $k$-vector space with basis given by the images of $x^{u_i}$ for $i \in S$. To make our use of \eqref{toric volume equation} transparent, consider the rational polytopes $P_i = C \cap (u_i - C)$ for $i = 0,1,2$; we then have that $\Wsig(W_S)$ is the Euclidean volume of $\bigcup_{i\in S} P_i$ divided by the number of elements in $S$. One computes
\begin{equation*}
    \begin{array}{l@{\qquad}l@{\qquad}l}
        \Wsig(W_{\{0\}}) = 136/441 & \Wsig(W_{\{0,1\}}) = 187/882 & \Wsig(W_{\{0,1,2\}}) = 101/588 \\
        \Wsig(W_{\{1\}}) = 167/882 & \Wsig(W_{\{0,2\}}) = 89/441 & \\
        \Wsig(W_{\{2\}}) = 80/441 & \Wsig(W_{\{1,2\}}) =  571/3528 & \\
    \end{array}
\end{equation*}
and checks that the smallest value achieved is thus $\dsig(R) = \csig(R) =  571/3528$. In contrast, taking the minimum of the $\Wsig(W_{\{i\}})$ for $i = 0,1,2$ gives $\rsig(R) = 80/441$ which is strictly larger. Moreover, unlike what was seen for Veronese subrings in the previous example, $\csig(R)$ is also not achieved by taking the (normalized) relative Hilbert-Kunz multiplicity for the entire socle modulo a parameter ideal. Explicitly, note that we have
\begin{equation*}
    \Wsig(W_{\{0,1,2\}}) = \Wsig(\omega_R / \mathfrak{m} \omega_R) = \frac{\ehk(I) - \ehk(I:\mathfrak{m})}{\ell\left((I:\mathfrak{m})/(I)\right)} = 101/588
\end{equation*}
for an ideal $I \subseteq \mathfrak{m}$ corresponding to a parameter ideal of $R_\mathfrak{m}$.

\end{example}

\section{Some open questions}
This work opens a number of natural questions inspired by the existing theory of Hilbert--Kunz multiplicity and F-signature. We want to highlight questions that were touched but not resolved in this work.

\subsubsection{Beyond F-finite}
One benefit of $\csig(R)$ is that it is defined via Hilbert--Kunz theory and the definition makes sense for any local ring 
of positive characteristic. We developed Section~\ref{relative} without the F-finite hypothesis and showed that $\csig(R)$ many good properties. However, all further results are tied to the dual F-signature.

\begin{question}
Is $\csig(R)$ semicontinuous without the F-finite assumption? 
\end{question}

We suspect that for this question one needs an interpretation of dual F-signature for non F-finite rings. This is related to the following question, since we would like to get rid of the residue field extension appearing in the definition of the dual F-signature.
\begin{question}
Is $b_e$ always divisible by $[k : k^{p^e}]$?
\end{question}
Note that this will follow if one could remove the constant $C$ from Corollary~\ref{cor dual and relative}.

\subsubsection{Good fibers}
\begin{question}\label{flatty}
If $R \to S$ is a flat local map, under what conditions is $\csig(R) = \csig(S)$?
\end{question}

In particular, it is desirable to show that $\csig (R) = \csig(S)$ for a regular map. 
One way to achieve this would be to reduce to F-finite case by means of the 
so-called \emph{$\Gamma$-construction} \cite{HochsterHuneke2}. 
This motivates the following question.

\begin{question}
If $\Gamma$ varies over the cofinite subsets of a $p$-base of a coefficient field $k$ of $R$, then is $\sup_{\Gamma} \csig(R^\Gamma) = \csig(R)$? Is the supremum achieved, \textit{i.e.} is $\sup_{\Gamma} \csig(R^\Gamma) = \csig(R^{\Gamma'})$ for all sufficiently small $\Gamma'$? 
\end{question}

\noindent
Here, note that we have always have $\csig(R^{\Gamma}) \leq \csig(R)$, and moreover $R$ is $F$-rational if and only if $\csig(R^{\Gamma'}) > 0$ for all sufficiently small $\Gamma'$.

\begin{question}
Let $(R, \mf m, k)$ be a complete F-finite Cohen-Macaulay local ring and $\ell$ be a finite separable field extension of $k$. Do $R$ and $S := R \cotimes{k} \ell$ have equal Sannai's sequences $b_e$?
\end{question}

\subsubsection{Rees algebras}
We explored a connection with Rees algebras in Corollary~\ref{cor Rees}, 
but it is likely that one can say more. 
For example, it was conjectured in \cite{HaraWatanabeYoshida}
and proved in \cite{KK} that for an $\mf m$-primary ideal $I$ 
the extended Rees algebra $R[It, t^{-1}]$ is F-rational if and only if $R$ and 
the Rees algebra $R[It]$ are F-rational. It is desirable to give a connection in terms 
of F-rational signature akin to Corollary~\ref{cor Rees}.

\bibliographystyle{alpha}
\bibliography{over}

\begin{appendix}
\counterwithin{theorem}{section}

\section{A criterion for simultaneous injection of vector spaces}

Throughout this section, we will work with vector spaces over a field $k$. For finite-dimensional vector spaces $V,W$ and subspaces $U \subseteq V$ and $H \subseteq \Hom_k(V,W)$, we denote by $H(U) := \sum_{h \in H} h(U)$ the total image of $U$ under $H$. We shall say that there are $n$ simultaneous injections from $V$ to $W$ in $H$ provided there exist $\phi_1, \ldots, \phi_n \in H$ with 
\begin{equation*}
    \dim( \sum_{i=1}^n \im \phi_i ) = n \dim(V)
\end{equation*}
or equivalently the induced map 
\begin{equation*}
    \Phi = (\phi_1, \ldots, \phi_n) \colon \bigoplus_{i=1}^n V \to W
\end{equation*}
is an injection. Given such an injection, for any $k$-vector subspace $U$ of $V$, we must have
\begin{equation} 
    \label{eq:easyUbound}
    n \cdot \dim(U) = \dim(\Phi(\bigoplus_{i=1}^n U)) = \dim(\sum_{i=1}^n \phi_i(U)) \leq \dim(H(U))
\end{equation}
and so $n \leq \min_{0 \neq U \subseteq V} \lfloor \dim(H(U)) / \dim(U) \rfloor $ where $U$ varies over all of the non-zero subspaces of $V$. The problem we seek to address here is the optimality of this upper bound, and the main technical result of this section is the following criterion.

\begin{theorem}
    \label{thm:linalgcriterion}
     Let $k$ be a field and  suppose $V$ is a finite dimensional vector space over $k$. Then, there exists a positive constant $C$ with the following property: for any finite dimensional vector space $W$ and vector subspace $H \subseteq \Hom_k(V,W)$ we have
     \begin{equation}
        \label{eq:linalgconstant}
         0 \leq \min_{0 \neq U \subseteq V} \left\lfloor \frac{\dim(H(U))}{\dim(U)} \right\rfloor - n \leq C
     \end{equation}
     where $U$ varies over all non-zero subspaces of $V$ and $n$ 
     is the maximal number of simultaneous injections from $V$ to $W$ in $H$. In other words, $n$ is the largest integer so that there exist $\phi_1, \ldots, \phi_n \in H$ where
     \begin{equation*}
         \Phi = (\phi_1, \ldots, \phi_n) \colon \bigoplus^n V \to W
     \end{equation*}
     is an injection.
\end{theorem}

\begin{remark}
    Our proof will show that the constant $C$ appearing \autoref{thm:linalgcriterion} can be taken to be 
    \begin{equation*}
        C = \sum_{i=1}^{\dim(V)-1} i (\dim(V)-i)\cdot \dim(V) = \frac{1}{6} \left(\dim(V)\right)^2 \left((\dim(V))^2 - 1\right)
    \end{equation*}
    independently of the ground field $k$. However, we believe this bound to be far from optimal in many cases. In particular, when working over an infinite field $k$ and using general $k$ linear combinations of maps in $H$ appropriately, we believe it is possible to exhibit a quadratic bound in terms of $\dim(V)$.
    Moreover, we also have the following important result when $\dim(V) = 2$.
\end{remark}

\begin{theorem}
    \label{thm:dimtwocriterion}
    Let $k$ be a field and suppose $V$ is a two-dimensional vector space over $k$. For any finite-dimensional vector space $W$ and vector subspace $H \subseteq \Hom_k(V,W)$, we have
    \begin{equation*}
        n = \min_{0 \neq U \subseteq V} \left\lfloor \frac{\dim(H(U))}{\dim(U)} \right\rfloor
    \end{equation*}
    where $U$ varies over all non-zero subspaces of $V$ and $n$ is the maximal number of simultaneous injections from $V$ to $W$ in $H$. In other words, $n$ is the largest integer so that there exist $\phi_1, \ldots, \phi_n \in H$ where
    \begin{equation*}
        \Phi = (\phi_1, \ldots, \phi_n) \colon \bigoplus^n V \to W
    \end{equation*}
    is an injection.
\end{theorem}

The following example shows that the assumption $\dim(V) = 2$ is essential in \autoref{thm:dimtwocriterion}.

\begin{example}
    Let $k$ be an arbitrary field and set $V = W = k^{\oplus 3}$ with standard basis vectors $\vec{e}_1, \vec{e}_2, \vec{e}_3$.
    Consider the linear transformations from $V$ to $W$ given by the matrices 
    \[
    f= \begin{bmatrix}
    1 & 0 & 0\\
    0 & 1 & 0\\
    0 & 0 & 0
    \end{bmatrix},
    g= \begin{bmatrix}
    0 & 0 & 0\\
    0 & 0 & -1\\
    1 & 0 & 0
    \end{bmatrix},
    h= \begin{bmatrix}
    0 & 0 & 1\\
    0 & 0 & 0\\
    0 & 1 & 0
    \end{bmatrix}
    \]
    and let $H$ be the linear span of $f,g,h$ in $\Hom_k(V,W)$.  
    Note that there are no injections in $H$ at all, since for all $\lambda, \mu, \nu \in k$ we have
    \begin{equation*}
        \det(\lambda f + \mu g + \nu h) = \det \begin{bmatrix}
            \lambda & 0 & \nu\\
            0 & \lambda & -\mu\\
            \mu & \nu & 0
            \end{bmatrix} = 0.
    \end{equation*}
    Nevertheless, we will show $ \min_{0 \neq U \subseteq V}\left\lfloor \dim(H(U))/\dim(U) \right\rfloor = 1$. Indeed, if $0 \neq \vec{v} = a \vec{e}_1 + b \vec{e}_2 + c \vec{e}_3$ so that $U = k\cdot \vec{v}$ has dimension one, then $H(U)$ is the column space of the singular matrix
    \begin{equation*}
        \begin{bmatrix}
            a & 0 & c \\
            b & -c & 0 \\
            0 & a & b
        \end{bmatrix}
    \end{equation*}
    whose two by two minors include among them $a^2$, $b^2$, and $c^2$ and are not all zero. Thus, if $U \subseteq V$ has dimension one, we have $\dim(H(U)) = \lfloor \dim(H(U)) / \dim(U) \rfloor = 2$. If instead $U \subseteq V$ has dimension two,\footnote{In fact, it is easily checked that $H(U) = W$ for every $U \subseteq V$ of dimension two.} taking $0 \neq \vec{v} \in V$ it follows that $\dim(H(U)) \geq \dim(H(k \cdot \vec{v})) = 2$ and so $\lfloor \dim(H(U)) / \dim(U) \rfloor = 1$. Finally, if $U = V$ has dimension three, then clearly  $\lfloor \dim(H(U)) / \dim(U) \rfloor = 1$ as $H(V)=W$. 
    \end{example}

    \begin{remark}
    If $k$ is infinite, then 
    in the setting of \autoref{thm:linalgcriterion}, one expects there to be smaller (in terms of $\dim V$) optimal constants $C$. This is because new injections can appear in $H$ after extending from a finite field. For example, over the field $F_2 = \mathbb{Z}/2\mathbb{Z}$, let $V = W = F_2^{\oplus 3}$ and consider the subspace
    \begin{equation*}
        H = \left \{
        A = \begin{bmatrix}
        1 & 0 & 0\\
        0 & 1 & 0\\
        0 & 0 & 0
        \end{bmatrix},
        B = \begin{bmatrix}
        1 & 0 & 0\\
        1 & 0 & 0\\
        0 & 0 & 1
        \end{bmatrix},
        A + B = \begin{bmatrix}
        0 & 0 & 0\\
        1 & 1 & 0\\
        0 & 0 & 1
        \end{bmatrix}
        \right \}
    \end{equation*}
    of $\Hom_{F_2}(V,W)$. While there is no injection in $H$, after extending to $F_4 = F_2[x]/(x^2 + x + 1)$ we see that
    \begin{equation*}
        xA + B = \begin{bmatrix}
            x+1 & 0 & 0\\
            1 & x & 0\\
            0 & 0 & 1
            \end{bmatrix}
    \end{equation*}
    is injective.
    \end{remark}

    Turning first towards a proof of \autoref{thm:dimtwocriterion}, we start with an elementary lemma.
    
    \begin{lemma}\label{finite field patching}
    Let $V,W$ be two finite dimensional vector spaces over a field $k$. 
    If $\phi, \xi \in \Hom_k (V, W)$ are such that $\im \phi \cap \im \xi = 0$,
    then $\rank (\phi + \xi) \geq \rank \phi$ 
    and the inequality is strict provided there exists $\vec{v} \in \ker \phi$ 
    such that $\xi(\vec{v}) \neq 0$. 
    \end{lemma}
    \begin{proof}
    If $\vec{v} \in \ker (\phi + \xi)$, then $\phi(\vec{v}) = -\xi({v}) \in \im \phi \cap \im \psi = 0$. Thus, we see that $\ker (\phi + \xi) = (\ker \phi \cap \ker \xi) \subseteq \ker \phi$ and the desired inequality $\rank(\phi+\xi) \geq \rank \phi$ follows and is strict provided $\ker \phi \cap \ker \xi \subsetneq \ker \phi$.
    \end{proof}

    \begin{proposition}
        \label{prop:twodiminjections}
        Let $k$ be a field, $V$ a two-dimensional $k$-vector space, and $n \geq 0$ is an integer. 
        Let $W$ be a finite-dimensional $k$-vector space and 
        $H$ a subspace of $\Hom (V,W)$ such that for every $0 \neq U \subseteq V$ we have
        $\dim H(U) \geq n\dim U$. Then there is an injection $\oplus^n V \to W$, 
        where each component is in $H$.
        \end{proposition}

{
\begin{proof}
    If $n = 0$ the statement holds trivially, and so we proceed by induction. So, we assume that $\dim H(U) \geq (n+1) \dim U$ for all $0 \neq U \subseteq V$ and  
    we are given independent injections $\phi_1, \ldots, \phi_n$. Let $W' = W / (\sum_{i=1}^n \im \phi_i)$, and for any $\phi \in \Hom_k(V,W)$ we shall denote by $\phi' \in \Hom_k(V,W')$ the map $V \xrightarrow{\phi} W \to W'$ induced by quotienting out by $\sum_{i=1}^n \im \phi_i$. Because $\dim H(V) \geq 2n + 2$, there is a $\psi \in H$ with $\rank \psi' > 0$. If $\rank \psi' = 2$ we are done, so assume $\rank \psi' = 1$. Using that $\dim H(V) \geq 2n + 2$, we can find $\xi \in H$ with $\im \xi' \not\subseteq \im \psi'$. Again, if $\rank \xi' = 2$ then we are done, so assume $\rank \xi' = 1$. It follows that $\im \psi' \cap \im \xi' = 0$. If $\ker \psi' \neq \ker \xi'$, then $\rank(\psi' + \xi') = 2$ using \autoref{finite field patching} and again we are done, so assume $K = \ker \psi' = \ker \xi'$ and take $0 \neq \vec{v} \in K$. Since $\dim K = 1$, we have $\dim H(K) \geq n + 1 > \dim (\sum_{i=1}^n \phi_i(K))$. Take $\mu \in H$ so that $\mu(\vec{v}) \not\in \sum_{i=1}^n \phi_i(K)$. 
    
    \setcounter{case}{0}
    \begin{case}
        $\mu(\vec{v}) \not\in \sum_{i=1}^n \im \phi_i$.
    \end{case}

If $\rank(\mu') = 2$, we are done, so assume $\rank(\mu') = 1$. Since $\im(\psi') \neq \im(\xi')$, it follows that at least one of $\im(\mu') \neq \im(\psi')$ or $\im(\mu') \neq \im(\xi')$ is satisfied. If $\im(\mu') \neq \im(\psi')$, then we have $\im(\mu') \cap \im(\psi') = 0$ and $\vec{v} \in \ker(\psi')\setminus \ker(\mu')$, so it follows $\psi' + \mu'$ has rank $2$ using \autoref{finite field patching}. Thus, we are done as $\phi_1, \ldots, \phi_n, \psi + \mu$ give the required $n+1$ simultaneous injections from $V$ to $W$ in $H$.
Similarly, in the other case where $\im(\mu') \neq \im(\xi')$, we have $\im(\mu') \cap \im(\xi') = 0$ and $\vec{v} \in \ker(\xi')\setminus \ker(\mu')$, so it follows using \autoref{finite field patching} that $\phi_1, \ldots, \phi_n, \xi + \mu$ give the required $n+1$ simultaneous injections from $V$ to $W$ in $H$.

    \begin{case}
        $\mu(\vec{v}) \in \sum_{i=1}^n \im \phi_i$.
    \end{case}
  
    Note first that in order for this case to occur, we must have $\sum_{i=1}^n \im \phi_i \neq 0$, which necessitates $n \geq 1$. Since $\mu(\vec{v}) \not\in \sum_{i=1}^n \phi_i(K)$, after reordering the $\phi_i$ we may arrange so that $\mu(\vec{v}) \not\in \phi_1(K) + \sum_{i=2}^n \im \phi_i$. Let $\overline{W} = W / (\sum_{i=2}^n \im \phi_i)$, and for any $\phi \in \Hom_k(V,W)$ we shall denote by $\overline{\phi} \in \Hom_k(V,\overline{W})$ the map $V \xrightarrow{\phi} W \to \overline{W}$ induced by quotienting out by $\sum_{i=2}^n \im \phi_i$. We have $\overline{\mu}(\vec{v}) \in \overline{\phi_1}(V) \setminus \overline{\phi_1}(K)$, so there is some $\vec{u} \in V \setminus K$ with $\overline{\mu}(\vec{v}) =\overline{\phi_1}(\vec{u})$. Note in particular that $\vec{u}, \vec{v}$ are a basis of $V$. Since $\im(\xi') \cap \im(\psi')=0$ and $\vec{u} \not\in K$, we have that $\xi'(\vec{u})$ and $\psi'(\vec{u})$ must be linearly independent. Thus, we see that $\overline{\phi_1}(\vec{u}),\overline{\phi_1}(\vec{v}), \overline{\psi}(\vec{u}), \overline{\xi}(\vec{u})$ are linearly independent in $\overline{W}$. Extending this latter set of vectors to a basis of $\overline{W}$, the matrices of these linear transformations with respect to these bases have the form 
    \begin{equation*}
       \overline{\phi_1} = \begin{bmatrix}
            1 & 0 \\
            0 & 1 \\
            0 & 0 \\
            0 & 0 \\
            \vdots & \vdots
        \end{bmatrix} \quad
        \overline{\psi} = \begin{bmatrix}
            0 & a \\
            0 & b \\
            1 & 0 \\
            0 & 0 \\
            \vdots & \vdots
        \end{bmatrix} \quad 
        \overline{\xi} = \begin{bmatrix}
            0 & c \\
            0 & d \\
            0 & 0 \\
            1 & 0 \\
            \vdots & \vdots
        \end{bmatrix}
        \quad 
        \overline{\mu} = \begin{bmatrix}
            e & 1 \\
            f & 0 \\
            g & 0 \\
            h & 0 \\
            \vdots & \vdots
        \end{bmatrix}
    \end{equation*}
    where $a,b,c,d,e,f,g,h \in k$. 
    
    With respect to the given basis of $\overline{W}$, $\im (\overline{\phi_1} + \overline{\xi}) + \im \overline{\psi}$ is identified with the column space of the matrix
    \begin{equation*}
    \begin{bmatrix}
    1 &  c & 0 & a\\
    0 & 1+ d & 0 & b\\
    0 & 0 & 1 & 0 \\ 1 & 0 & 0 & 0\\
    \vdots & \vdots & \vdots & \vdots
    \end{bmatrix}
    \end{equation*}
    whose top minor computes to $-a  + bc  -ad$. Thus, if $a \neq bc-ad$, then the assertion follows as $\phi_2, \dots, \phi_n,\phi_1 + \xi, \psi$ give the required $n+1$ simultaneous injections.  Similarly, if $c \neq ad - bc$, then $\phi_2, \dots, \phi_n,\phi_1 + \psi, \xi$ give the required $n+1$ simultaneous injections.
    So we now assume $a = bc-ad = -c$.

    With respect to the given basis of $\overline{W}$, $\im (\overline{\phi_1} + \overline{\psi}) + \im \overline{\mu}$ is identified with the column space of the matrix
    \begin{equation*}
    \begin{bmatrix}
    1 &  a & e & 1\\
    0 & 1+ b & f & 0\\
    1 & 0 & g & 0 \\ 0 & 0 & h & 0\\
    \vdots & \vdots & \vdots & \vdots
    \end{bmatrix}
    \end{equation*}
    whose top minor computes to $h + bh$. Thus, if $h+bh \neq 0$, then $\phi_2, \dots, \phi_n,\phi_1 + \psi, \mu$ give the required $n+1$ simultaneous injections. %
    Similarly, the assertion follows if $-g - dg\neq 0$, with $\phi_2, \dots, \phi_n,\phi_1 + \xi, \mu$ giving the required $n+1$ simultaneous injections.
    So we now assume $h+bh = 0 = g+dg$.

    With respect to the given basis of $\overline{W}$, $\im (\overline{\phi_1} + \overline{\xi}) + \im (\overline{\psi} + \overline{\mu})$ is identified with the column space of the matrix
    \begin{equation*}
    \begin{bmatrix}
    1 &  c & e & a+1\\
    0 & 1+ d & f & b\\
    0 & 0 & g+1 & 0 \\ 1 & 0 & h & 0\\
    \vdots & \vdots & \vdots & \vdots
    \end{bmatrix}
    \end{equation*}
    whose top minor computes to $(g+1)(bc-ad-a-d-1)$. Using the relations above, this simplifies to$-d-1$. Thus the assertion follows unless $d = -1$, since $\phi_2, \dots, \phi_n,\phi_1 + \xi, \psi + \mu$ give the required $n+1$ simultaneous injections. Similarly, if $b \neq -1$, we may use $\phi_2, \dots, \phi_n,\phi_1 + \psi, \xi + \mu$ as the required $n+1$ simultaneous injections. So we now assume $b = d = -1$, which also gives that $a = -c+a = -c$ and so $a = c = 0$.

    Now, with respect to the given basis of $\overline{W}$, $\im \overline{\psi} + \im \overline{\mu}$ is identified with the column space of the matrix
    \begin{equation*}
    \begin{bmatrix}
    0 &  0 & e & 1\\
    0 & -1 & f & 0\\
    1 & 0 & g & 0 \\ 0 & 0 & h & 0\\
    \vdots & \vdots & \vdots & \vdots
    \end{bmatrix}
    \end{equation*}
    whose top minor computes to $-h$. Thus, if $h \neq 0$ we are done as $\phi_2, \dots, \phi_n,\psi, \mu$ give the required $n+1$ simultaneous injections. Similarly, if $g \neq 0$ we are done as $\phi_2, \dots, \phi_n,\xi, \mu$ give the required $n+1$ simultaneous injections. Thus, we assume $g = h = 0$.

    Finally, we now have that $\im \overline{\psi} + \im (\overline{\mu}+\overline{\xi})$ is identified with the column space of the matrix
    \begin{equation*}
    \begin{bmatrix}
    0 &  0 & e & 1\\
    0 & -1 & f & -1\\
    1 & 0 & 0 & 0 \\ 0 & 0 & 1 & 0\\
    \vdots & \vdots & \vdots & \vdots
    \end{bmatrix}
    \end{equation*}
    whose top minor computes to $-1$. Thus, $\phi_2, \dots, \phi_n,\psi, \mu+\xi$ give the required $n+1$ simultaneous injections in this case and the exhaustive proof of the assertion is finished. 
\end{proof}
}

\begin{proof}[Proof of \autoref{thm:dimtwocriterion}]
    From \eqref{eq:easyUbound}, we have that $n \dim(U) \leq \dim H(U)$ for all subspaces $0 \neq U \subseteq V$ and hence $n \leq \min_{0 \neq U \subseteq V} \lfloor \frac{\dim H(U)}{\dim(U)}\rfloor$. Moreover, we must have $\dim H(U_0) < (n+1)\dim(U_0)$ for some $0 \neq U_0 \subseteq V$ by \autoref{prop:twodiminjections}. Altogether this gives  
    \begin{equation*}
       n \leq \min_{0 \neq U \subseteq V} \left\lfloor \frac{\dim(H(U))}{\dim(U)} \right\rfloor \leq    \frac{\dim(H(U_0))}{\dim(U_0)} \leq n
    \end{equation*}
    and so equality must hold throughout completing the proof.
\end{proof}

Our proof of \autoref{thm:linalgcriterion} follows roughly the same framework as the proof of \autoref{thm:dimtwocriterion} above, though the requisite inductive constructions that follow in \autoref{thm:finitefieldinductionstep} and \autoref{cor d induction step} are quite a bit more involved than that of \autoref{prop:twodiminjections}. Additionally, the elementary result below will also be used to avoid using general linear combinations over finite fields.

    \begin{lemma}\label{lemma disjoint}
    Let $U, W$ be vector spaces over a field $k$. 
    Suppose that $\phi_1, \ldots, \phi_N \in \Hom (U, W)$ are such that 
    $\Phi = (\phi_1, \ldots, \phi_N) \colon \bigoplus_{i=1}^N U \to W$ is an injection. 
    If $Z$ is a subspace of $W$ such that $\dim(Z \cap \im \Phi) = d \leq N$,
    then omitting some $d$ of the $\phi_1, \ldots, \phi_N$ will yield an injection $\bigoplus^{N-d} U \to W$ with image disjoint from $Z$. In other words, after reordering $\phi_1, \ldots, \phi_N$, one can ensure that $Z \cap (\sum_{i=d+1}^N \im \phi_i) = 0$.
    \end{lemma}
    \begin{proof}
    We proceed by induction on $d$, noting first that the lemma is a tautology when $d = 0$. Now, assume the statement holds for all $0 \leq n < d$ and we have
    an injection $\Phi = (\phi_1, \ldots, \phi_N) \colon \bigoplus_{i=1}^N U \to W$ and a subspace $Z \subseteq W$ with $\dim(Z \cap \sum_{i=1}^N \im \phi_i) = d \leq N$. Let $0 \neq \vec{v} \in (Z \cap \sum_{i=1}^N \im \phi_i)$. Let $W_j = \sum_{\substack{i=1\\i \neq j}}^N \im \phi_i$. Since $\bigcap_{i=1}^N W_j = 0$, it follows that $\vec{v} \not\in W_j$ for some $j$. After reordering, we may assume $\vec{v} \not\in W_1$. In particular, it follows that $\Phi' := (\phi_2, \ldots, \phi_N) \colon \bigoplus_{i=1}^N U \to W$ is an injection with $Z \cap \im \Phi' = Z \cap W_1 \subsetneq Z \cap \im \Phi$ so that $\dim(Z \cap \im \Phi') = n \leq d-1 \leq N-1$. Using the induction assumption on $\Phi'$ and $Z$, it follows that we can reorder $\phi_2, \ldots, \phi_N$ to achieve $0 =  \left( Z \cap (\sum_{i=n+2}^N \im \phi_i) \right) \supseteq \left( Z \cap (\sum_{i=d+1}^N \im \phi_i) \right)$ as desired.
    \end{proof}

    \begin{theorem}\label{thm:finitefieldinductionstep}
    Let $V$ and $W$ be finite dimensional vector spaces over a field $k$, and $H$ a subspace of $\Hom_k (V, W)$. Suppose $n \geq 0$ and $1 \leq d < \dim V$ are integers and assume the following conditions are satisfied.
    \begin{enumerate}
    \item There exist $\phi_1, \ldots, \phi_n \in H$
    giving an injection $(\phi_1, \ldots, \phi_n) \colon \bigoplus_{i=1}^n V \to W$.
    \item We have $\dim(H(U)) > n \dim(U)$ for any non-zero subspace $0 \neq U \subseteq V$. 
    \item Writing $m := (\dim(V) - d)\cdot \dim(V) + 1$, there exist $\psi_1, \ldots, \psi_m \in H$ so that 
    \begin{equation*}
        \dim(\sum_{i=1}^n \im \phi_i + \sum_{j=1}^{\ell} \im \psi_j) \geq d + \dim(\sum_{i=1}^n \im \phi_i + \sum_{j=1}^{\ell - 1} \im \psi_j)
    \end{equation*}
    for $\ell = 1, \ldots, m$. In other words, we have that each $\psi_\ell$ has rank at least $d$ modulo $\sum_{i=1}^n \phi_i + \sum_{j=1}^{\ell - 1} \im \psi_j$.
    \end{enumerate}
    Then there are maps $\tilde \phi_1, \ldots, \tilde \phi_n, \psi \in H$ so that $\tilde \Phi = (\tilde \phi_1, \ldots, \tilde \phi_n) \colon \bigoplus_{i=1}^nV \to W$ is an injection and $\psi$ has rank at least $d+1$ modulo $\im \tilde{\Phi}$, \textit{i.e.}
    \begin{equation*}
        \dim(\im \psi + \sum_{i=1}^n \im \tilde \phi_i) \geq d + 1 + \dim(\sum_{i=1}^n \im \tilde \phi_i).
    \end{equation*}
    \end{theorem}

    \begin{proof}
    In order to have the rank at least $d$ modulo $\sum_{i=1}^n \phi_i + \sum_{j=1}^{\ell - 1} \im \psi_j$, each $\psi_\ell$ must necessarily have rank at least $d$ modulo $\sum_{i=1}^n \phi_i$. The assertion follows trivially if any $\psi_\ell$ has rank at least $d+1$ modulo $\sum_{i=1}^n \phi_i$, so we may assume each $\psi_\ell$ has rank exactly $d$ modulo either of $\sum_{i=1}^n \phi_i$ or $\sum_{i=1}^n \phi_i + \sum_{j=1}^{\ell - 1} \im \psi_j$.

    Let $W' = W / (\sum_{i=1}^n \im \phi_i)$, and for any $\phi \in \Hom_k(V,W)$ we shall denote by $\phi' \in \Hom_k(V,W')$ the map $V \xrightarrow{\phi} W \to W'$ induced by quotienting out by $\sum_{i=1}^n \im \phi_i$. Since $\rank \psi_\ell' = d$ and does not change modulo $\sum_{j=1}^{\ell -1} \im \psi_j'$, we have $\im \psi_\ell' \cap \left(\sum_{j=1}^{\ell -1} \im \psi_j' \right) = 0$ for any $1 \leq \ell \leq m$ and in particular $\im \psi_i' \cap \im \psi_j' = 0$ for any $i \neq j$. If ever $\ker \psi_i' \neq \ker \psi_j'$ for some $i \neq j$, then $\rank (\psi_i' + \psi_j') > d$ by \autoref{finite field patching} and the assertion follows. Hence, assume now that all of these kernels are equal, and set $K = \ker \psi_\ell'$ for all $1 \leq \ell \leq m$. 
    
    Let $U$ be a vector space complement of $K$ in $V$, so that $\dim(U) = d = \dim(V) - \dim(K)$ and $V = U + K$ with $U \cap K = 0$. 
    Observe that $\im \psi_\ell' = \psi_\ell'(U)$ so the restriction of each $\psi_\ell'$ to $U$ is injective as $\rank \psi_\ell' = \dim(U) = d$. Moreover, setting $\Psi' = (\psi_1', \ldots,\psi_m') \colon \oplus_{j=1}^m V \to W'$, we similarly have that $\Psi'\vert_{\oplus_{j=1}^m U}$ is an injection as $\dim(\sum_{j=1}^m \psi_j'(U)) = dm$. In particular, it follows each $\psi_\ell'$ has rank $d$ modulo $\sum_{\substack{j=1 \\ j \neq \ell}}^m \im \psi_j'$, \textit{i.e.} $\im \psi_\ell \cap (\sum_{\substack{j=1 \\j\neq \ell}}^m \im \psi_j') = 0$. Note also we may permute $\psi_1, \ldots, \psi_m$ as needed below while preserving our setup. 
    
    

    Suppose first that $H(K) \not\subseteq \sum_{i=1}^n \im \phi_i$, and take $h \in H$ and $\vec{v} \in K$ with $h(\vec{v}) \not\in \sum_{i=1}^n \im \phi_i$ so that $h'(\vec{v}) \neq 0$. 
    Since $\im \psi_\ell$ are disjoint from each other, 
    at most $\rank h' \leq \dim(V) < m$ of the $\im \psi_\ell'$ can intersect $\im h'$ non-trivially, so after reordering we may assume $\im \psi_1' \cap \im h' = 0$ (as in the proof of \autoref{lemma disjoint}). By \autoref{finite field patching}, $\rank (\psi_1' + h') > \rank \psi_1' = d$ and the assertion follows. Thus, we may assume going forward that $H(K) \subseteq \sum_{i=1}^n \im \phi_i$. In particular, observe that this implies $\sum_{i=1}^n \im \phi_i \neq 0$ so we must have $n \geq 1$.
    We break the rest of the proof up into two further cases to be analyzed separately.

    \setcounter{case}{0}
    \begin{case}
        There is some $1 \leq \ell \leq m$ with $\psi_\ell(K) \not\subseteq \left( \sum_{i=1}^n \phi_i(K)\right)$.
    \end{case}

    Because $\psi_\ell(K) \subseteq \sum_{i=1}^n \im \phi_i \cong \bigoplus_{i=1}^n V$ and an element of $\bigoplus_{i=1}^n V$ is in $\bigoplus_{i=1}^n K$ if and only if each of the components is in $K$, we may reorder the $\phi_i$ so that $\psi_\ell(K) \not\subseteq \left(\phi_1(K) + \sum_{i=2}^n \im \phi_i\right)$ and choose $\vec{v} \in K$ with $\psi_\ell(\vec{v}) \not\in \left(\phi_1(K) + \sum_{i=2}^n \im \phi_i\right)$. Let $\overline{W} = W / (\sum_{i=2}^n \im \phi_i)$, and for any $\phi \in \Hom_k(V,W)$ we shall denote by $\overline{\phi} \in \Hom_k(V,\overline{W})$ the map $V \xrightarrow{\phi} W \to \overline{W}$ induced by quotienting out by $\sum_{i=2}^n \im \phi_i$. Since $\dim(\Hom_k(K,\im \overline{\phi_1})) = \dim(K)\cdot \dim(V) = (\dim(V)-d)\cdot \dim(V) < m$, there must be a nontrivial linear combination $\xi = \sum_{i=1}^m \alpha_i \psi_i$ with $\alpha_1, \ldots, \alpha_m \in k$ not all zero and $K \subseteq \ker \overline{\xi}$. 
    If $0 \neq \vec{u} \in U$, then $\xi'(\vec{u}) = \Psi'(\alpha_1 \vec{u}, \ldots, \alpha_m \vec{u}) \neq 0$ as $\Psi'\vert_{\oplus_{j=1}^m U}$ is an injection and $\alpha_j \neq 0$ for some $j$. Thus, $\overline{\xi}(\vec{u}) \neq 0$ as well, so it follows $\overline{\xi}|_U$ is an injection, $\rank \overline{\xi} = d$, and $K = \ker \overline{\xi}$. Moreover, since $\ker \overline{\psi_\ell} \neq K$, $\xi$ cannot be a scalar multiple of $\psi_\ell$ and we must have that $\alpha_j \neq 0$ for some $j \neq \ell$.
    
    We will now show that $\tilde\phi_1 := \phi_1 + \xi, \tilde\phi_2 := \phi_2, \ldots, \tilde\phi_n := \phi_n$ and $\psi := \psi_\ell$ are the required maps. Let us first check that $\im \overline{\xi} \cap (\im \overline{\phi_1} + \im \overline{\psi_\ell}) = 0$. Since $\im \overline{\xi} = \overline{\xi}(U)$, suppose $\vec{u} \in U$ and $\overline\xi(\vec{u})\in (\im \overline{\phi_1} + \im \overline{\psi_\ell})$. It follows that $\xi'(\vec{u}) \in \im  \psi_\ell' = \psi_\ell'(U)$, so say $\xi'(\vec{u}) = \psi_\ell'(\vec{w})$ for some $\vec{w} \in U$. We have $0 = \xi'(\vec{u}) - \psi_\ell'(\vec{w}) = \Psi'(\alpha_1 \vec{u}, \ldots,\alpha_{\ell -1} \vec{u}, \alpha_\ell \vec{u} - \vec{w}, \alpha_{\ell+1}\vec{u}, \ldots,\alpha_m \vec{u})$, which implies $\alpha_j \vec{u} = 0$ for all $j \neq \ell$ by the injectivity of $\Psi'\vert_{\oplus_{j=1}^m U}$. As $\alpha_j \neq 0$ for some $j \neq \ell$, this implies $\vec{u} = 0$. Thus, we conclude $\im \overline{\xi} \cap (\im \overline{\phi_1} + \im \overline{\psi_\ell}) = 0$. In particular, we have $\im \overline{\xi} \cap \im \overline{\phi_1}$ giving that $\overline{\phi_1}+\overline{\xi}$ is injective by applying \autoref{finite field patching} and using that $\overline{\phi_1}$ is injective.

    It remains to show that $\overline{\psi_\ell}$ has rank at least $d+1$ modulo $\im(\overline{\phi_1} + \overline{\xi})$. To that end, let us first check that $\overline{\psi_\ell}$ has rank at least $d+1$ modulo $\overline{\phi_1}(K)$.
    By our choice of $\vec{v} \in K$ above, we have that $\overline{\psi_\ell}(\vec{v}) \in \im \overline{\phi_1}\setminus \overline{\phi_1}(K)$. Put $T = U + k  \vec{v}$, which has dimension $d + 1$.  Suppose we have $\vec{u} \in U$ and $\lambda \in k$ with $\overline{\psi_\ell}(\vec{u} + \lambda \vec{v}) \in \overline{\phi_1}(K)$. It follows that $\overline{\psi_\ell}(\vec{u}) \in \im \overline{\phi_1}$ and so also $\psi_\ell'(\vec{u}) = 0$ which gives $\vec{u} = 0$ as $\ker \psi_\ell' = K$.  Thus, $\overline{\psi_\ell}(\lambda \vec{v}) = \lambda \overline{\psi_\ell}(\vec{v}) \in \overline{\phi_1}(K) $ which yields $\lambda = 0$ as $\overline{\psi_\ell}(\vec{v}) \not\in \overline{\phi_1}(K)$. It follows that $\overline{\psi_\ell}|_{T}$ is injective modulo $\overline{\phi_1}(K)$, \textit{i.e.} $\overline{\psi_\ell}$ is injective on $T$ and $\overline{\psi_\ell}(T) \cap \overline{\phi_1}(K) = 0$. To conclude the stronger statement that $\overline{\psi_\ell}$ has rank at least $d+1$ modulo $\im(\overline{\phi_1} + \overline{\xi})$, it suffices verify $\im(\overline{\phi_1} + \overline{\xi})  \cap \overline{\psi_\ell}(T) = 0$. Suppose we have some $\vec{w} \in V$ with $\overline{\phi_1}(\vec{w}) + \overline{\xi}(\vec{w}) \in \overline{\psi_\ell}(T)$. It follows that $\overline{\xi}(\vec{w}) \in \im \overline{\xi} \cap (\im \overline{\phi_1} + \im \overline{\psi_\ell}) = 0$ and $\vec{w} \in \ker \overline{\xi} = K$. Thus we must have $\overline{\phi_1}(\vec{w}) \in \overline{\psi_\ell}(T) \cap \overline{\phi_1}(K) = 0$, so that $\overline{\phi_1}(\vec{w}) + \overline{\xi}(\vec{w}) = 0$ and hence $\im(\overline{\phi_1} + \overline{\xi})  \cap \overline{\psi_\ell}(T) = 0$.

    \begin{case}
        For all $1 \leq \ell \leq m$, $\psi_\ell(K) \subseteq \left( \sum_{i=1}^n \phi_i(K)\right)$.
    \end{case}

    Since $\dim(H(K)) > n \dim(K)$, there is some $h \in H$ with $h(K) \not\subseteq \sum_{i=1}^n \phi_i(K)$. Because $h(K) \subseteq H(K) \subseteq \sum_{i=1}^n \im \phi_i \cong \bigoplus_{i=1}^n V$ and an element of $\bigoplus_{i=1}^n V$ is in $\bigoplus_{i=1}^n K$ if and only if each of the components is in $K$, we may reorder the $\phi_i$ so that $h(K) \not\subseteq \left(\phi_1(K) + \sum_{i=2}^n \im \phi_i\right)$ and choose $\vec{v} \in K$ with $h(\vec{v}) \not\in \left(\phi_1(K) + \sum_{i=2}^n \im \phi_i\right)$. Let $\overline{W} = W / (\sum_{i=2}^n \im \phi_i)$, and for any $\phi \in \Hom_k(V,W)$ we shall denote by $\overline{\phi} \in \Hom_k(V,\overline{W})$ the map $V \xrightarrow{\phi} W \to \overline{W}$ induced by quotienting out by $\sum_{i=2}^n \im \phi_i$. Set $Z = \overline{h}(U) + \im \overline{\phi_1}$.

    Since $\dim \left( \sum_{j=1}^m \overline{\psi_j}(U) \right) = \dim(\sum_{j=1}^m  \psi_j'(U)) = md$, we see that $\dim \left( \sum_{j=1}^m \overline{\psi_j}(U) \right)$ remains unchanged modulo $\im \overline{\phi_1}$ and thus $\left( \sum_{j=1}^m \overline{\psi_j}(U) \right) \cap \im \overline{\phi_1} = 0$. It follows that
    \begin{equation*}
        Z \cap \left( \sum_{j=1}^m \overline{\psi_j}(U) \right) \cong \frac{Z \cap \left( \sum_{j=1}^m \overline{\psi_j}(U) \right)}{Z \cap \left( \sum_{j=1}^m \overline{\psi_j}(U) \right)\cap \im \overline{\phi_1}} \subseteq \frac{Z}{\im \overline{\phi_1}} = \frac{\overline{h}(U) + \im \overline{\phi_1}}{\im \overline{\phi_1}} = h'(U)
    \end{equation*}
    and in particular $\dim\left(Z \cap \left( \sum_{j=1}^m \overline{\psi_j}(U) \right)\right) \leq \dim(h'(U)) \leq \dim(U)=d$. Applying \autoref{lemma disjoint} to $\overline{\psi_1}|_U, \ldots, \overline{\psi_m}|_U \in \Hom(U,\overline{W})$ and $Z$, it follows that after reordering $\psi_1, \ldots, \psi_m$ we may assume  
    \begin{equation}
        \label{eq:case2reduction}
        Z \cap \left( \sum_{j=d+1}^m \overline{\psi_j}(U) \right) = \Big(\overline{h}(U) + \im \overline{\phi_1} \Big) \cap \left( \sum_{j=d+1}^m \overline{\psi_j}(U) \right) = 0.
    \end{equation}

    Since $\overline{\psi_{d+1}}|_K, \ldots, \overline{\psi_m}|_K \in \Hom_k(K,\overline{\phi_1}(K))$ and $\dim(\Hom_k(K,\overline{\phi_1}(K))) = (\dim(V)-d)^2 < m - d$, there must be a nontrivial linear combination $\xi = \sum_{i=d+1}^m \alpha_i \psi_i$ with $\alpha_{d+1}, \ldots, \alpha_m \in k$ not all zero and $K \subseteq \ker \overline{\xi}$. 
    If $0 \neq \vec{u} \in U$, then $\xi'(\vec{u}) = \Psi'(0, \ldots, 0, \alpha_{d+1} \vec{u}, \ldots, \alpha_m \vec{u}) \neq 0$ as $\Psi'\vert_{\oplus_{j=1}^m U}$ is an injection and $\alpha_j \neq 0$ for some $d+1 \leq j \leq m$. Thus, $\overline{\xi}(\vec{u}) \neq 0$ as well, so it follows $\overline{\xi}|_U$ is an injection, $\rank \overline{\xi} = d$, and $K = \ker \overline{\xi}$. 
    Since $m \geq \dim(V)+1 \geq d +2$, we may choose $\psi_\ell$ with $d+1 \leq \ell \leq m$ so that $\alpha_j \neq 0$ for some $d+1 \leq j \leq m$ and $j \neq \ell$. We know $\overline{\psi_\ell}|_U$ is injective as $\psi_\ell'|_U$ is injective, and it follows from \eqref{eq:case2reduction} that $\overline{\psi_\ell}(U) \cap \overline{h}(U) = 0$ so that $\overline{\psi_\ell} + \overline{h}$ restricts to an injection on $U$ by \autoref{finite field patching}. Furthermore, let us argue that 
    \begin{equation}
        \label{eq:case2intersection}
        \overline{\xi}(U) \cap \left(\overline{\psi_\ell}(U) + Z \right) = \overline{\xi}(U) \cap \left(\overline{\psi_\ell}(U) + \overline{h}(U) + \im \overline{\phi_1} \right) = 0.
    \end{equation}
    Suppose $\vec{u}, \vec{u}', \vec{u}'' \in U$ and $\vec{w} \in V$ with $\overline{\xi}(\vec{u}) = \overline{\psi_\ell}(\vec{u}') + \overline{h}(\vec{u}'') + \overline{\phi_1}(\vec{w})$. Then $- \overline{\psi_\ell}(\vec{u}') + \sum_{i=d+1}^m \alpha_i \overline{\psi_i}(\vec{u}) \in Z \cap \left( \sum_{j=d+1}^m \overline{\psi_j}(U) \right) = 0$ by \eqref{eq:case2reduction}, which implies 
    \begin{equation*}
        \Psi'(0, \ldots, 0, \alpha_{d+1}\vec{u}, \ldots, \alpha_{\ell-1}\vec{u},\alpha_{\ell}\vec{u}-\vec{u}',\alpha_{\ell+1}\vec{u}, \ldots ,\alpha_{m}\vec{u}) = 0
    \end{equation*}
    giving that $\alpha_j \vec{u} = 0$ by the injectivity of $\Psi'\vert_{\oplus_{j=1}^m U}$ as $j \neq \ell$, thus we must have $\vec{u} = 0$ and \eqref{eq:case2intersection} follows.

    We will now show that $\tilde\phi_1 := \phi_1 + \xi, \tilde\phi_2 := \phi_2, \ldots, \tilde\phi_n := \phi_n$ and $\psi := \psi_\ell + h$ are the required maps. Notice first that since $\im \overline{\xi} \cap \im \overline{\phi_1}= \overline{\xi}(U)\cap \im \overline{\phi_1} = 0$ by \eqref{eq:case2intersection} and $\overline{\phi_1}$ is injective, \autoref{finite field patching} implies that $\overline{\phi_1} + \overline{\xi}$ remains injective.
    To finish, we need to show that $\overline{\psi_\ell} + \overline{h}$ has rank at least $d + 1$ modulo $\im(\overline{\phi_1}+\overline{\xi})$. By our choice of $\vec{v} \in K$ above, we have that $\overline{h}(\vec{v}) \in \im \overline{\phi_1}\setminus \overline{\phi_1}(K)$. Put $T = U + k  \vec{v}$, which has dimension $d + 1$. Suppose we have $\vec{u} \in U$ and $\lambda \in K$ with $(\overline{\psi_\ell} + \overline{h})(\vec{u}+\lambda \vec{v}) \in \im(\overline{\phi_1}+\overline{\xi})$. As $V = U + K$, suppose $\vec{u}' \in U$ and $\vec{w} \in K$ with $(\overline{\psi_\ell} + \overline{h})(\vec{u}+\lambda \vec{v}) = (\overline{\phi_1}+\overline{\xi})(\vec{u}' + \vec{w})$. Then, as $\vec{v},\vec{w} \in K = \ker \overline{\xi}$ and $H(K)\subseteq \sum_{i=1}^n \im \phi_i$, we see
    \begin{equation*}
        \overline{\xi}(\vec{u}') =  \overline{\psi_\ell} (\vec{u})  + \overline{h}(\vec{u}) + \overline{\psi_\ell} (\lambda \vec{v}) +\overline{h}(\lambda \vec{v})- \overline{\phi_1}(\vec{u}' + \vec{w}) \in \left(\overline{\psi_\ell}(U) + \overline{h}(U) + \im \overline{\phi_1} \right)
    \end{equation*}
    and it follows that $\overline{\xi}(\vec{u}') = 0$ by \eqref{eq:case2intersection} and also $\vec{u}' = 0$ because $\overline{\xi}|_U$ is an injection. Rearranging once again, we have 
    \begin{equation*}
        \overline{\psi_\ell}(\vec{u}) = -\overline{h}(\vec{u})-\overline{\psi_\ell} (\lambda \vec{v}) -\overline{h}(\lambda \vec{v})+ \overline{\phi_1}(\vec{w}) \in \overline{h}(U) + \im \overline{\phi_1} = Z
    \end{equation*}
    and it follows that $\overline{\psi_\ell}(\vec{u}) = 0$ by \eqref{eq:case2reduction} and also $\vec{u} = 0$ because $\overline{\psi_\ell}|_U$ is an injection. Using that $\vec{w} \in K$ and $\psi_\ell(K) \subseteq (\sum_{i=1}^n \phi_i(K))$, this leaves
    \begin{equation*}
        \overline{h}(\lambda \vec{v})  = \overline{\phi_1}(\vec{w}) -\overline{\psi_\ell} (\lambda \vec{v}) \in \overline{\phi_1}(K)
    \end{equation*}
    which is only possible if $\lambda = 0$ as $h(\vec{v})\not\in \overline{\phi_1}(K)$. Putting all of this together, we conclude that given $\vec{u} \in U$ and $\lambda \in k$, we have $(\overline{\psi_\ell} + \overline{h})(\vec{u}+\lambda \vec{v}) \in \im(\overline{\phi_1}+\overline{\xi})$ only when $\vec{u} = 0$ and $\lambda = 0$. It follows that $\overline{\psi_\ell} + \overline{h}$ restricts to an injection on $T$ which persists modulo $\im(\overline{\phi_1}+\overline{\xi})$, which concludes the proof.
    \end{proof}

    \begin{corollary}\label{cor d induction step}
    Let $V$ and $W$ be finite dimensional vector spaces over a field $k$, and $H$ a subspace of $\Hom_k (V, W)$. 
    Suppose $n \geq 0$ and $1 \leq d \leq \dim V$ are integers and assume the following conditions are satisfied.
    \begin{enumerate}
    \item There exist $\phi_1, \ldots, \phi_n \in H$
    giving an injection $(\phi_1, \ldots, \phi_n) \colon \bigoplus_{i=1}^n V \to W$.
    \item We have 
    \begin{equation*}
        \dim H(U) \geq n \dim U + 1 + \sum_{i = 1}^{d-1} i(\dim V - i)\dim V .
    \end{equation*} for any non-zero subspace $0 \neq U \subseteq V$. 
    \end{enumerate}
    Then there are maps $\tilde \phi_1, \ldots, \tilde \phi_n, \psi \in H$ so that $\tilde \Phi = (\tilde \phi_1, \ldots, \tilde \phi_n) \colon \bigoplus_{i=1}^nV \to W$ is an injection and $\psi$ has rank at least $d$ modulo $\im \tilde{\Phi}$, \textit{i.e.}
    \begin{equation*}
        \dim(\im \psi + \sum_{i=1}^n \tilde \phi_i) \geq d + \dim(\sum_{i=1}^n \tilde \phi_i).
    \end{equation*}
    
\end{corollary}

\begin{proof}
If $d = 1$, we have $\dim(H(V)) \geq n \dim V + 1 = n \dim\left(\sum_{i=1}^n \im \phi_i \right)$. Taking $\psi \in H$ with $\psi(V) \not\subseteq \sum_{i=1}^n \im \phi_i$, we see that $\tilde{\phi_1} = \phi_1, \ldots, \tilde{\phi_n} = \phi_n, \psi \in H$ give a suitable collection of maps. Proceeding inductively, assume now the conclusion holds for some $d \geq 1$. 
Suppose we have finite dimensional vector spaces $V,W$ with $\dim(V) \leq d$ admitting $n$ simultaneous injections from $V$ to $W$ in $H \subseteq \Hom_k(V,W)$  and so that 
\begin{equation}
    \label{eq:Uboundininductionstep}
    \dim(H(U)) \geq n \dim(U) + 1 + \sum_{i=1}^{d} i(\dim V - i)\dim V  
\end{equation}
for any non-zero subspace $0 \neq U \subseteq V$. We need to show that there is a map in $H$ with rank at least $d+1$ modulo the image of some $n$ potentially different simultaneous injections from $V$ to $W$ in $H$.

By our induction assumption, there are $\phi_1^{(1)}, \ldots, \phi_n^{(1)},\psi_1 \in H$ with 
\begin{equation*}
        \dim(\sum_{i=1}^n \im \phi_i^{(1)} + \im \psi_1) \geq \dim(\sum_{i=1}^n \im \phi_i^{(1)}) + d =  n \dim(V) + d.
\end{equation*}
 We proceed to define $\psi_1, \ldots, \psi_\ell \in H$ and $\phi_1^{(\ell)},\ldots, \phi_n^{(\ell)} \in H$ recursively until either the desired conlusion is satisfied or we reach $\ell = (\dim V -d)\dim V + 1$. Assume we have $\psi_1, \ldots, \psi_{\ell-1} \in H$ and $\phi_1^{(\ell-1)},\ldots, \phi_n^{(\ell-1)} \in H$ so that $\dim(\sum_{i=1}^n \im \phi_i^{(\ell-1)}) = n \dim(V)$ and
\begin{equation*}
    \dim(\sum_{i=1}^n \im \phi_i^{(\ell-1)} + \sum_{j=1}^{\ell'} \im \psi_j) \geq \dim(\sum_{i=1}^n \im \phi_i^{(\ell-1)}+  \sum_{j=1}^{\ell'-1} \im \psi_j)) + d
\end{equation*}
for all $1 \leq \ell' < \ell$. In particular, we also have 
\begin{equation*}
    \dim(\sum_{i=1}^n \im \phi_i^{(\ell-1)} + \im \psi_{\ell'} ) \geq \dim(\sum_{i=1}^n \im \phi_i^{(\ell-1)}) + d
\end{equation*}
for any $1 \leq \ell' < \ell$. If this inequality is ever strict we are done as $\tilde{\phi_1} = \phi_1^{(\ell-1)}, \ldots, \tilde{\phi_n} = \phi_n^{(\ell-1)}, \psi = \psi_{\ell'} \in H$ give the desired maps, so we shall assume we have equality for all $1 \leq \ell' < \ell$. Moreover, if we have
\begin{equation*}
    (\psi_{\ell'}^{(\ell-1)})^{-1}(\sum_{i=1}^n \im \phi_i^{(\ell-1)}) \neq (\psi_{\ell''}^{(\ell-1)})^{-1}(\sum_{i=1}^n \im \phi_i^{(\ell-1)})
\end{equation*}
for some $1 \leq \ell'' < \ell' < \ell$, \autoref{finite field patching} gives that the rank of $\psi_{\ell'} + \psi_{\ell''}$ is at least $d + 1$ modulo $\sum_{i=1}^n \im \phi_i^{(\ell-1)}$ and again we are done with $\tilde{\phi_1} = \phi_1^{(\ell-1)}, \ldots, \tilde{\phi_n} = \phi_n^{(\ell-1)}, \psi = \psi_{\ell'}+ \psi_{\ell''} \in H$ giving the desired maps. Thus, we may assume $K_\ell = (\psi_{\ell'}^{(\ell-1)})^{-1}(\sum_{i=1}^n \im \phi_i^{(\ell-1)})$ is independent of $1 \leq \ell' < \ell$. Picking $U_\ell$ to be a complement of $K_\ell$, we have that $U_\ell$ has dimension $d$. Let $\overline{W} = W / (\sum_{j=1}^{\ell - 1} \psi_j(U_\ell))$, and for any $\phi \in \Hom_k(V,W)$ we shall denote by $\overline{\phi} \in \Hom_k(V,\overline{W})$ the map $V \xrightarrow{\phi} W \to \overline{W}$ induced by quotienting out by $\sum_{j=1}^{\ell - 1} \psi_j(U_\ell)$. Note that $\psi_{\ell'}(U_\ell) = \im\psi_{\ell'}$ modulo either $\sum_{i=1}^n \im \phi_i^{(\ell-1)}$ or $\sum_{i=1}^n \im \phi_i^{(\ell-1)}+  \sum_{j=1}^{\ell'-1} \im \psi_j$ for$1 \leq \ell' < \ell$, so that  $\dim(\sum_{j=1}^{\ell - 1} \psi_j(U_\ell)) = (\ell - 1)d$. Consider also $\overline{H} = \{\overline{\phi} \mid \phi \in H\} \subseteq \Hom_k(V,\overline{W})$. We have that $\overline{\phi_1^{(\ell-1)}},\ldots,\overline{\phi_n^{(\ell-1)}}$ give $n$ simultaneous injections from $V$ to $\overline{W}$ in $\overline{H}$, and
for any subspace $0 \neq U \subseteq V$ we compute
\begin{equation*}\begin{array}[]{rcl}
    \dim \overline{H}(U) & \geq & \dim H(U) - \dim(\sum_{i=1}^{\ell - 1} \psi_i(U_\ell)) \\
    & \geq& n \dim(U) + 1 + \sum_{i=1}^{d} i(\dim V - i)\dim V   - (\ell-1)d \\
    & \geq & n \dim(U) + 1 + \sum_{i=1}^{d-1} i(\dim V - i)\dim V + ((\dim V - d)\dim V + 1 - \ell)d \\
    & \geq &  n \dim(U) + 1 + \sum_{i=1}^{d-1} i(\dim V - i)\dim V 
\end{array}
\end{equation*}
since $\ell \leq (\dim V - d)\dim V + 1$. Thus, by our indiction assupmtion, there are maps $\phi_1^{(\ell)}, \ldots, \phi_n^{(\ell)}, \psi_\ell \in H$ so that $\dim(\sum_{i=1}^n \im \overline{\phi_i^{(\ell)}} + \im \overline{\psi_\ell}) \geq \dim(\sum_{i=1}^n \im \overline{\phi_i^{(\ell)}}) + d$ and $\dim(\sum_{i=1}^n \im \overline{\phi_i^{(\ell)}}) = n \dim(V)$. In particular, it follows that $\dim(\sum_{i=1}^n \im \phi_i^{(\ell)}) = n \dim(V)$. If the rank of any $\psi_{\ell'}$ modulo $\sum_{i=1}^n \im \phi_i^{(\ell)}$ is at least $d + 1$ for some $1 \leq \ell' < \ell$ we are again done, so we may assume this rank is at most $d$. As $\dim(\sum_{i=1}^n \im \phi_i^{(\ell)})$ does not change modulo $\sum_{j=1}^{\ell-1} \psi_j(U_\ell)$, we similarly must have that $\dim(\sum_{j=1}^{\ell-1} \psi_j(U_\ell)) = (l-1)d$ does not change modulo $\sum_{i=1}^n \im \phi_i^{(\ell)}$. Thus, $(\psi_1,\ldots, \psi_{\ell-1}): \bigoplus_{j=1}^{\ell-1}U_\ell \to W$ is injective and remains so after going modulo $\sum_{i=1}^n \im \phi_i^{(\ell)}$. Moreover, we must have that $\psi_{\ell'}(U_\ell) = \im\psi_{\ell'}$ modulo either $\sum_{i=1}^n \im \phi_i^{(\ell)}$ or $\sum_{i=1}^n \im \phi_i^{(\ell)}+  \sum_{j=1}^{\ell'-1} \im \psi_j$ for $1 \leq \ell' < \ell$, and also 
\begin{equation*}
    \begin{array}{rcl}
        \dim(\sum_{i=1}^n \im \phi_i^{(\ell)} + \sum_{j=1}^{\ell} \im \psi_j) & = & \dim(\sum_{i=1}^n \im \phi_i^{(\ell)} + \sum_{j=1}^{\ell-1}  \psi_j(U_\ell) + \im \psi_\ell)\\
        & = &  \dim(\sum_{i=1}^n \im \overline{\phi_i^{(\ell)}}+  \im \overline{\psi_\ell})  \\
        & \geq & \dim(\sum_{i=1}^n \im \overline{\phi_i^{(\ell)}}) + d \\ 
        & = &  \dim(\sum_{i=1}^n \im \phi_i^{(\ell)}+  \sum_{j=1}^{\ell-1}  \psi_j(U_\ell)) + d \\ 
        & = &  \dim(\sum_{i=1}^n \im \phi_i^{(\ell)}+  \sum_{j=1}^{\ell-1} \im \psi_j) + d.
    \end{array}   
\end{equation*}
This completes our recursive construction, as it now follows $\dim(\sum_{i=1}^n \im \phi_i^{(\ell)}) = n \dim(V)$ and
\begin{equation*}
    \dim(\sum_{i=1}^n \im \phi_i^{(\ell)} + \sum_{j=1}^{\ell'} \im \psi_j) \geq \dim(\sum_{i=1}^n \im \phi_i^{(\ell)}+  \sum_{j=1}^{\ell'} \im \psi_j)) + d
\end{equation*}
for all $1 \leq \ell' \leq \ell$. 

To finish the proof, we need only address the remaining case where the recursion above proceeded all the way to $\ell - (\dim(V)-d)\dim(V) + 1$. However, the desired conclusion now follows from \autoref{thm:finitefieldinductionstep}, as conditions (a) and (c) from \autoref{thm:finitefieldinductionstep} are satisfied by $\phi_1^{{\ell}}, \ldots, \phi_n^{(\ell)}$ and $\psi_1, \ldots, \psi_\ell$, and condition (b) is immediate from \eqref{eq:Uboundininductionstep}.
    %
    \end{proof}

    \begin{corollary}\label{cor finite injections exist}
    Let $k$ be an arbitrary field, $V, W$ be finite-dimensional vector spaces over $k$, and $H$ be a subspace of $\Hom_k (V, W)$.
    Suppose that for some $n \geq 0$
    and any $0 \neq U \subseteq V$ we have
    \begin{equation*}
        \dim H(U) \geq (n-1) \dim U + 1 + \sum_{i = 1}^{\dim V - 1} i(\dim V - i)\dim V .
    \end{equation*}
    Then there is an injection $\oplus^{n} V \to W$ where all components are in $H$.
    \end{corollary}
    \begin{proof}
    We use induction on $n$, noting first that the base case $n = 0$ is trivially satisifed. Assume now that the statement holds for some $n \geq 0$ and we have finite dimensional $k$-vector spaces $V,W$ and $H \subseteq \Hom_k(V,W)$ with
    \begin{equation*}
            \dim H(U) \geq \displaystyle n\dim U + 1 + \sum_{i = 1}^{\dim V - 1} i(\dim V - i)\dim V.
    \end{equation*} 
    The induction hypothesis implies there exist $\phi_1, \ldots, \phi_n \in H$ so that $(\phi_1, \ldots, \phi_n) \colon \bigoplus_{i=1}^n V \to W$ is an injection. Applying
    Corollary~\ref{cor d induction step}, we have $\psi \in H$ with full rank modulo simultaneous injections $\tilde{\phi}_1, \ldots, \tilde{\phi}_n \in H$, giving an injection $(\tilde{\phi}_1, \ldots, \tilde{\phi}_n, \psi) \colon \bigoplus^{n+1} V \to W$ with all components in $H$ as desired.
    \end{proof}

    \begin{proof}[Proof of \autoref{thm:linalgcriterion}]
        Combining \eqref{eq:easyUbound} and \autoref{cor finite injections exist}, we must have 
        \begin{equation*}
            n \dim(U) \leq \dim H(U) \leq n \dim(U) + \sum_{i = 1}^{\dim V - 1} i(\dim V - i)\dim V
        \end{equation*}
        for all subspaces $0 \neq U \subseteq V$.
        Dividing through by $\dim(U)$, it follows the constant $C := \sum_{i = 1}^{\dim V - 1} i(\dim V - i)\dim V = \frac{1}{6} \left(\dim(V)\right)^2 \left((\dim(V))^2 - 1\right)$ satisfies \eqref{eq:linalgconstant}.
    \end{proof}

To conclude this section, we exhibit a dual formulation of the above results that is tailored towards our desired applications in later sections. 

\begin{corollary}\label{cor:linalgsurjections}
    The polynomial $P(T) = \frac{1}{6}T^2(T^2-1) \in \mathbb Q[T]$ is an increasing function on positive integers 
    with the following property: for any integer $n \geq 1$, any field $k$, all finite dimensional vector spaces $X,Y$ over $k$, and all subspaces $H \subseteq \Hom_k(X,Y)$ so that 
    \[
    \dim \left( X/(\cap_{h \in H} \ker (\pi_Z \circ h)) \right) \geq \left( n + P(\dim Y) \right )\dim Z,
    \]
    for all non-trivial quotients $\pi_Z\colon Y \to Z \neq 0$, there exists a surjection $X \to \oplus^n Y$ with all components in $H$.
    \end{corollary}

\begin{proof}
    Let $h_1, \ldots, h_\ell$ be a basis of $H$. For any surjection $\pi_Z \colon Y \to Z \neq 0$ of vector spaces, observe first that $\cap_{h\in H} \ker(\pi_Z \circ h) = \cap_{i=1}^\ell \ker(\pi_Z \circ h_i)$. Writing $\Phi_Z$ for the composition
    \begin{equation*}
        X \xrightarrow{(h_1, \ldots, h_\ell)} \bigoplus_{i=1}^\ell Y \xrightarrow{\oplus \pi_Z} \bigoplus_{i=1}^\ell Z,
    \end{equation*}
    our assumptions give that $\rank \Phi_Z \geq (n+P(\dim Y))\dim Z$ for all non-trivial quotients $Z$ of $Y$.

    We let $(\blank)^* = \Hom_k(\blank,k)$ and use duality for finite-dimensional vector spaces over $k$. Put $H^* = \{\phi^* \mid \phi \in H \} \subseteq \Hom_k(Y^*,X^*)$. For every non-trivial subspace $0 \neq U \subseteq Y^*$,  the rank of
    \begin{equation*}
        \bigoplus_{i=1}^\ell U \xrightarrow{} \bigoplus_{i=1}^\ell Y^* \xrightarrow{(h_1^*, \ldots, h_\ell^*)} X^*
    \end{equation*}
    equals the rank of $\Phi_{U^*}$ and so is at least $(n+P(\dim Y))\dim U$. It follows that $\dim H^*(U) \geq (n+P(\dim Y))\dim U$ for all $0 \neq U \subseteq Y^*$. Applying \autoref{thm:linalgcriterion}, we have that there exists an injection
    $
        \bigoplus^n Y^* \xrightarrow{(\phi_1^*, \ldots, \phi_n^*)} X^*
    $
    for some $\phi_1^*, \ldots, \phi_n^* \in H^*$. Dualizing yields a surjection $X \xrightarrow{(\phi_1, \ldots, \phi_n)} \bigoplus^n Y$ with $\phi_1, \ldots, \phi_n \in H$ as desired.
\end{proof}
\label{linear algebra}

\end{appendix}

\end{document}